%% file: 0MainAiC.tex
\documentclass{aic}

\aicAUTHORdetails{%
  title = {Canonical Decompositions of 3-Connected Graphs}, 
  author = {Johannes Carmesin and Jan Kurkofka},
  plaintextauthor = {Johannes Carmesin, Jan Kurkofka},
  keywords = {3-connected, decomposition, canonical, wheel, 4-connected, 3-separation, tri-separation},
}   

\aicEDITORdetails{%
   year={2025},
   number={7},
   received={11 April 2024},   
   published={12 September 2025},  
   doi={10.19086/aic.2025.7},      
}   

\input{1Preamble}

\begin{document}

\begin{frontmatter}[classification=text]

\title{Canonical Decompositions\\of 3-Connected Graphs} 

\author[jc]{Johannes Carmesin\thanks{Funded by EPSRC, grant number EP/T016221/1}}
\author[jk]{Jan Kurkofka\thanks{Funded by EPSRC, grant number EP/T016221/1}}

\begin{abstract}
We offer a new structural basis for the theory of 3-connected
graphs, providing a unique decomposition of every such graph into
parts that are either quasi 4-connected, wheels, or
thickened $K_{3,m}$'s. Our construction is explicit,
canonical, and has the following applications:  we obtain a new theorem characterising all finite Cayley graphs
as either essentially 4-connected, cycles, or complete graphs on at
most four vertices, and we provide an automatic proof of Tutte's wheel theorem.
\end{abstract}
\end{frontmatter}

\input{2Content}


\section*{Acknowledgments} 
We are grateful to Nathan Bowler and Reinhard Diestel for questions that indirectly led us to a streamlining of the definition of tri-separations.
We are deeply grateful to the two reviewers for their outstanding support with this overlong paper.
We thank the first reviewer for finding errors in the proofs of \autoref{bttExtra}, \autoref{is_bold} and \autoref{totally-nice}, and for optimising the proofs of \autoref{tetraXtrisep} and \autoref{3cutsAreNested}.
We thank the second reviewer for finding a gap in the proof of \autoref{nestedviacorners}.
We thank Romain Bourneuf for independently finding the error in the proof of \autoref{totally-nice}.
Raphael W.~Jacobs and Paul Knappe have studied mixed-separations independently from us, and we thank them for interesting discussions on mixed-separations.
We thank Dominik Blankenhagen and Nicolás Pich Preuss for spotting typos in the appendix.

\providecommand{\bysame}{\leavevmode\hbox to3em{\hrulefill}\thinspace}
\bibliographystyle{amsplain}


\begin{aicauthors}
\begin{authorinfo}[jc]
  Johannes Carmesin\\
  University of Birmingham\\
  Birmingham, UK\\
  \emph{Current info:}\\
  Professor\\
  TU Freiberg\\
  Freiberg, Germany\\
  johannes\imagedot{}carmesin\imageat{}math\imagedot{}tu-freiberg\imagedot{}de \\
  \url{https://j-carmesin.github.io}
\end{authorinfo}
\begin{authorinfo}[jk]
  Jan Kurkofka\\
  University of Birmingham\\
  Birmingham, UK\\
  \emph{Current info:}\\
  TU Freiberg\\
  Freiberg, Germany\\
  jan\imagedot{}kurkofka\imageat{}math\imagedot{}tu-freiberg\imagedot{}de \\
  \url{https://www.jan-kurkofka.eu/}
\end{authorinfo}
\end{aicauthors}

\end{document}

%% file: 2Content.tex
\clearpage
\tableofcontents
\newpage

\newcommand{\thechapter}{0}

\section*{Introduction}
\addcontentsline{toc}{section}{Introduction}

A tried and tested approach to a fair share of problems in structural and topological graph theory -- such as the two-paths problem~\cite{GMIX,GMXIII,GMXVI} or Kuratowski's theorem~\cite{ThomassenKura} -- is to first solve the problem for 4-connected\footnote{A graph $G$ is \emph{$k$-connected}, for a~$k\in\Nbb$, if $G$ has more than $k$ vertices and deleting fewer than $k$ vertices from~$G$ does not disconnect~$G$.} graphs.
Then, in an intermediate step, the solutions for the 4-connected graphs are extended to the 3-connected graphs, by drawing from a theory of 3-connected graphs that has been established to this end.
Finally, the solutions for the 3-connected graphs are extended to all graphs, in a systematic way by employing decompositions of general graphs along their cutvertices and 2-separators.

The intermediate step of this strategy seems curious: why should the step from 4-connected to 3-connected require an entirely different treatment than the systematic step from 3-connected to the general case?
Indeed, the intermediate step carries the implicit believe that it is not possible to decompose 3-connected graphs along 3-separators in a way that is on a par with the renowned decompositions along separators of size at most two.
Our main result offers a solution to this long-standing hindrance.
To explain this, we start by giving a brief overview of the renowned decompositions along low-order separators.

Graphs trivially decompose into their components, which either are $1$-connected or consist of isolated vertices.
The $1$-connected graphs are easily decomposed further, along their cutvertices, into subgraphs that either are $2$-connected or $K_2$'s which stem from bridges.

When decomposing 2-connected graphs further, however, things begin to get more complicated.
Indeed, a 2-separator -- a set of two vertices such that deleting the two vertices disconnects the graph --  may separate the vertices of another 2-separator.
Then if we choose one of them to decompose the graph by cutting at the 2-separator, we loose the other.
In particular, it is not possible to decompose a 2-connected graph simply by cutting at all its 2-separators.
An illustrative example for this are the 2-separators of a cycle.

There is an elegant way to resolve this problem.
If two 2-separators are compatible with each other, in the sense that they do not cut through each other, then we say that these 2-separators are \emph{nested} with each other.
Let us call a 2-separator \emph{totally-nested} if it is nested with every 2-separator of the graph.
The solution is that every 2-connected graph decomposes into 3-connected graphs, cycles and~$K_2$'s, by cutting precisely at its totally-nested 2-separators. 
Tutte~\cite{TutteGrTh} found this decomposition first, but with a different description. 
The description via total nestedness was discovered later by Cunningham and Edmonds \cite{cunningham_edmonds_1980}.

The~obvious guess how the solution might extend to 3-separators of 3-connected graphs is this: every 3-connected graph decomposes into 4-connected graphs, wheels\footnote{Every $k$-connected graph $G$ can be made $(k+1)$-connected by adding one more vertex and joining it to all original vertices. From this perspective, wheels are the 3-connected analogues of cycles.} and~$K_3$'s, by cutting precisely at its totally-nested 3-separators.
This guess turns out to be wrong, as the following three examples demonstrate.\medskip

\noindent\begin{minipage}{0.2\textwidth}
\includegraphics[width=\linewidth]{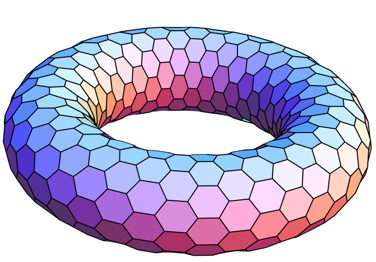}
\end{minipage}%
\hfill%
\begin{minipage}{0.75\textwidth}
Let $G$ be a toroidal hex-grid as depicted on the left~\cite{ToroidalHexGrid}, and note that $G$ is 3-connected.
The neighbourhoods of the vertices of $G$ are precisely the 3-separators of~$G$, so no 3-separator of~$G$ is totally-nested.
However, $G$ is neither $4$-connected nor a wheel.
But we will see that $G$ is `quasi 4-connected', as no 3-separator cuts off more from~$G$ than just one vertex.
\end{minipage}\medskip

\noindent\begin{minipage}{0.78\textwidth}
$3\times k$ grids as depicted on the right, slightly extended to make them 3-connected, have no totally-nested 3-separators; yet they are neither 4-connected nor wheels.
\end{minipage}
\hfill%
\begin{minipage}{0.2\textwidth}
\includegraphics[width=\linewidth]{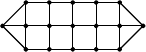}
\end{minipage}\medskip

\noindent\begin{minipage}{0.2\textwidth}
\includegraphics[width=\linewidth]{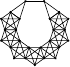}
\end{minipage}%
\hfill%
\begin{minipage}{0.75\textwidth}
Let $G$ be the graph on the left.
Every 3-separator of~$G$ consists of one of the two top vertices of degree three, and two vertices in the intersection of two neighbouring~$K_5$'s; or it is the neighbourhood of either degree-three vertex.
Hence $G$ has no totally-nested 3-separators.
This remains true if we replace the $K_5$'s in~$G$ with arbitrary 3-connected graphs.
Thus, $G$ represents a class of counterexamples that is as complex as the class of 3-connected graphs.
\end{minipage}
\medskip

\noindent We resolve these problems with a twofold approach:
\begin{enumerate}
    \item We relax the notion of 4-connectivity to that of \emph{quasi 4-connectivity}.
    We learned about this idea from Grohe's work~\cite{grohe2016quasi}.
    \item We introduce the new notion of a \emph{tri-separation}, which we use instead of 3-separators. The key difference is that tri-separations may use edges in addition to vertices to separate the graph.
\end{enumerate}

A \emph{mixed-separation} of a graph~$G$ is a pair~$(A,B)$ such that $A\cup B=V(G)$ and both $A\sm B$ and $B\sm A$ are nonempty.
We refer to $A$ and $B$ as the \emph{sides} of the mixed-separation.
The \emph{separator} of $(A,B)$ is the disjoint union of the vertex set $A\cap B$ and the edge set $E(A\sm B,B\sm A)$.
If the separator of $(A,B)$ has size three, we call $(A,B)$ a \emph{mixed 3-separation}.
Mixed-separations can be viewed as separations of the barycentric subdivision of~$G$.

\begin{tridfn}[Tri-separation]
A~\emph{tri-separation} of a graph $G$ is a mixed 3-separation $(A,B)$ of~$G$ such 
that every vertex in $A\cap B$ has at least two neighbours in both $G[A]$~and~$G[B]$.
\end{tridfn}

Two mixed-separations $(A,B)$ and $(C,D)$ of~$G$ are \emph{nested} if, after possibly switching the name $A$ with $B$ or the name $C$ with $D$, we have $A\se C$ and $B\supseteq D$.
A~tri-separation of~$G$ is \emph{totally-nested} if it is nested with every tri-separation of~$G$.
A~tri-separation $(A,B)$ of a 3-connected graph~$G$ is \emph{nontrivial} if both $G[A]$ and $G[B]$ contain a cycle; otherwise it is \emph{trivial}.
So $(A,B)$ is trivial if and only if $A$ and $B$ are the sides of a 3-edge-cut with a side of size one (\autoref{trivial}).
The tri-separations that we will use to decompose $G$ are the totally-nested nontrivial tri-separations of~$G$.\medskip

\noindent\begin{minipage}{0.2\textwidth}
\includegraphics[width=\linewidth]{ToroidalHexGrid.png}
\end{minipage}%
\hfill%
\begin{minipage}{0.75\textwidth}
Every vertex of the toroidal hex-grid~$G$ forms the singleton side of a trivial tri-separation.
Since there are no other tri-separations, all these tri-separations are totally-nested -- but they are trivial.
While $G$ is not 4-connected, it is \emph{quasi 4-connected}: $G$ is 3-connected, has more than four vertices, and every 3-separation of~$G$ (a mixed 3-separation whose separator consists of vertices only) has a side of size at most four.
\end{minipage}\bigskip

\noindent\begin{minipage}{0.78\textwidth}
The coloured 3-edge-cuts determine nontrivial tri-separations of the slightly extended $3\times k$ grid, and these are precisely the totally-nested nontrivial tri-separations.
\end{minipage}
\hfill%
\begin{minipage}{0.2\textwidth}
\includegraphics[width=\linewidth]{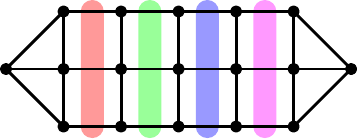}
\end{minipage}\medskip

\noindent\begin{minipage}{0.2\textwidth}
\includegraphics[width=\linewidth]{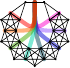}
\end{minipage}%
\hfill%
\begin{minipage}{0.75\textwidth}
Every nontrivial tri-separation of the graph on the left has a separator that consists of the top edge together with two vertices in a coloured set.
As these tri-separations are pairwise nested, they are precisely the totally-nested nontrivial tri-separations.
\end{minipage}\medskip

\noindent\begin{minipage}{\jaf{0.44}{0.44}{0.49}\textwidth}
Wheels have no totally-nested tri-separations.
\end{minipage}
\begin{minipage}{0.1\textwidth}
\includegraphics[height=3\baselineskip]{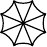}
\end{minipage}\medskip

Given a 3-connected graph $G$ and a set $N$ of pairwise nested tri-separations, we can say which parts we obtain by decomposing~$G$ along~$N$.
Roughly speaking, these are maximal subgraphs of~$G$ that lie on the same side of every tri-separation in~$N$, with some edges added to represent external connectivity in~$G$.
We call the resulting minors of~$G$ the \emph{torsos} of~$N$, as they generalise the well known torsos of tree-decompositions from the theory of graph minors.
We consider two types of torsos in this paper: compressed torsos and expanded torsos. What we call torso here will be called compressed torso later.
See~\autoref{sec:basicsForDecomp} for details; in particular \autoref{univ_3sepr}.

While the block-cutvertex decomposition and the Tutte-decomposition can be described in terms of tree-decompositions, this is no longer evident for the tri-separation decomposition as the adhesion-sets of tree-decompositions cannot contain edges by definition.
To overcome this, we introduce a notion of \emph{mixed-tree-decomposition} that allows both vertices and edges in the adhesion-sets.
Mixed-tree-decompositions sit in between tree-decompositions (that only allow vertices in adhesion-sets) and Wollan's tree-cut decompositions~\cite{WollanExcludedImmersion} (that only allow edges in adhesion-sets).
See \autoref{sec:TLD} for details.
Every mixed-tree-decomposition of a graph $G$ can be viewed as a tree-decomposition of the barycentric subdivision of~$G$.

According to the 2-separator theorem, some of the building blocks for 2-connected graphs are~$K_2$'s.
The analogue of these building blocks for 3-connected graphs turn out to be~\emph{thickened $K_{3,m}$'s} with $m\ge 0$: these are obtained from $K_{3,m}$ by adding edges to its left class of size three to turn it into a triangle.

\begin{main}\label{mainIntro}
Let $G$ be a 3-connected graph. 
Let $\Tcal(G)$ denote the mixed-tree-decomposition of~$G$ that is uniquely determined by the set $N(G)$ of all totally-nested nontrivial tri-separations of~$G$.
Each torso~$\tau$ of $\Tcal(G)$ is a minor of~$G$ and satisfies one of the following:
\begin{enumerate}
    \item $\tau$ is quasi 4-connected;
    \item $\tau$ is a wheel;
    \item $\tau$ is a thickened $K_{3,m}$ or $G=K_{3,m}$ with $m\geq 0$.
\end{enumerate}
\end{main}
We emphasise that the sets~$N(G)$ and mixed-tree-decompositions $\Tcal(G)$ are obviously canonical, meaning that they commute with graph-isomorphisms: $N(\varphi(G))=\varphi(N(G))$ and $\Tcal(\varphi(G))=\varphi(\Tcal(G))$ for all graph-isomorphisms~$\varphi\colon G\to G'$.
Our proof of \autoref{mainIntro} offers additional structural insights which can be used to refine \autoref{mainIntro}; see \autoref{univ_3sepr}.
All graphs in this paper are finite or infinite unless stated otherwise; in particular, \autoref{mainIntro} includes infinite 3-connected graphs.

\subsection*{Applications}
We provide the following applications of our work.
It is well known that all Cayley graphs of finite groups are either 3-connected, cycles, or complete graphs on at most two vertices~\cite{infiniteSPQR}.
By heavily exploiting the fact that our decomposition of 3-connected graphs is canonical, we can refine this fact:

\begin{mcor}\label{vx-trans}
Every vertex-transitive finite connected graph~$G$ either is essentially 4-connected, a cycle, or a complete graph on at most four vertices.
\end{mcor}
We give the precise definition of `essentially 4-connected' in \autoref{sec:Cor2}; the main difference to `quasi 4-connected' is that we allow 3-edge-cuts that have a side which is equal to a triangle.
A classical tool in geometric group theory from the textbook of Godsil and Royle \cite[Theorem 3.4.2]{AlgebraicGT} states that every finite vertex-transitive $d$-regular graph is $k$-connected for $k:=\lceil \frac{2}{3}(d+1)\rceil$.
\autoref{vx-trans} strengthens this tool in a special case.

Another application comes in the form of an automatic proof of Tutte's wheel theorem~\cite{TutteWheel}. 
In 
\jaf{
the upcoming work~\cite{conny_aug},
}{%
the upcoming work~\cite{conny_aug},
}{%
upcoming work,
}
\autoref{mainIntro} will be used to construct an FPT algorithm for connectivity augmentation from 0 to 4, and the property of total nestedness is crucial for that; see \autoref{sec:conc}.

\subsection*{When canonicity and an explicit description matter}

Recall that the tri-separation decomposition of \autoref{mainIntro} is canonical and is explicitly described so that it is uniquely determined for every 3-connected graph.
These two of its aspects are \emph{absolutely crucial} for a number of its applications:
\begin{enumerate}
    \item For vertex-transitive graphs, such as Cayley graphs, exploiting the combination of canonicity and total-nestedness makes up the entire proof of \autoref{vx-trans}.
    Just recently, this combination has also been exploited when using the Tutte-decomposition in the proof of a low-order Stallings-type theorem for finite nilpotent groups~\cite{Cayley2sep}.
    An obvious next step in this direction would be to exploit this combination with the tri-separation decomposition.
    \item For Connectivity Augmentation, canonicity and access to an explicit description are key~\cite{conny_aug}.
    \item Total-nestedness is incredibly desirable in Parallel Computing, the foundation of every supercomputer. Splitting the workload of finding the decomposition is a lot easier when all the partial solutions, which would come in the form of sets of already found totally-nested tri-separations, can always be combined without conflict.
    \item Recently, coverings as known from Topology have been employed to systematically construct \emph{graph-decompositions}, tree-decompositions where the tree may be any arbitrary graph, by applying classical theorems about tree-decompostions to the covering of a graph, and then folding the tree-decomposition to a graph-decomposition~\cite{GraphDec}.
    For the Tutte-decomposition, it is known how to achieve this directly without employing coverings~\cite{loc2sepr}.
    For the covering approach, canonicity is paramount: it makes the entire construction work.
    For the direct approach, the description of the Tutte-decomposition via total-nestedness is key.
    The construction from~\cite{GraphDec} can be generalised so that it can be applied to the tri-separation decomposition; we would be excited to see this happen.
    \item Finally, as the Tutte-decomposition is canonical and explicit, we believe that any decomposition result that claims to generalise Tutte should be both canonical and explicit.
\end{enumerate}

Grohe showed in pioneering work that every 3-connected graph has a tree-decomposition of adhesion~3 into torsos that are quasi 4-connected, $K_4$ or~$K_3$~\cite{grohe2016quasi}.
Grohe's decomposition is exciting and has indeed quite a few applications, however, they do not include (1)--(5).
Since our decomposition in particular satisfies~(5), we regard it as an analogue of Tutte's decomposition.

\subsection*{More related work}

It would be most exciting to try to extend our work to separators of larger size, see  \autoref{sec:conc} for an open question in this direction.
Just recently, our work has been extended to separators of size~4 in~\cite{Tutte4con}.
Our work complements existing work on decomposing graphs along separators of arbitrary size and the corresponding theory of tangles~\cite{albrechtsen2023refining,CarmesinToTshort,CDHH13CanonicalAlg,CDHH13CanonicalParts,confing,CG14:isolatingblocks,Entanglements,ASS,TreeSets,StructuralASS,ProfilesNew,FiniteSplinters,InfiniteSplinters,RefiningToT,GroheTangles3,ComputingTangles,jacobs2023efficiently,GMX}.
A~fair share of the work on 3-connected graphs studies which substructures are `\mbox{(in-)}essential' to 3-connectivity (in the sense that their contraction or deletion preserves 3-connectivity); this includes the work of Ando, Enomoto and Saito~\cite{ando1987contractible} and of Kriesell~\cite{KriesellNonEdges,KriesellContractibleSubgraph,KriesellAlmostAll,KriesellTriangleFree,kriesell2008number,KriesellVertexSuppression}.
The structure of 3-separations in matroids is a well-studied topic; a fundamental result in this area is the decomposition result of Oxley, Semple and Whittle~\cite{Structure3sepsMatroids}. It would be most natural to extend our ideas to matroids and this problem is discussed in \autoref{sec:conc}.
Also see~\cite{4connectedMatroids}.
Hopes for a generalisation to directed graphs are fuelled by recent work of Bowler, Gut, Hatzel, Kawarabayashi, Muzi and Reich~\cite{Flo}.
Our work is related to recent work of Esperet, Giocanti and Legrand-Duchesne~\cite{Ugo}, who employ Grohe's techniques for 3-connected graphs to prove a general decomposition result of infinite graphs without an infinite clique-minor, and our result might provide an alternative perspective. 
There is also a connection to twin-width, see for example the recent work of Heinrich and Raßmann~\cite{TwinWidth}.

\subsection*{Organisation of the paper}
An overview of the proof of \autoref{mainIntro} is given in \autoref{sec:OV}. 
The remainder of the paper consists of four chapters.
Each chapter will feature its own comprehensive overview.
In the first chapter, we introduce and prove the \nameref{Angry} (\ref{Angry}); this classifies the 3-connected graphs in which all nontrivial tri-separations are crossed (hence the name), and it will be the key ingredient of the proof of \autoref{mainIntro} as it deals with the special case $N(G)=\emptyset$.
\autoref{vx-trans} can already be derived from the \nameref{Angry}, so the first chapter also includes the proof of \autoref{vx-trans}.
In the second chapter, we prove \autoref{mainIntro}.
The third chapter provides an outlook.
Finally, the fourth chapter offers a proof of the 2-separator theorem, but phrased in the language of this paper and with a structural strengthening added, for the sake of convenience and completeness.

\section{Overview of the proof}\label{sec:OV}

Let $G$ be a 3-connected graph, and let $N$ denote the set of totally-nested nontrivial tri-separations of~$G$.
By adapting standard arguments, we will find a unique mixed-tree-decomposition $\Tcal=(T,(V_t)_{t\in V(T)})$ of~$G$ such that the edges of $T$ induce precisely the tri-separations in~$N$.
Let $\tau$ be an arbitrary torso of~$N$, which can also be viewed as a torso of bag $V_t$ of the mixed-tree-decomposition~$\Tcal$.
It is routine to verify that $\tau$ is 3-connected or a~$K_3$, and that $\tau$ is a minor of~$G$.
So it remains to show that $\tau$ either is quasi 4-connected, a wheel, a thickened $K_{3,m}$ or $G=K_{3,m}$ with $m\ge 0$.

Our approach is to link these three outcomes to the structure of the tri-separations of~$G$ that `interlace' the torso~$\tau$, as follows.
Let $(A,B)$ be a nontrivial tri-separation of~$G$.
Roughly speaking, we say that $(A,B)$ \emph{interlaces} $\tau$ if $\tau$ has vertices in $A\sm B$ and in $B\sm A$, so $(A,B)$ `cuts through'~$\tau$.
If additionally $G[A\sm B]$ or $G[B\sm A]$ is connected, then we say that $(A,B)$ interlaces $\tau$ \emph{heavily}.
Else if both $G[A\sm B]$ and $G[B\sm A]$ have at least two components, then we say that $(A,B)$ interlaces $\tau$ \emph{lightly}.
This allows for the following structural strengthening of \autoref{mainIntro}:
\begin{enumerate}
    \item if $\tau$ is not interlaced, then $\tau$ is quasi 4-connected or a $K_4$ or~$K_3$;
    \item if $\tau$ is heavily interlaced, then $\tau$ is a wheel;
    \item if $\tau$ is lightly interlaced, then $\tau$ is a thickened~$K_{3,m}$ or $G=K_{3,m}$.
\end{enumerate}

\vspace{.3cm}

For the proof of (1), suppose that $\tau$ is not interlaced.
Let us assume for a contradiction that $\tau$ is neither quasi 4-connected, nor a $K_4$ nor~$K_3$.
Then $\tau$ has a 3-separation $(A,B)$ with two sides of size at least five.
In a 3-page technical argument, we `lift' $(A,B)$ from $\tau$ to a tri-separation of~$G$ that interlaces~$\tau$ -- a contradiction.

The proof of (2) is a bit more tricky and will be explained below. 

For the proof of~(3), suppose that $\tau$ is lightly interlaced by a tri-separation~$(A,B)$.
Then we first note that the separator $S$ of $(A,B)$ consists of three vertices, and that $G\sm S$ has at least four components.
The four components ensure that $S$ is `4-connected' in~$G$, as every two vertices in $S$ can be linked by four internally vertex-disjoint paths in $G$ through the four components.
The 4-connectivity of $S$ can then be used to show that every component $C$ of $G\sm S$ determines a totally-nested tri-separation of~$G$, with $C$ on one side, whose separator consists of vertices of $S$ or edges that join $C$ to~$S$.
The nontrivial ones amongst the totally-nested tri-separations determined by the components of $G\sm S$ are precisely the ones that bound the torso~$\tau$ (i.e., the ones that are induced by the edges of $T$ at~$t$).
If $G[S]$ is edgeless and all components of $G\sm S$ have size one, then $G=K_{3,m}$ where $m$ is equal to the number of components of~$G\sm S$.
Otherwise, a brief analysis shows that $\tau$ is a thickened~$K_{3,m}$, where $m$ is equal to the number of components $C$ of $G\sm S$ such that $|C|=1$ and $G[S]$ is edgeless.

\begin{figure}[ht]
    \centering
    \includegraphics[height=10\baselineskip]{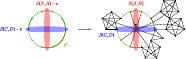}
    \caption{The graph $G$ (on the right) can be constructed from a cycle $O$ (on the left) by replacing its bold edges by 2-connected graphs and adding the vertex $v$ together with its incident edges.     
    Here $S(A,B)$ and $S(C,D)$ denote the separators of $(A,B)$ and $(C,D)$, respectively. They intersect in the vertex~$v$. The cycle $O$ is highlighted in green.}
    \label{fig:Overview2sep}
\end{figure}

\subsection*{Proof of (2)} 
Assume we are given a tri-separation $(A,B)$ of~$G$ that interlaces $\tau$ heavily. 
Since $(A,B)$ is not in~$N$, it is \emph{crossed} by a nontrivial tri-separation~$(C,D)$ of~$G$, meaning that $(A,B)$ and $(C,D)$ are not nested with each other. 
The first step is to extend the standard theory of crossing separations to our context of tri-separations.
From this we learn that the separators of $(A,B)$ and $(C,D)$ intersect in exactly one vertex; call it~$v$. 
The next step is to apply the 2-separator theorem to the graph $G-v$. 
This tells us that the graph $G-v$ can be obtained from a cycle $O$ by replacing some of its edges by 2-connected subgraphs of~$G$. 
We refer to these replaced edges of $O$ as \emph{bold}. 
The separators of $(A,B)$ and $(C,D)$ without $v$ alternate on the cycle~$O$, see \autoref{fig:Overview2sep}. 
Intuitively, one might hope that the vertex $v$ together with the two endvertices of a bold edge of~$O$ forms the separator of a totally-nested nontrivial tri-separation.
This is almost true, but some of the vertices in the separator might fail to have two neighbours on some side.
We resolve this issue by replacing these vertices with one of their incident edges.
The resulting tri-separation is referred to as the \emph{pseudo-reduction} at the bold edge.

A key challenge is to show that the pseudo-reductions at the bold edges of~$O$ are totally-nested.
We approach this challenge in a systematic way by showing that a connectivity property of the separator of a tri-separation implies total nestedness (we very much believe that this can be turned into a characterisation of total nestedness, but given the length of the paper, we do not attempt this here). 
For this approach to succeed, we need to verify that the separators of the pseudo-reductions satisfy this connectivity property. 
The connectivity property itself is a little technical, as we require different amounts of connectivity depending on whether edges or vertices are in the separator. 
For simplicity, let us assume that the separator consists of three vertices. 
Then our task is to find three internally vertex-disjoint paths between every pair of non-adjacent vertices in the separator avoiding the third vertex of the separator.
Between the two endvertices of a bold edge $e$ of~$O$ we find two paths in the 2-connected subgraph that is associated with the bold edge~$e$, and a third path follows the course of the path~$O-e$.
So it remains to construct three internally vertex-disjoint paths between an endvertex of the bold edge $e$ and~$v$. 
If $O$ is short, we need to consider a few cases, and if $O$ is long (length five suffices), then we study how the bold edges are distributed on~$O$. 
We identify five possible patterns that cover all cases and verify that three internally vertex-disjoint paths exist for each of the five patterns, see \autoref{fig:Overview2sep2}. 
Hence the pseudo-reductions at bold edges are totally-nested. 

\begin{figure}[ht]
    \centering
    \includegraphics[height=12\baselineskip]{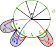}
    \caption{The graph $G$ together with the bold edges of~$O$. 
    For each bold edge of $O$ we indicate in grey its corresponding 2-connected replacement graph as given in the construction of $G$ from~$O$. 
    In this figure we refer to them as \emph{bags}. 
    The vertex $a$ is an endvertex of the bold edge $e$. Here the cycle $O$ has the \emph{pattern} \ty{btx} with regard to $e$ and~$a$: the first letter \ty{b} indicates that the edge $f_1$ on $O$ incident with $a$ aside from $e$ is bold, the second letter \ty{t} indicates that the edge $f_2$ after that on $O$ is not bold (\ty{t}imid), and the letter \ty{x} indicates that the edge $f_3$ after that can be arbitrary, bold or not. In our example, it is bold.
    For this pattern, there are three internally vertex-disjoint paths from $a$ to~$v$: the first path $P_1$ connects to $v$ from the bag at~$e$, the second path $P_2$ connects to $v$ from the bag at~$f_1$, and the final path $P_3$ uses a path disjoint from $P_2-a$ through the bag at~$f_1$, then traverses~$f_2$ and eventually connects to $v$ from the bag at~$f_3$.}
    \label{fig:Overview2sep2}
\end{figure}

Having shown that the pseudo-reductions at the bold edges of~$O$ are totally-nested, we know that they bound a torso~$\tau'$; this is not necessarily a torso of~$N$, but of the set of pseudo-reductions.
That is to say that $\tau'$ is the torso of a mixed-tree-decomposition which is coarser than~$\Tcal$, so $\tau'$ includes $\tau$, and we will later show that $\tau'=\tau$.
Using our knowledge of the structure of $G-v$ provided by~$O$, and 3-connectedness of~$G$, it is straightforward to show that $\tau'$ is a wheel.
So all that remains to show is that $\tau'$ is equal to~$\tau$.
For this, we have to show that the pseudo-reductions at the bold edges of~$O$ are precisely the tri-separations induced by the edges of $T$ at~$t$.
Equivalently, we have to show that no totally-nested nontrivial tri-separation of~$G$ interlaces~$\tau'$.
So we assume for a contradiction that some totally-nested nontrivial tri-separation $(U,W)$ of~$G$ interlaces~$\tau'$.
Roughly speaking, $(U,W)$ induces a mixed 2-separation of~$O$.
This induced mixed 2-separation cuts $O$ into two intervals.
We carefully select a vertex or non-bold edge from each interval, and then add $v$ to obtain the separator of a mixed 3-separation $(E,F)$ of~$G$.
Then $(E,F)$ crosses $(U,W)$ by standard arguments, and with a bit of extra work we turn $(E,F)$ into a tri-separation that still crosses $(U,W)$, yielding a contradiction.
Hence $\tau'$ is equal to~$\tau$.
As we have already shown that $\tau'$ is a wheel, the overview of the proof of~(2) is complete.

\clearpage
\renewcommand{\thechapter}{1}
\phantomsection
\noindent{\huge{\textbf{Chapter \thechapter \\ \\An angry theorem for tri-separations}}}
\addcontentsline{toc}{section}{\thechapter\ An angry theorem for tri-separations}

\setcounter{section}{0}

\section{Overview of this chapter}

This chapter provides the key ingredient for the proof of \autoref{mainIntro}, which we call the \nameref{Angry}.
Essentially, it states that \autoref{mainIntro} holds in the special case where~$N(G)$, the set of all totally-nested non-trivial tri-separations of~$G$, is empty;
that is: if every nontrivial tri-separation of~$G$ is crossed (hence the name `angry'), then $G$ is either quasi 4-connected, a wheel or a $K_{3,m}$ for some~$m\ge 3$.
The \nameref{Angry} extends to 3-connectivity the \nameref{Angry2Sep} (\ref{Angry2Sep}) of Tutte, which characterises the 2-connected graphs in which every every 2-separation is crossed by some other 2-separation: these are precisely the 3-connected graphs and the cycles.
We recommend having an early glance at \autoref{sec:2SeparatorTheorem} to see how the \nameref{Angry2Sep} fits into the proof of Tutte's decomposition for 2-connectivity.

\subsection[Statement of the Angry Tri-Separation Theorem]{Statement of the \nameref{Angry}}

Here we give all the necessary definitions to then state the \nameref{Angry}~(\ref{Angry}).
An (oriented) \emph{mixed-separation} of a graph~$G$ is an ordered pair $(A,B)$ such that $A\cup 
B=V(G)$ and both $A\sm B$ and $B\sm A$ are non-empty.
We call $A$~and~$B$ the \emph{sides} of~$(A,B)$.
The~\emph{separator} of $(A,B)$ is the disjoint union of the vertex set $A\cap B$ and the edge set 
$E(A\setminus B,B\setminus A)$.
We~denote the separator of $(A,B)$ by~$S(A,B)$.
The~\emph{order} of $(A,B)$ is the size $\vert S(A,B)\vert$ of its separator.
A~mixed-separation of order~$k$ for $k\in\Nbb$ is called a \emph{mixed $k$-separation} for short.
The~separator of a mixed $k$-separation is a \emph{mixed $k$-separator}.
A~mixed-separation $(A,B)$ of~$G$ with no edges in its separator is called a \emph{separation} of~$G$.
Separations of order $k$ are called \emph{$k$-separations} and their separators are called 
\emph{$k$-separators}.

A mixed 3-separation $(A,B)$~of~$G$ is \emph{nontrivial} if both $G[A]$ and $G[B]$ include a cycle.

\begin{dfn}[Tri-separation]
A~\emph{tri-separation} of a graph $G$ is a mixed-separation $(A,B)$ of $G$ of order three such 
that every vertex in $A\cap B$ has at least two neighbours in both $G[A]$~and~$G[B]$.
The separator of a tri-separation of~$G$ is a \emph{tri-separator} of~$G$.
A~tri-separation is \emph{strong} if every vertex in its separator has degree at least four.
\end{dfn}

\begin{figure}[ht]
\centering
\includegraphics[height=2.5\baselineskip]{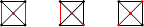}
\caption{Separators of nontrivial tri-separations of the 4-wheel.}
\label{fig:4wheel}
\end{figure}

\begin{eg}\label{WheelExample}
Let $G$ be a wheel with rim $O$ of length at least four.
Let $v$ denote the centre of~$G$.
Adding $v$ to any mixed 2-separator of~$O$ yields the separator of a nontrivial tri-separation of~$G$.
In fact, every nontrivial tri-separation of~$G$ can be obtained in this way.

If $G$ is a 3-wheel, aka~$K_4$, then every separator of a nontrivial tri-separation of~$G$ consists of two vertices and one edge.
\end{eg}

Similarly as is common for separations~\cite{DiestelBookCurrent}, we define a partial ordering on the mixed-separations of any graph by letting $(A,B)\le (C,D)$ if and only if $A\se C$ and $B\supseteq D$.
Two mixed-separations $(A,B)$ and~$(C,D)$ are \emph{nested}~if, after possibly switching the name $A$ with~$B$ or the name $C$ with~$D$, we have $A\se C$ and~$B\supseteq D$.
If two mixed-separations are not nested, they \emph{cross}.
A set of mixed-separations is \emph{nested} if its elements are pairwise nested.

\begin{dfn}[Totally nested]
A tri-separation of~$G$ is \emph{totally nested} if it is nested with every tri-separation of~$G$.
\end{dfn}

\begin{eg}
Every nontrivial tri-separation in the 4-wheel, as found in~\autoref{WheelExample}, is crossed by another nontrivial tri-separation; see~\autoref{fig:4wheel}.
By contrast, every 3-cut with both sides of size at least two in a 3-connected graph determines a totally-nested nontrivial tri-separation (we will see this in \autoref{3cutsAreNested}).
\end{eg}

A graph~$G$ is \emph{internally 4-connected} if it is 3-connected, every 3-separation of~$G$ has a side that induces a claw, and $G\notin\{ K_4,K_{3,3}\}$.
Internally 4-connected graphs are quasi 4-connected.

\begin{thm}[Angry Tri-Separation Theorem]\label{Angry}
For every 3-connected graph~$G$, exactly one of the following is true:
\begin{enumerate}
    \item $G$ has a totally-nested nontrivial tri-separation;
    \item $G$ is a wheel or a $K_{3,m}$ for some~$m\ge 3$;
    \item $G$ is internally 4-connected.
\end{enumerate}
\end{thm}

\subsection{Organisation of this chapter}

In \autoref{sec:TLD}, we explore nested sets of tri-separations from a mixed-tree-decomposition perspective.
In \autoref{sec:triseps}, we collect useful properties of tri-separations.
In \autoref{sec:cornerDias}, we introduce tools that allow us to systematically study how tri-separations can cross.
In \autoref{sec:Cor2}, we deduce \autoref{vx-trans} from the \nameref{Angry}.
In \autoref{sec:externalTricon}, we employ the tools from the previous section to find necessary conditions for when a tri-separation is totally nested.
In \autoref{sec:BackgroundOnTutte}, we recall the 2-separation-version of \autoref{mainIntro}, as we will need it in the proof of the \nameref{Angry}.
In \autoref{sec:ProofMain}, we will see why the 2-separation-version of \autoref{mainIntro} is helpful for finding totally-nested nontrivial tri-separations.
In \autoref{sec:specialAngry} and \autoref{sec:GeneralAngry}, we put together the tools developed in the previous sections and we complete the proof of the \nameref{Angry}, where the former deals with special cases and the latter solves the general case.

\section{Mixed-tree-decompositions}\label{sec:TLD}

It is well-known that there is a natural correspondence between nested sets of separations and tree-decompositions from the Theory of Graph Minors~\cite{GMX}.
This correspondence does not extend to \emph{mixed}-separations.
Hence it is not clear to us how the sets $N(G)$ in \autoref{mainIntro} could determine tree-decompositions.
However, we believe that this is not surprising, and there is a natural solution.
To explain this, we need a bit more background first.

Wollan introduced tree-cut decompositions to study the immersion-relation, an alternative to the graph-minor relation~\cite{WollanExcludedImmersion}.
His tree-cut decompositions naturally correspond to nested sets of edge-cuts.
Since tree-decompositions correspond to separations (with vertex-separators), and tree-cut decompositions correspond to edge-cuts, the two notions of decomposition are not more general than each other, yet they are closely related.
As mixed-separations generalise both separations and edge-cuts, they should correspond to a notion of tree-like decomposition that generalises both tree-decompositions and tree-cut decompositions.
Indeed, such a notion exists.

Let $G$ be a graph.
Let us call a pair $(T,(V_t)_{t\in T})$ of a tree $T$ and a family of vertex sets $V_t\se V(G)$ indexed by the nodes $t\in T$ a \emph{mixed-tree-decomposition} of~$G$ if it satisfies the following two conditions:
\begin{enumerate}[label=\textnormal{(M\arabic*)}]
    \item $V(G)=\bigcup_{t\in T}V_t$;
    \item\label{M2} the subgraph of $T$ induced by $\{\,t\in T:v\in V_t\,\}$ is connected for every vertex $v\in G$.
\end{enumerate}
We refer to the vertex sets $V_t$ as \emph{bags}.

The difference to tree-decompositions is that edges are not required to have both ends in some bag~$V_t$.
The differences to tree-cut decompositions are that, on the one hand, we allow bags associated to distinct nodes to intersect, and on the other hand, we additionally require~\ref{M2}.
It is straightforward to check that all tree-decompositions and all tree-cut decompositions are mixed-tree-decompositions.

Let $\Tcal:=(T,(V_t)_{t\in T})$ be a mixed-tree-decomposition of a graph~$G$.
Write $\vE(T):=\{\,(x,y):xy\in E(T)\,\}$ for the set of all possible directions of edges in~$T$.
We can define a map $\alpha_{\Tcal}$ with domain $\vE(T)$ that assigns to each $(t_1,t_2)\in\vE(T)$ the pair $(U_1,U_2)$, where $U_i$ is the union of all bags $V_t$ with $t$ contained in the component of $T-t_1 t_2$ that includes~$t_i$ (for $i=1,2$).
A set $M$ of mixed-separations is \emph{symmetric} if for every $(A,B)\in M$ it also contains $(B,A)$.
From the abstract theory of~\cite{TreeSets}, it follows that every nested symmetric set $M$ of mixed-separations of~$G$ uniquely determines (up to isomorphism) a mixed-tree-decomposition $\Tcal$ of~$G$ such that $\alpha_{\Tcal}$ is bijective with image equal to~$M$.
In particular, the sets $N(G)$ in \autoref{mainIntro} can be expressed through mixed-tree-decompositions~$\Tcal(G)$.
The torsos of $\Tcal(G)$ are defined indirectly as the compressed torsos of~$N(G)$, which will be introduced in \autoref{subsec:Torsos}. For an alternative but equivalent definition, see~\cite[§2]{Tutte4con}.

\section{Properties of tri-separations}\label{sec:triseps}

A cut of a graph~$G$ is \emph{atomic} if it is of the form $E(v,V(G)\sm\{v\})$ for some vertex~$v\in V(G)$.

\begin{lem}\label{trivial}
The following are equivalent for every tri-separation $(A,B)$ of a 3-connected graph~$G$:
 \begin{enumerate}
  \item $(A,B)$ is trivial;
  \item $A$ and $B$ are the two sides of an atomic cut.
 \end{enumerate}
\end{lem}

\begin{proof}
The implication (2)$\to$(1) is clear.
For (1)$\to$(2) suppose that $G[A]$ contains no cycle, say.
We first show that the side $B$ cannot have exactly two vertices.
Indeed, if $B$ has size two, then $G[B]$ has maximum degree at most one, and since $(A,B)$ is a tri-separation it follows that $A\cap B$ must be empty.
Thus, $A$ and $B$ are the two sides of a cut of size three.
Hence some vertex in $B$ has degree at most two in~$G$, which contradicts 3-connectivity. 
So $B$ does not have size two.
As we are done otherwise, from now on we assume that the side~$B$ contains at least three vertices. 

If two of the edges in $E(A\sm B, B\sm A)$ had the same 
endvertex in~$B$, this vertex would be in a 2-separator of~$G$. 
So as the graph~$G$ is 3-connected, no two edges in $E(A\sm B, B\sm A)$ can have the same 
endvertex in~$B$. 
Let~$T$ be the graph obtained from $G[A]$ by adding all the edges from the separator of $(A,B)$.
By the above, $T$ is a tree. 
The tree $T$ has three leaves in~$B$.
Since $G$ is 3-connected, $T$~has no other leaves and also no vertices of degree two.
Hence~$T$ is a $K_{1,3}$ and (2) follows.
\end{proof}

\begin{cor}\label{trivial_is_nested}
The trivial tri-separations of a 3-connected graph~$G$ are nested with all strong tri-separations of~$G$.
\end{cor}

\begin{proof}
Let any trivial tri-separation be given.
By \autoref{trivial}, it is of the form $(\{v\},V\sm\{v\})$, say. 
Then the vertex $v$ has degree three in $G$. Let 
$(C,D)$ be any strong tri-separation. 
Since $(C,D)$ is strong, the vertex $v$ is not in 
$C\cap D$. So it is in precisely one of $C\sm D$ and $D\sm C$, say in $C\sm D$.
Then $D\se V\sm\{v\}$, which gives $(\{v\},V\sm\{v\})\le (C,D)$.  
\end{proof}

\begin{lem}\label{independentEdges}
Let $G$ be a 3-connected graph, and let $(A,B)$ be a nontrivial mixed 3-separation of~$G$.
Then the edges in $S(A,B)$ form a matching between $A\sm B$ and $B\sm A$.
\end{lem}
\begin{proof}
Let us show that no two edges in $S(A,B)$ share an end.
For this, let us suppose for a contradiction that two edges $e,f\in S(A,B)$ share an endvertex $v\in A\sm B$, say.
Let $x$ be the remaining element of $S(A,B)$ besides $e,f$ if it is a vertex, and otherwise let $x$ denote the endvertex in~$A\sm B$ of the edge in $S(A,B)$ besides~$e,f$.
Let $O$ be a cycle in $G[A]$.
Since $O$ has at least three vertices, one of them is distinct from $v$ and $x$, and so is not in $B\cup\{v,x\}$.
Hence the pair $(A,B\cup\{v,x\})$ is a 2-separation of~$G$ with separator equal to~$\{v,x\}$.
This contradicts the fact that $G$ is 3-connected.
\end{proof}

\begin{lem}\label{forUniqueReduction}
    Let $(X_1,X_2)$ be a mixed 3-separation of a 3-connected graph~$G$.
    Then for every vertex $v$ in the separator of~$(X_1,X_2)$, exactly one of the following holds:
    \begin{enumerate}
        \item The vertex $v$ has two neighbours in both~$X_1$ and~$X_2$.
        \item There exists a unique index $i\in\{1,2\}$ such that $v$ has precisely one neighbour in~$X_i$ but two neighbours in~$X_{3-i}$.
        The neighbour of $v$ in $X_i$ lies in $X_i\sm X_{3-i}$, so $v$ has two neighbours in $X_{3-i}$ that lie in $X_{3-i}\sm X_i$.
    \end{enumerate}
\end{lem}
\begin{proof}
    Since $G$ is 3-connected, $v$ has neighbours in $X_1\sm X_2$ and in $X_2\sm X_1$, and $v$ has degree at least three.
    So if (1) fails, we have~(2).
\end{proof}

\begin{dfn}[Reduction]
Let $(X_1,X_2)$ be a mixed 3-separation of a 3-connected graph.
We obtain a tri-separation from $(X_1,X_2)$ by deleting vertices from $X_1$ or~$X_2$, as follows.
For every vertex $v\in X_1\cap X_2$ that has fewer than two neighbours in some side~$X_i$, the index $i=:i(v)$ is unique, $v$~has a unique neighbour $x(v)$ in~$X_i$ that lies in~$X_i\sm X_{3-i}$, and $v$ has two neighbours in $X_{3-i}\sm X_i$ by~\autoref{forUniqueReduction}.
For both $j\in\{1,2\}$ we obtain $X_j'$ from $X_j$ by deleting all vertices $v$ with $i(v)=j$.
Then $(X_1',X_2')$ is a tri-separation of~$G$.
Every vertex $v\in X_1\cap X_2$ that was removed from some side is not in the separator of $(X_1',X_2')$, but instead the edge $\{v,x(v)\}$ is in $S(X_1',X_2')$.
In this context, we say that $v$ was \emph{reduced} to the edge $\{v,x(v)\}$.
We call $(X_1',X_2')$ the \emph{reduction} of~$(X_1,X_2)$.
Note that the reduction $(X_1',X_2')$ is nontrivial if $(X_1,X_2)$ is nontrivial.
\end{dfn}

Let $(A,B)$ be a mixed 3-separation of a graph~$G$.
A~\emph{strengthening} of $(A,B)$ is a mixed 3-separation $(A',B')$ that is obtained from $(A,B)$ by deleting at once all vertices of $A\cap B$ from one of the sides that have degree three in $G$ and have a neighbour in $A\cap B$, then taking a reduction.
Note that every vertex deleted from $A\cap B$ is replaced by a unique edge in the separator, so the order of the separator stays the same.

\begin{obs}\label{strengthening-good}
    If $(A',B')$ is a strengthening of $(A,B)$, then $A\sm B\se A'\se A$ and $B\sm A\se B'\se B$. 
    All strengthenings are strong tri-separations.\qed
\end{obs}

\begin{lem}\label{reduceToStrong}
Let $(A,B)$ be a mixed 3-separation of a 3-connected graph~$G$.
For every edge $uv$ in~$G$ with both ends $u,v$ in the separator of $(A,B)$,
there is a strengthening $(A',B')$ of $(A,B)$ so that $u,v\in B'$.
\end{lem}
\begin{proof}
Given $(A,B)$, we obtain $A''$ from~$A$ by deleting all vertices that lie in $A\cap B$ and have degree three in~$G$ and have some neighbour in $A\cap B$, and we put $B'':=B$.
Then we let $(A',B')$ be the reduction of~$(A'',B'')$.
Suppose now that $uv$ is an edge with ends $u,v\in A\cap B$.
Then $u$ and $v$ have two neighbours in~$B=B''$, so $u,v\in B'$.
\end{proof}

\begin{prop}\label{tetraXtrisep}
For every 3-connected graph~$G$, the following assertions are equivalent:
\begin{enumerate}
    \item\label{ItemTetra} $G$ is internally 4-connected or~$G\in\{K_4,K_{3,3}\}$;
        \item\label{ItemTetra_neu} every 3-separation of $G$ is trivial or~$G=K_4$;
        \item\label{ItemTrivial1} all tri-separations of~$G$ are trivial or $G=K_4$;
    \item\label{ItemTrivial2} all strong tri-separations of~$G$ are trivial.
\end{enumerate}
\end{prop}

\begin{proof}
\ref{ItemTetra}$\rightarrow$\ref{ItemTetra_neu}.
If $G$ is internally 4-connected or $G=K_{3,3}$, then every 3-separation $(A,B)$ of $G$ has a side ($A$, say) that induces a claw; in particular, $G[A]$ contains no cycle, so $(A,B)$ is trivial.

$\neg$\ref{ItemTrivial1}$\to\neg$\ref{ItemTetra_neu}.
Let $(A,B)$ be a nontrivial tri-separation of $G$.
Recall that $A\sm B$ and $B\sm A$ are nonempty by the definition of mixed-separation.
For each edge in $S(A,B)$ we pick one of its endvertices and add it to both sides. 
We pick these endvertices so that we preserve that $A\sm B$ and $B\sm A$ are nonempty. 
This is possible as $G$ has at least five vertices. As this preserves nontriviality, we end up with a nontrivial 3-separation of $G$. 

Clearly \ref{ItemTrivial1}$\to $\ref{ItemTrivial2}.

$\neg$\ref{ItemTetra}$\to\neg$\ref{ItemTrivial2}.
As $G$ is not internally 4-connected and $G\notin\{K_4,K_{3,3}\}$, we may let $(A,B)$ be a 3-separation of~$G$ none of whose sides induces a claw.
Let~$X:=A\cap B$.
If $G$ had at most four vertices, then~$G$ would be a~$K_4$, contradicting our assumption, so we have~$|G|\ge 5$.

{\bf Case 1:} the induced subgraph $G[X]$ has no edges.
Then $|A\sm B|\geq 2$ and~$|B\sm A|\geq 2$.
By \autoref{reduceToStrong}, there is a strong tri-separation $(A',B')$ of~$(A,B)$ with $A\sm B\se A'$ and $B\sm A\se B'$.
Since $A\sm B$ and $B\sm A$ have size at least two, so have $A'$ and~$B'$.
Thus~$(A',B')$ is nontrivial by \autoref{trivial}.

{\bf Case 2:} the induced subgraph $G[X]$ contains an edge~$x_1 x_2$.
Recall that $|G|\ge 5$.
The only 3-connected graphs on exactly five vertices that have a 3-separator are the 4-wheel and
$K_5^-$ ($K_5$~minus one edge). The unique $3$-separator of
$K_5^-$ is a triangle and thus it is the separator of a nontrivial tri-separation.
The 4-wheel has two nontrivial strong tri-separations.
So it remains to consider the
case that $G$ has at least six vertices. 
By symmetry, we may assume that $|A\sm B|\geq 2$.
Since $x_1 x_2$ is an edge of~$G$ with both ends in $X=A\cap B$, we find a strong tri-separation $(A',B')$ of~$G$ with $A\sm B\se A'$ and $B\sm A\se B'$ such that $x_1,x_2\in B'$.
Hence, the side~$B'$ misses at most one vertex of~$B$.
As $|B|\geq 4$, this gives~$|B'|\geq 3$. 
We also have~$|A'|\geq |A\sm B|$, and $|A\sm B|\ge 2$ by assumption. Hence the strong tri-separation $(A',B')$ is nontrivial by \autoref{trivial}.
\end{proof}

\section{Nested or crossed: analysing corner diagrams}\label{sec:cornerDias}

 What can we say about two mixed-separations if they cross?
In this section we address this question by introducing corner diagrams for mixed-separations.

Let $(A,B)$ and $(C,D)$ be mixed-separations of a graph~$G$.
The following definitions all depend on the context that $(A,B)$ and $(C,D)$ are given. They are supported by \autoref{fig:CornerDiagram}, which is commonly referred to as a `corner diagram'.

\begin{figure}[ht]
\centering
\includegraphics[height=8\baselineskip]{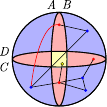}
\caption{$(A,B)$ and~$(C,D)$ cross. Corners are blue, links are red, the centre is yellow.}
\label{fig:CornerDiagram}
\end{figure}

The \emph{corner} for the pair $\{A,C\}$ is the vertex set $(A\sm B)\cap (C\sm D)$.
For each pair of sides, one from $(A,B)$ and one from $(C,D)$, we define its corner in the 
analogous way. Thus there are four corners in total. Two corners are \emph{adjacent} if their 
pairs share a side, otherwise they are \emph{opposite}. Note that there is a unique corner opposite 
of each corner and that each corner is the opposite of its opposite corner. 
Each corner has exactly 
two adjacent corners. 

\begin{eg}
The corner for $\{A,C\}$ is opposite to the corner for~$\{B,D\}$, and it is adjacent to the corner for~$\{A,D\}$ and to the corner for~$\{B,C\}$.
\end{eg}

An edge of $G$ is \emph{diagonal} if its endvertices lie in opposite corners. 
Note that an edge is diagonal if and only if it is contained in the separators of both separations $(A,B)$ and~$(C,D)$. 
The \emph{centre} consists of the diagonal edges together with the vertex set $A\cap B\cap C\cap D$. 

An edge $e$ in the separator of $(A,B)$ is \emph{in the edge-link for~$C$} if it is not 
diagonal and has an endvertex in one of the corners for~$C$; that is, in one of the corners for 
$\{A,C\}$ or~$\{B,C\}$.
The \emph{link for~$C$} is the union of the edge-link for~$C$ and the vertex set~$(A\cap B)\sm D$. 
In a slight abuse of notation we will sometimes say things like \lq a vertex of the link $(A\cap B)\sm D$\rq\ instead of the formally precise \lq a vertex of the link for~$C$\rq.
We define \lq the link for~$D$\rq\ as \lq the link for~$C$\rq\ with \lq$D$\rq\ in place of~\lq $C$\rq. 
Note that every edge in the separator of $(A,B)$ that is not diagonal lies in at most one of the edge-links for $C$ and~$D$. 
We define the links for the sides $A$ and $B$ of $(A,B)$ analogously with the 
separations \lq $(A,B)$\rq\ and \lq $(C,D)$\rq\ interchanged.
The link for a side $X$ is \emph{adjacent} to the two corners for the pairs that contain the side~$X$. 
Two links are \emph{adjacent} if there is a corner they are both adjacent~to. 
Every link is adjacent to all but one link; we refer to that link as its \emph{opposite} link.
Note that all four links are pairwise disjoint.

\begin{eg}
The link for~$C$ is adjacent to the two corners for the pairs $\{A,C\}$ and~$\{B,C\}$. 
It is adjacent to the links for $A$~and~$B$. 
It is opposite to the link for~$D$. 
\end{eg}

\begin{lem}\label{EdgesAdjacentLinks}
Let $(A,B)$ and $(C,D)$ be two mixed-separations of a graph~$G$.
Assume that the corner for $\{A,C\}$ is empty.
Let $L_A$ denote the link for~$A$ and let $L_C$ denote the link for~$C$.
Then either there is a vertex in $L_A\cup L_C$ or $L_A=\emptyset=L_C$.
\end{lem}
\begin{proof}
    Assume for a contradiction that $L_A\cup L_C\se E(G)$ but $L_A$ is nonempty.
    Let $e\in L_A$ be an edge.
    By the definition of link, $e$ has an endvertex in the corner for $\{A,C\}$ or in~$L_C$.
    The former is impossible as the corner for $\{A,C\}$ is empty by assumption.
    Hence $L_C$ contains an endvertex of~$e$, a contradiction.
\end{proof}

\begin{lem}\label{nestedviacorners}
Two mixed-separations $(A,B)$ and $(C,D)$ of a graph~$G$ are nested if and only if they admit a corner such that it and the two adjacent links are empty.
Thus, $(A,B)$ and $(C,D)$ cross as soon as two opposite links are nonempty.
\end{lem}
\begin{proof}
For the forward implication, suppose that $(A,B)\le (C,D)$, say.
Let $L_A$ denote the link for~$A$, and let $V_A:=L_A\cap V(G)$.
Similarly, let $L_D$ denote the link for~$D$, and let $V_D:=L_D\cap V(G)$.
Let $X$ denote the corner for~$\{A,D\}$.
From $A\se C$ we get that $X$ and $V_D$ are empty.
From $B\supseteq D$ we get that $X$ and $V_A$ are empty.
Now \autoref{EdgesAdjacentLinks} (with `$\{A,D\}$' in place of `$\{A,C\}$') gives $L_A=\emptyset=L_D$, as desired.

For the backward implication, suppose that $X=L_A=L_D=\emptyset$, say.
In particular, $V_A$ and $V_D$ are empty.
From $X$ and $V_D$ being empty we get $A\se C$.
From $X$ and $V_A$ being empty we get $B\supseteq D$.
Hence $(A,B)\le (C,D)$.
\end{proof}

\autoref{nestedviacorners} offers an alternative definition of nestedness. 
We may and will use the two definitions interchangeably.

Suppose that we are given sides $X\in\{A,B\}$ and $Y\in\{C,D\}$.
The \emph{corner-separator} $L(X,Y)$ at the corner for $\{X,Y\}$ is the union of the two links 
adjacent to the corner for $\{X,Y\}$ together with the centre but without those diagonal edges that 
do not have an endvertex in the corner for~$\{X,Y\}$; see \autoref{fig:CornerSeparator} for a picture.
Two corner-separators are \emph{opposite} or \emph{adjacent} if their respective corners are.

\begin{figure}[ht]
\centering
\includegraphics[height=8\baselineskip]{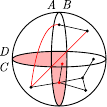}
\caption{$(A,B)$ and~$(C,D)$ cross. The corner separator $L(A,C)$ contains all red vertices and edges.}
\label{fig:CornerSeparator}
\end{figure}

An edge joining vertices in two opposite links is called a \emph{jumping edge}; see \autoref{fig:wickedK4}.
It is straightforward to check 
that jumping edges are the only edges in the separators $S(A,B)$ and $S(C,D)$ that are not in any 
links or the centre and thus not in any corner separators. 

As for separations, we get the following submodularity property:

\begin{lem}\label{submod}
For two mixed-separations $(A,B)$ and $(C,D)$ of a graph, we have
\[
  |L(A,C)|+|L(B,D)|\leq |S(A,B)|+|S(C,D)|.
\]
Moreover, if we have equality, there are no jumping edges, and every diagonal edge has its 
endvertices in the corners for $\{A,C\}$ and $\{B,D\}$.  
\end{lem}

\begin{proof}
This is a standard argument. We just check for each vertex or edge counted in $|L(A,C)|+|L(B,D)|$ 
that it is counted in $|S(A,B)|+|S(C,D)|$ with the same or greater multiplicity.
\end{proof}

\begin{lem}\label{magicLemmaStep1}
Let $G$ be a 3-connected graph.
Let $(A,B)$ and $(C,D)$ be two mixed 3-separations of~$G$ that cross so that two opposite corner-separators have size three.
Then either
\begin{enumerate}
    \item all links have the same size~$\ell$, for some $\ell\in\{0,1\}$; or
    \item two adjacent links have size $i$ and the other two links have size~$3-i$, for some $i\in\{1,2\}$.
\end{enumerate}
\end{lem}
\begin{proof}
Let $a,b,c,d$ denote the sizes of the links for $A,B,C,D$, respectively.
Let $x$ denote the size of the centre.
Without loss of generality, the separators at the corners for $\{A,C\}$ and $\{B,D\}$ have size three.
By \autoref{submod}, every diagonal edge has its endvertices in the corners for $\{A,C\}$ and $\{B,D\}$.
Hence $a+c+x=3$ and $b+d+x=3$.
Since $(A,B)$ and $(C,D)$ have order three, we further have $c+d+x=3$ and $a+b+x=3$.
Considering the two equations that contain~$a$, we find that $b=c$.
Considering the two equations that contain~$c$, we find that $a=d$.
Without loss of generality, $a=d\le b=c$.

Suppose first that $a,d=0$.
Then, since $(A,B)$ and $(C,D)$ cross, the corner for~$\{A,D\}$ must be nonempty.
As $G$ is 3-connected, it follows that the centre has size~$x=3$.
Hence we can read from the equations that $b,c=0$, giving outcome~(1).

Otherwise $a,d=1$, since $a,d\ge 2$ would imply $b,c\le 1<a,d$.
Hence $b,c\le 3-a=2$.
But also $1=a,d\le b,c$.
So $b,c$ take the same value in~$\{1,2\}$, giving outcome (1) or~(2).
\end{proof}

\begin{cor}\label{magicLemmaStep2}
If two tri-separations of a 3-connected graph cross so that two opposite corner-separators have size three, then all links have the same size~$\ell$, for some $\ell\in\{0,1\}$.
\end{cor}
\begin{proof}
Suppose for a contradiction that this fails.
Then two adjacent links have size one, and the other two links have size two, by \autoref{magicLemmaStep1}.
See \autoref{fig:magicLemmaStep2}.
As any two opposite links have sizes one and two but are contained in a common separator of size three, the centre must be empty.
Let $X$ denote the corner whose adjacent links have size one.
Denote these two links by~$L_1$ and~$L_2$.
As the corner-separator for~$X$ has size at most two, but $G$ is 3-connected, the corner~$X$ is empty.
By \autoref{EdgesAdjacentLinks}, there is a vertex $v$ in $L_1$ or~$L_2$.
Since $L_1$ and $L_2$ have size one and the corner~$X$ is empty, it follows that the vertex $v$ has at most one neighbour in one of the sides of the two crossing tri-separations.
This contradicts the definition of tri-separation.
\end{proof}

\begin{figure}[ht]
    \centering
    \includegraphics[height=6\baselineskip]{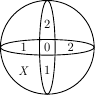}
    \caption{The situation in the proof of \autoref{magicLemmaStep2}}
    \label{fig:magicLemmaStep2}
\end{figure}

\begin{lem}\label{noDiagonalEdges}
If two tri-separations $(A,B)$ and $(C,D)$ of a 3-connected graph~$G$ cross so that two opposite corner-separators have size three and all links have the same size $\ell\in\{0,1\}$, then there are no diagonal edges.
\end{lem}
\begin{proof}
Without loss of generality, the separators at the corners for $\{A,C\}$ and $\{B,D\}$ have size three.
Suppose for a contradiction that there is a diagonal edge~$uv$.
By \autoref{submod}, the ends $u$ and $v$ lie in the corners for $\{A,C\}$ and $\{B,D\}$, respectively say.

\begin{sublem}\label{FeelsFamiliar}
All links and the centre have size one.
\end{sublem}
\begin{cproof}
By \autoref{submod}, there are no jumping edges.
So the separator $S(A,B)$ is partitioned into two opposite links and the centre, and so is the separator $S(C,D)$.
As both separators have size three, to show that all links and the centre have size one it suffices to show that all links have size~$\ell=1$.
Suppose for a contradiction that all links are empty, i.e.\ that~$\ell=0$.
Then the centre has size three.
But since the centre contains the diagonal edge $uv$, the separators at the corners for~$\{A,D\}$ and $\{B,C\}$ have size two. 
Then, since $G$ is 3-connected, the corners for $\{A,D\}$ and $\{B,C\}$ are empty (as otherwise the separator at such a corner is a mixed 2-separator of $G$ that separates a vertex in the corner from $u$ and~$v$, contradicting 3-connectivity).
Since all links are empty as well, it follows that $(A,B)$ and $(C,D)$ are nested, contradicting our assumption that they cross.
\end{cproof}\medskip

By \autoref{FeelsFamiliar}, all links and the centre have size one.
So the centre only consists of the diagonal edge~$uv$.
Hence the separators at the two corners for $\{A,D\}$ and $\{B,C\}$ have size two.
As $G$ is 3-connected, the two corners for $\{A,D\}$ and $\{B,C\}$ are empty (as otherwise the separator at such a corner is a mixed 2-separator of $G$ that separates a vertex in the corner from $u$ and~$v$).
It follows that all four links contain no vertices, since any vertex in a link would fail to have two neighbours in some side of $(A,B)$ or $(C,D)$, contradicting that $(A,B)$ and $(C,D)$ are tri-separations.
Hence all four links contain edges, and only edges.
But then each of these edges must have an end in the corner for $\{A,D\}$ or $\{B,C\}$, contradicting that these corners are empty.
\end{proof}

\begin{figure}[ht]
    \centering
    \includegraphics[height=8\baselineskip]{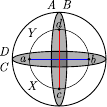}
    \caption{Two tri-separations of a $K_4$ cross with two jumping edges (red and blue)}
    \label{fig:wickedK4}
\end{figure}

\begin{lem}\label{magicLemmaStep3}
If two nontrivial tri-separations $(A,B)$ and $(C,D)$ of a 3-connected graph~$G$ cross so that no two opposite corner-separators have size three, then~$G=K_4$.
\end{lem}
\begin{proof}
This proof is supported by \autoref{fig:wickedK4}.
Since no two opposite corner-separators have size three, we find two adjacent corners whose separators have size at most two by \autoref{submod}.
Say these are the corners $X$ and $Y$ for $\{A,C\}$ and $\{A,D\}$, respectively.
As $G$ is 3-connected, both corners~$X$ and $Y$ must be empty.
Since $(A,B)$ is a tri-separation, it is also a mixed-separation, and by the definition of mixed-separation we have that $A\sm B$ contains some vertex~$a$.
Since $X$ and $Y$ are empty, the vertex~$a$ lies in the link for~$A$.
As $G$ is 3-connected, $a$ has degree at least three.
Since $X$ and $Y$ are empty, and since the corner-separators at~$X$ and $Y$ have size at most two, some edge incident to~$a$ is a jumping edge.
Let $b$ denote the other endvertex of this jumping edge, so $b$ lies in the link for~$B$.

Since $(A,B)$ is nontrivial, there is a cycle $O$ included in~$G[A]$.
As $X$ and $Y$ are empty, the vertices of this cycle lie in the two corner-separators at~$X$ and~$Y$.
The two corner-separators share the vertex~$a$, and have size at most two, hence $O$ must be a triangle which contains~$a$ and whose other two vertices $c,d$ lie in the links for~$C$ and~$D$, respectively.
Thus $cd$ is another jumping edge.
Therefore, the separators at the two corners besides $X$ and~$Y$ have size at most two.
By symmetry, we find that the two corners besides $X$ and $Y$ are empty, and that $bcd$ is a triangle.
Hence~$G=K_4$.
\end{proof}

A mixed-separation $(A,B)$ of a graph~$G$ is \emph{half-connected} if $G[A\sm B]$
or $G[B\sm A]$ is connected.

\begin{lem}\label{crossing implies centre not three}
Let $(A,B)$ and $(C,D)$ be crossing mixed 3-separations of a graph~$G$.
If $(A,B)$ is half-connected, then the centre cannot have size three.
\end{lem}
\begin{proof}
Without loss of generality, $G[A\sm B]$ is connected.
Assume for a contradiction that the centre has size three.
Then all links are empty.
As $G[A\sm B]$ is non-empty, we know that at least one of the two corners included in $A\sm B$ is 
non-empty.
But since $G[A\sm B]$ is connected, the other of the two corners must be empty, contradicting that 
$(A,B)$ and $(C,D)$ are crossing.
\end{proof}

\begin{lem}[Crossing Lemma]\label{dasBesteLemma}
Let $G$ be a 3-connected graph other than~$K_4$.
Let $(A,B)$ and $(C,D)$ be two nontrivial tri-separations of~$G$ that cross.
Then exactly one of the following holds:
\begin{enumerate}
    \item\label{11case} all links have size one and the centre consists of a single vertex;
    \item\label{3case} all links are empty and the centre consists of three vertices.
\end{enumerate}
In particular, there are no jumping edges.
Moreover, if $(A,B)$ or $(C,D)$ is half-connected, then \ref{11case} holds.
\end{lem}
\begin{proof}
Since $G\neq K_4$, it follows from \autoref{magicLemmaStep3} that two opposite corner-separators have size three.
Then all links have the same size $\ell\in\{0,1\}$ by \autoref{magicLemmaStep2}.
By \autoref{submod}, there are no jumping edges.
By \autoref{noDiagonalEdges}, there are no diagonal edges either.
Hence the centre contains no edges, and the size of the centre is determined by~$\ell$.
If $\ell=0$, then the centre has size three; if $\ell=1$, then the centre has size one.

The `Moreover' part follows from \autoref{crossing implies centre not three}.
\end{proof}

\begin{cor}\label{3cutsAreNested}
Let $G$ be a 3-connected graph.
Let $A$ and $B$ be the sides of a non-atomic 3-cut of~$G$.
Then $(A,B)$ is a totally-nested nontrivial tri-separation of~$G$.
\end{cor}
\begin{proof}
By \autoref{trivial}, $(A,B)$ is nontrivial.
Since $S(A,B)$ consists of edges, $(A,B)$ is strong.
By \autoref{trivial_is_nested}, $(A,B)$ is nested with all trivial tri-separations of~$G$.
By the \nameref{dasBesteLemma} (\ref{dasBesteLemma}), $(A,B)$ is nested with all nontrivial tri-separations of~$G$.
\end{proof}

\begin{lem}\label{orangeLemma}
Let $G$ be a 3-connected graph.
If a strong nontrivial tri-separation $(A,B)$ of~$G$ is crossed by a tri-separation of~$G$, then $(A,B)$ is also crossed by a tri-separation of~$G$ that is strong.
\end{lem}
\begin{proof}
Suppose that $(C,D)$ is a tri-separation of~$G$ that crosses~$(A,B)$.
If $(C,D)$ is strong, we are done, so we may assume that some vertex $u$ of~$G$ of degree three lies in the separator of~$(C,D)$.
Since $(A,B)$ is a strong tri-separation, the vertex $u$ cannot lie in the centre, so $u$ lies in a link, say it lies in the link for~$A$.
As a $K_4$ has no strong nontrivial tri-separation, $G$ is not a~$K_4$.
By \autoref{trivial_is_nested}, $(C,D)$ is nontrivial.
Hence we may apply the \nameref{dasBesteLemma} (\ref{dasBesteLemma}) to find that the existence of~$u$ implies that all links have size one while the centre consists of a single vertex, and that there are no jumping edges.
Since $(C,D)$ is a tri-separation and $u$ has degree three, $u$ has a neighbour $v$ in~$C\cap D$.
As there are no jumping edges, the neighbour~$v$ of~$u$ can only lie in the centre.
Since $(A,B)$ is a tri-separation, $v$ has a neighbour $w$ in~$A$ besides~$u$.
By symmetry~$w\in C$.

\begin{sublem}\label{wHasThreeNeighboursInC}
The vertex $v$ has at least three neighbours in~$C$.
\end{sublem}

\begin{cproof}
If the corner $\{B,C\}$ is nonempty, then as $G$ is 3-connected the corner contains a neighbour of~$v$, and together with $u$ and $w$ we have found three neighbours of~$v$ in~$C$. Thus assume that the corner  $\{B,C\}$ is empty. 
By \autoref{EdgesAdjacentLinks}, at least one adjacent link contains a vertex~$y$.
Since the corner for $\{B,C\}$ is empty but $y$ has two neighbours either in $B$ or in~$C$ (depending on which tri-separator $S(A,B)$ or $S(C,D)$ contains~$y$), it follows that $y$ is adjacent to~$v$.
If $y$ is in the link for $B$, then $u$, $y$ and $w$ are three distinct neighbours of $v$ in $C$, and we were done.
So assume that $y$ is in the separator of the strong tri-separation $(A,B)$. Thus $y$ has degree at least four. And since the corner $\{B,C\}$ is empty, the vertex $y$ must have one of its neighbours outside the mixed 3-separator $S(C,D)$ in the corner $\{A,C\}$, and this nonempty corner contains a neighbour of $v$ by 3-connectivity, which is different from $u$ and~$y$. 
\end{cproof}
\medskip

Let $c$ denote the unique neighbour of~$u$ in $C\sm D$.

\begin{sublem}\label{uHasNeighbourInCorner}
The edge $uc$ does not lie in the link for~$C$.
\end{sublem}
\begin{cproof}
Suppose for a contradiction that $uc$ lies in the link for~$C$.
Then $w$ must lie in the corner for $\{A,C\}$.
In particular, the corner for $\{A,C\}$ is nonempty.
So $\{u,v\}$ is a 2-separator, a contradiction to 3-connectivity.
\end{cproof}\medskip

Let $C':=C-u$ and $D':=D$.
Then the separator of $(C',D')$ arises from the separator of $(C,D)$ by replacing the vertex $u$ with the edge~$uc$.
The only vertex in the separator of $(C,D)$ that might lose a neighbour when moving to $(C',D')$ is the vertex~$v$, which loses its neighbour $u$ in~$C'$.
However, \autoref{wHasThreeNeighboursInC} ensures that $v$ has two neighbours in~$C'$.
Hence $(C',D')$ is a tri-separation, and it has fewer vertices of degree three in its separator than~$(C,D)$.
Moreover, $(C',D')$ crosses $(A,B)$ with the same links and centre as for~$(C,D)$, with just one exception: the link for~$A$, which consisted of~$u$ for~$(C,D)$, consists of the edge $uc$ for~$(C',D')$.
By iterating at most two times, we obtain a strong tri-separation that crosses~$(A,B)$.
\end{proof}

\begin{lem}\label{notStrongImpliesCrossed}
If a mixed 3-separation of a 3-connected graph is not strong, then it is crossed by a trivial tri-separation.
\end{lem}
\begin{proof}
Let $(A,B)$ be a mixed 3-separation of a 3-connected graph~$G$, and let $v\in A\cap B$ be a vertex of degree three.
Since $A\sm B$ and $B\sm A$ are nonempty, and since $G$ is 3-connected, the vertex $v$ must have neighbours in $A\sm B$ and in $B\sm A$.
Hence the trivial tri-separation with $\{v\}$ as one side crosses $(A,B)$.
\end{proof}

\begin{prop}\label{expendable}
Let $G$ be a 3-connected graph, and let $(A,B)$ be a nontrivial tri-separation of~$G$.
Then the following assertions are equivalent:
\begin{enumerate}
    \item $(A,B)$ is totally nested;
    \item $(A,B)$ is strong and nested with every strong nontrivial tri-separation of~$G$.
\end{enumerate}
\end{prop}
\begin{proof}
(1)$\to$(2).
We only have to show that $(A,B)$ is strong.
This follows from \autoref{notStrongImpliesCrossed}.

(2)$\to$(1).
Suppose for a contradiction that $(A,B)$ is crossed by a tri-separation $(C,D)$ of~$G$.
By \autoref{orangeLemma}, we may assume that~$(C,D)$ is strong.
Since $(A,B)$ is strong, $(C,D)$ is nontrivial by \autoref{trivial_is_nested}.
This contradicts~(2).
\end{proof}

\section[Proof of Corollary 2]{Proof of \autoref{vx-trans}}\label{sec:Cor2}

Before we prove the \nameref{Angry}, let us see how it implies \autoref{vx-trans}.
A~graph~$G$ is \emph{essentially 4-connected} if it is 3-connected, every nontrivial strong tri-separation has three edges in its separator and the subgraph induced by one side is equal to a triangle, and $G\neq K_4$.
A graph $G$ is \emph{vertex-transitive} if the automorphism group of~$G$ acts transitively on its vertex set~$V(G)$.

\begin{proof}[Proof of \autoref{vx-trans}]
Let $G$ be a vertex-transitive finite connected graph.
We have to show that $G$ either is essentially 4-connected, a cycle, or a complete graph on at most four vertices.
By \jaf{\cite[Theorem~3]{infiniteSPQR}}{\autoref{CayleyCase2}}{\autoref{CayleyCase2}}, $G$ is a cycle, $K_2$, $K_1$ or 3-connected. 
Since we are done otherwise, let us assume that $G$ is 3-connected.
By the \nameref{Angry} (\ref{Angry}), $G$~is internally 4-connected, a $K_{3,m}$ with $m\geq 3$, a wheel, or $G$ has a totally-nested nontrivial tri-separation. 
If $G$ is internally 4-connected, then $G\notin \{K_4,K_{3,3}\}$ by definition, and all strong tri-separations of~$G$ are trivial by \autoref{tetraXtrisep}; in particular, $G$ is essentially 4-connected.
If $G$ is a $K_{3,m}$ for some~$m\ge 3$, then $m=3$ since $G$ is vertex-transitive, and $G=K_{3,3}$ is essentially 4-connected since all its strong tri-separations are trivial.
If $G$ is a wheel, then $G$ can only be a $K_4$ by vertex-transitivity, and $K_4$ is a possible outcome.

As we are done otherwise, we may assume that $G$ has a totally-nested nontrivial tri-separation $(A,B)$.
Every automorphism $\varphi$ of~$G$ takes $(A,B)$ to~$(\varphi(A),\varphi(B))$.
Let $O$ denote the union of the orbits of $(A,B)$ and $(B,A)$ under the automorphism group of~$G$.
As $G$ is finite, we may let $(U,W)$ be $\le$-minimal in~$O$; so $(U,W)\le (C,D)$ or $(U,W)\le (D,C)$ for all $(C,D)\in O$ as $(U,W)$ is totally-nested.

\begin{sublem}\label{Cayley:threeCut}
The separator of $(U,W)$ consists of three edges.
\end{sublem}
\begin{cproof}
Suppose for a contradiction that there is a vertex $v\in U\cap W$.
Since $(U,W)$ is a mixed-separation, there is a vertex $u\in U\sm W$.
Let $\varphi\in\text{Aut}(G)$ send $v$ to~$u$.
Since $(U,W)$ is a tri-separation and $\varphi$ is an automorphism, the pair $(\varphi(U),\varphi(W))$ is also a tri-separation.
Since $(U,W)$ is totally-nested, it is in particular nested with $(\varphi(U),\varphi(W))$.
We cannot have $(\varphi(U),\varphi(W))\le (W,U)$ or $(\varphi(W),\varphi(U))\le (W,U)$ since $u\in \varphi(U)\cap\varphi(W)$ and $u\in U\sm W$.
Hence the only two possibilities for $(U,W)$ and $(\varphi(U),\varphi(W))$ to be nested are $(\varphi(U),\varphi(W))\le (U,W)$ or $(\varphi(W),\varphi(U))\le (U,W)$.
In either case the inequality is strict since $u$ does not lie in $S(U,W)$ but does so after the application of~$\varphi$. This contradicts the choice of~$(U,W)$.
\end{cproof}\medskip

\begin{sublem}\label{Cayley:TriangleSide}
$G[U]=K_3$.
\end{sublem}
\begin{cproof}
Since $G$ is 3-connected, since $(U,W)$ is nontrivial and since $S(U,W)$ consists of three edges by \autoref{Cayley:threeCut}, it suffices to show that every vertex in $U$ is incident with an edge in~$S(U,W)$.
The proof is analogous to the proof of \autoref{Cayley:threeCut}.
\end{cproof}\medskip

By \autoref{Cayley:threeCut} and \autoref{Cayley:TriangleSide}, every tri-separation in~$O$ has three edges in its separator and a side that induces a triangle.
As $(A,B)$ was chosen arbitrarily, every totally-nested nontrivial tri-separation has three edges in its separator and a side that induces a triangle.
Hence to show that every nontrivial strong tri-separation also has three edges in its separator and a side that induces a triangle, it suffices to show that

\begin{sublem}\label{Cayley:TotallyNested}
Every nontrivial strong tri-separation of~$G$ is totally-nested.
\end{sublem}
\begin{cproof}
Since $S(U,W)$ consists of three edges (which share no ends by \autoref{independentEdges}) and $G[U]=K_3$, all vertices in $U$ have degree three.
Hence all vertices of $G$ have degree three.
Let $(C,D)$ be an arbitrary nontrivial strong tri-separation of~$G$.
As $(C,D)$ is strong, the separator of~$(C,D)$ consists of three edges.
Then $(C,D)$ is totally-nested by \autoref{3cutsAreNested}.
\end{cproof}\medskip

Combining \autoref{Cayley:threeCut}, \autoref{Cayley:TriangleSide} and \autoref{Cayley:TotallyNested} yields that $G$ is essentially 4-connected.
\end{proof}

\begin{oprob}
Can \autoref{vx-trans} be used to simplify existing characterisations of classes of finite Cayley graphs (like characterisations of the finite Cayley graphs that embed in the torus or some other surface, as in or similar to~\cite{Proulx,Tucker1,Tucker2})?
\end{oprob}

Another area where our ideas might turn out to be fruitful is in the study of infinite planar Cayley graphs, see \cite{Geo_plane_cubic,georgakopoulos2020planar,georgakopoulos2019planarI,georgakopoulos2023planarII}.

\section{Understanding nestedness through connectivity}\label{sec:externalTricon}

In this section, we provide sufficient conditions for when a tri-separation is totally nested.
Let $v$ be a vertex of a graph $G$. We say that a vertex $w$ of $G$ is \emph{$v$-free} if it is not 
adjacent to $v$ or if it has degree at most three; that is, a vertex is not $v$-free if it is adjacent 
to $v$ and has degree at least four.

Given a mixed 3-separator $\{x_1,x_2,x_3\}$ of~$G$, we say that $\{x_1,x_2,x_3\}$ is \emph{externally
tri-connected around} a vertex $x_i$ with $i\in \Zbb_3$ if one of the following holds:

\begin{enumerate}
\item[(:)] The pair $\{x_{i+1},x_{i+2}\}$ consists of two vertices and these vertices are adjacent 
or joined by three internally disjoint paths in $G-x_i$.
\item[($\dotminus$)] The pair $\{x_{i+1},x_{i+2}\}$ consists of one vertex~$x$ (say) and one 
edge~$e$ (say) such that $e$ has an $x_i$-free endvertex~$y$ for which there are two internally disjoint $x$--$y$ paths in $G-x_i-e$.
\item[($=$)] The pair $\{x_{i+1},x_{i+2}\}$ consists of two edges which have $x_i$-free endvertices $y_{i+1}$ and~$y_{i+2}$, respectively, such that there are two internally disjoint $y_{i+1}$--$y_{i+2}$ paths in \mbox{$G-x_1-x_2-x_3$}.
\end{enumerate}

We say that a mixed 3-separator $\{x_1,x_2,x_3\}$ is \emph{externally tri-connected} if 
$\{x_1,x_2,x_3\}$ is externally tri-connected around each vertex 
$x_i\in\{x_1,x_2,x_3\}$.
We say that a mixed 3-separation is \emph{externally tri-connected} if its separator is 
externally tri-connected.

\begin{eg}
A mixed-separator that consists of three edges or that induces a clique is externally tri-connected. 
\end{eg}

\begin{figure}[ht]
    \centering
    \includegraphics[height=8\baselineskip]{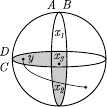}
    \caption{The situation excluded by \autoref{vfreeInCorner}}
    \label{fig:vfreeInCorner}
\end{figure}

For a depiction of the situation excluded by our next lemma, see \autoref{fig:vfreeInCorner}.

\begin{lem}\label{vfreeInCorner}
Let $G$ be a 3-connected graph and $(A,B)$ a half-connected tri-separation of~$G$.
Denote the separator of $(A,B)$ by $\{x_1,x_2,x_3\}$.
Assume that $x_2$ is an edge with an $x_3$-free endvertex~$y$.
If~$(A,B)$ is crossed by a strong tri-separation of~$G$ so that $x_3$ lies in the centre, then $y$ cannot lie in a link.
\end{lem}
\begin{proof}
Let $(C,D)$ be a strong tri-separation of~$G$ that crosses $(A,B)$ so that $x_3$ is in the centre.
Since $K_4$ has no strong tri-separation, $G$ cannot be a~$K_4$.
By the \nameref{dasBesteLemma} (\ref{dasBesteLemma}) and since $(A,B)$ is half-connected, $x_3$~is a vertex and the only element of the centre, all links have size one, and there are no jumping edges.
Without loss of generality, $x_2$ lies in the link for~$C$, and $y\in A\sm B$.
If $y$ lies in the corner for $\{A,C\}$, then we are done.
Otherwise, $y$~must lie in the link for~$A$.
The corner for $\{A,C\}$ must be empty, since otherwise $\{y,x_3\}$ would be a 2-separator of~$G$, contradicting 3-connectivity.
As $y\in S(C,D)$ must have two neighbours in $G[C]$, and since there are no jumping edges, it follows that $yx_3$ must be an edge in~$G$.
But $y\in S(C,D)$ also means that $y$ cannot have degree three as $(C,D)$ is strong, and since $y$ is $x_3$-free this means that the edge $yx_3$ must not be present in~$G$, a contradiction.
\end{proof}

\begin{figure}[ht]
    \centering
    \includegraphics[height=8\baselineskip]{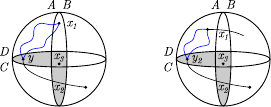}
    \caption{The situation in the second and third case of the proof of \autoref{extTriconcentre}}
    \label{fig:is_nested_dotminus}
\end{figure}

\begin{lem}\label{extTriconcentre}
Let $G$ be a 3-connected graph and $(A,B)$ a half-connected tri-separation of~$G$.
If~$S(A,B)$ is externally tri-connected around some vertex in $S(A,B)$,
then no strong tri-separation of~$G$ can cross $(A,B)$ so that this vertex is in the centre.
\end{lem}

\begin{proof}
Let us denote the separator of $(A,B)$ by $\{x_1,x_2,x_3\}$, and let us assume that $x_3$ is a vertex and that the separator is externally tri-connected around~$x_3$.
Let us assume for a 
contradiction that $(A,B)$ is crossed by a strong tri-separation $(C,D)$ so that $x_3$ lies in the centre.
Since $K_4$ has no strong tri-separation, $G$ is not a~$K_4$. By the \nameref{dasBesteLemma} (\ref{dasBesteLemma}), all links have size one, $x_3$~is the only element of the centre, and there are no jumping edges.
We distinguish three cases. 

(:) In the first case, $x_1$ and $x_2$ are vertices. 
Since there are no jumping edges, 
$x_1$ and $x_2$ are not adjacent. So by external tri-connectivity, there are three internally 
disjoint paths in~$G$ from $x_1$~to~$x_2$ avoiding~$x_3$. Each of them has to meet the two links that contain neither $x_1$ nor~$x_2$, which is not possible as there are three paths and the two links have size one, a contradiction.

($\dotminus$) In the second case, $x_1$ is a vertex and $x_2$ is an edge, say.
Without loss of generality, $x_1$ lies in the link for $D$ while $x_2$ lies in the link for~$C$.
For a depiction of the situation, see \autoref{fig:is_nested_dotminus}.
By external 
tri-connectivity, there are two internally 
disjoint paths $P,Q$ from $x_1$ to an endvertex $y$ of $x_2$ that is $x_3$-free, and these paths avoid  
$x_3$ and~$x_2$. 
Without loss of generality, $y$ lies in~$A\sm B$.
Then the two paths $P,Q$ are contained in~$G[A]$.
Since the link for~$A$ has size one, not both of the two paths $P,Q$ can meet it in internal vertices.
Hence the vertex~$y$ must lie in the link for~$A$.
This contradicts \autoref{vfreeInCorner}.

($=$) In the third case, $x_1$ and $x_2$ are edges. By external 
tri-connectivity, these edges have $x_3$-free endvertices $y_1$ and $y_2$ and there are two 
internally 
disjoint paths $P,Q$ from $y_1$ to $y_2$ avoiding~$x_1, x_2,x_3$. 
By~symmetry assume that the vertex $y_1$ lies in the side~$A$. 
Then the two paths $P,Q$ are contained in $G[A]$ and $y_2$ is in~$A$ as well. 
Since the link for~$A$ has size one, not both of the two paths~$P,Q$ can meet it in internal vertices.
So $y_1$ or~$y_2$ must lie in the link for~$A$.
This contradicts \autoref{vfreeInCorner}. 
\end{proof}

\begin{prop}
\label{is_nested}
Let $G$ be a 3-connected graph and let $(A,B)$ be a tri-sepa\-ra\-tion of~$G$.
If $(A,B)$ is externally tri-connected, half-connected, strong and nontrivial, then $(A,B)$ is totally nested.
\end{prop}
\begin{proof}
Let $(A,B)$ be a tri-separation of~$G$ that is externally tri-connected, half-connected, strong and nontrivial.
Assume for a contradiction that $(A,B)$ is crossed by a tri-separation $(C,D)$ of~$G$.
By \autoref{expendable}, we may assume that $(C,D)$ is strong and nontrivial.
As $K_4$ has no strong nontrivial tri-separation, $G$ is not a~$K_4$.
Hence by the \nameref{dasBesteLemma} (\ref{dasBesteLemma}), $(A,B)$ and $(C,D)$ cross so that the centre contains a vertex.
This contradicts \autoref{extTriconcentre}.
\end{proof}

\begin{eg}\label{always_nested}
In a 3-connected graph~$G$, every strong and nontrivial tri-separation $(A,B)$ with $G[A\sm B]$ connected and 
$G[B\sm A]$ 
disconnected is totally nested.
\end{eg}
\begin{proof}[Proof of \autoref{always_nested}]
Note first that $E(A\sm B,B\sm A)$ is empty since $G$ is 3-connected.
Hence it suffices to find three internally disjoint paths between any pair of vertices in the 
separator $A\cap B$ avoiding the third vertex, by criterion~(:) and \autoref{is_nested}.
By assumption, \mbox{$G\sm (A\cap B)$} has at least three components, and each component has 
neighbourhood 
equal to $A\cap B$ since $G$ is 3-connected.
Thus we find three internally disjoint paths for each pair of vertices in $A\cap B$ through these 
components.
\end{proof}

\begin{lem}\label{K3n lemma}
Let $G$ be a 3-connected graph, and $X$ a set of three vertices in~$G$ such that $G\sm X$ has at least three components.
Let $K$ be a component of $G\sm X$, let $A:=V(K)\cup X$ and let $B:=V(G\sm K)$.
We denote by $(A',B')$ the reduction  of the 3-separation $(A,B)$.
Then the following assertions hold:
\begin{enumerate}
    \item $B'=B$ and $(A',B')\le (A,B)$;
    \item $(A',B')$ is half-connected and strong.
    \item If $(A',B')$ is nontrivial, then it is totally nested.
    \item $(A',B')$ is nontrivial if and only if two vertices in~$X$ are adjacent or $|K|\ge 2$.
\end{enumerate}
\end{lem}
\begin{proof}
(1).~Since every vertex in $X$ has at least one neighbour in every component of $G\sm X$ and since $B$ includes the vertex-sets of at least two such components, the graph~$G[B]$ has minimum degree~$\ge 2$.
From this, we deduce that $B'=B$, and so $(A',B')\le (A,B)$.

(2).~Since the vertex set of the component $K$ is equal to $A'\sm B'$, the tri-separation $(A',B')$ is half-connected. 
To see that $(A',B')$ is strong, let $v$ be a vertex in the separator $A'\cap B'$. 
As $(A',B')$ is a tri-separation, $v$ has two neighbours in~$A'$.
Furthermore, $v$ has two neighbours in $B\sm A$, one in each component by 3-connectivity. Note that $B\sm A\se B'\sm A'$. So $v$ has at least four neighbours.

(3).~Suppose that $(A',B')$ is nontrivial; we have to show that $(A',B')$ is totally nested. For this, it suffices to show that $(A',B')$ is externally tri-connected, by \autoref{is_nested}.
    In the case (:) we construct the three internally-disjoint paths so that they have their internal vertices in different components of $G\sm X$.   
    So assume that we are in the cases ($\dotminus$) or~($=$).
Every vertex $x\in X$ that is reduced to an edge in $S(A',B')$ is $x'$-free for every other vertex $x'\in X$, as $x$ and $x'$ are not adjacent in this case. 
Hence to show that $(A',B')$ is externally tri-connected, it suffices to find two internally 
disjoint paths in $G[B']$ between every two vertices in $X$ avoiding the third vertex in~$X$; these are picked so that their internal vertices are in the two components of $G\sm X$ aside from~$K$.

(4).~If $(A',B')$ is nontrivial, then $G[A]$ contains a cycle, so two vertices in~$X$ are adjacent or~$|K|\ge 2$.
Conversely, suppose now that two vertices in $X$ are adjacent or that~$|K|\ge 2$.
Since $B'=B$ and $|B|\ge 2$, it follows that $|B'| \ge 2$. 
Thus it suffices to show that $A'$ contains at least two vertices, by \autoref{trivial}. 
If~two vertices in~$X$ are adjacent, then these two vertices are not reduced to edges in $S(A',B')$, so they lie in~$A'$ and we are done.
So assume that $|V(K)|\ge 2$. Since $V(K)\se A'$, we are done as well. 
\end{proof}

\begin{cor}\label{conjunction}
Let $G$ be a 3-connected graph with a tri-separation $(C,D)$ that is not half-connected.
If $G$ has no totally-nested nontrivial tri-separation, then $G=K_{3,m}$ for some~$m\ge 4$.
\end{cor}
\begin{proof}
As  $(C,D)$ is not half-connected, by 3-connectivity of $G$  its separator $X$ consists of three vertices; and $G\sm X$ has at least four components.
By \autoref{K3n lemma}, no two vertices in $X$ are adjacent, and every component of $G\sm X$ is trivial.
As $G$ is 3-connected, every component of $G\sm X$ has neighbourhood equal to~$X$. 
So $G$ is a $K_{3,m}$ for some~$m\ge 4$.
\end{proof}

\section{Background on 2-separations}\label{sec:BackgroundOnTutte}

 A~\emph{tree-decomposition} of a graph $G$ is a pair $(T,\Vcal)$ of a tree $T$ and a family 
$\Vcal=(V_t)_{t\in T}$ of vertex sets $V_t\se V(G)$ indexed by the nodes $t$ of~$T$ such that the 
following conditions are satisfied:
\begin{enumerate}
    \item[(T1)] $G=\bigcup_{t\in T}G[V_t]$;
    \item[(T2)] for every $v\in V(G)$, the vertex set $\{\,t\in T\mid v\in V_t\,\}$ is connected 
in~$T$.
\end{enumerate}
The vertex sets $V_t$ and the subgraphs $G[V_t]$ they induce are the \emph{bags} of this
decomposition.
The intersections $V_{t_1}\cap V_{t_2}$ for edges $t_1 t_2\in E(T)$ are the \emph{adhesion sets} 
of~$(T,\Vcal)$.
The \emph{adhesion} of $(T,\Vcal)$ is the maximum size of an adhesion set of~$(T,\Vcal)$.
The \emph{torso} of a bag is the graph obtained from $G[V_t]$ by adding for every neighbour $t'$ of
$t$ in~$T$ every possible edge $xy$ with both endvertices in the adhesion set $V_t\cap V_{t'}$.
We point out that the edges $xy$ are not required to be edges of~$G$, so each adhesion set $V_t\cap 
V_{t'}$ induces a complete graph in the torso of $G[V_t]$, and in particular torsos need not be 
subgraphs of~$G$.
The edges of a torso that are not edges of the bag are called \emph{torso edges}.

Recall that in this paper, every separation $(A,B)$ is required to satisfy $A\sm B\neq \emptyset\neq B\sm A$.
Every edge $t_1 t_2$ of~$T$, when directed from $t_1$ to $t_2$ say, \emph{induces} the separation 
$(X_1,X_2)$ of~$G$ for $X_i:=\bigcup_{t\in T_i}V_t$, where $T_i$ is the component of $T-t_1 t_2$ 
that contains~$t_i$, provided that both $X_1\sm X_2$ and $X_2\sm X_1$ are non-empty.
We call these separations the \emph{induced} separations of~$(T,\Vcal)$.
In this paper, all tree-decompositions have the property that all their edges induce separations.
The separator of $(X_1,X_2)$ is the adhesion set $V_{t_1}\cap V_{t_2}$, which is why we also refer 
to the adhesion sets of $(T,\Vcal)$ as the \emph{separators} of~$(T,\Vcal)$.

Let us call a set~$S$ of separations of $G$ \emph{symmetric} if $(A,B)\in S$ implies $(B,A)\in S$ for 
all $(A,B)\in S$.
A set $S$ of separations of~$G$ \emph{induces} a tree-decomposition $(T,\Vcal)$ of~$G$ if the map $(t_1,t_2)\mapsto (X_1,X_2)$ is a bijection between the directed edges of~$T$ and the set~$S$.

We shall use the \nameref{2SeparatorTheorem} of Tutte~\cite{TutteGrTh} with the total-nestedness description by Cunningham and Edmonds 
\cite{cunningham_edmonds_1980}, which we recall below with the notation most suitable here.
Let us say that a 2-separation of a graph is \emph{totally nested} if it is nested with every 2-separation of the graph.
A cycle $O$ \emph{alternates} between two sets $X,Y\se V(G)\cup E(G)$ if $O$ contains exactly two elements $x_1,x_2$ from $X$ and exactly two elements $y_1,y_2$ from $Y$ and $O$ induces the cycling ordering $x_1 < y_1 < x_2 < y_2$ on these four elements (possibly after swapping the names of the $x_i$ or of the $y_i$).
We emphasise that $O$ may have arbitrary length and still alternate.

\begin{thm}[2-Separation Theorem]\label{2SeparatorTheorem}
For every 2-connected graph~$G$, the totally-nested 2-separa\-tions of~$G$ induce a tree-decomposition 
$(T,\Vcal)$ of $G$ all whose torsos are minors of $G$ and are 3-connected, cycles, or~$K_2$'s.
Moreover, $(T,\Vcal)$ is canonical and has the following two properties:
\begin{enumerate}
    \item If $(A,B)$ and $(C,D)$ are two mixed 2-separations of $G$ that cross so that all four links have size one (and the centre is empty), then there exists a unique node $t\in T$ such that the associated torso is a cycle which alternates between $S(A,B)$ and~$S(C,D)$.
    \item If the torso associated with a node $t\in T$ is 3-connected or a cycle, then the adhesion sets induced by the edges $st\in E(T)$ are pairwise distinct.
\end{enumerate}
\end{thm}
We refer to the unique tree-decomposition of $G$ provided by the \nameref{2SeparatorTheorem} as the \emph{Tutte-decomposition} of~$G$ as customary.
We provide a proof of the \nameref{2SeparatorTheorem} in \autoref{sec:2SeparatorTheorem}.
A far reaching extension of the 2-separation theorem (that also applies to infinite matroids and extends \cite{infiniteSPQR,richter2004decomposing}) was proved by  Aigner-Horev, Diestel and Postle~\cite{Matroid2seps}.

\section{Apex-decompositions}\label{sec:ProofMain}

Recall that a \emph{star} is a rooted tree with at most two levels.
The root of the star is commonly referred to as its \emph{centre}.
A \emph{star-decomposition} means a tree-decomposition whose decomposition tree is a star.

Let $G$ be a graph and $v\in V(G)$ a vertex.
An \emph{apex-decomposition} of~$G$ with \emph{centre}~$v$ is a star-decomposition $\Acal$ of $G-v$ of adhesion two such that its central torso is a cycle~$O$ and all adhesion sets are pairwise distinct.
We refer to $O$ as the \emph{central torso-cycle} of~$\Acal$.
The intersection of a leaf-bag $B_\ell$ of~$\Acal$ with the centre-bag of~$\Acal$ is the \emph{adhesion set of~$B_\ell$}.
We call the edges of~$O$ that are spanned by adhesion sets of leaf-bags \emph{bold}, and all other edges of~$O$ are \emph{timid}.
Note that the timid edges of~$O$ exist in~$G$, while possibly some but not necessarily all bold edges of~$O$ exist in~$G$.
An apex-decomposition is \emph{2-connected} if all its leaf-bags $G[B_\ell]$ are 2-connected.
We stress that we really consider the leaf-bags $G[B_\ell]$ rather than the torsos, and that this will be important in the proofs to come.

\begin{lem}\label{v has neighbours}
Let $G$ be a 3-connected graph, let $\Acal$ be a 2-connected apex-decomposition of~$G$ with centre~$v$, and let $O$ denote the central torso-cycle of~$\Acal$.
In~$G$, the vertex $v$ has a neighbour in $B_\ell \sm V(O)$ for every leaf-bag~$B_\ell$ of~$\Acal$.
\end{lem}
\begin{proof}
Since $\Acal$ is 2-connected, $B_\ell$ has at least three vertices, so $B_\ell\sm V(O)$ is non-empty.
If $v$ had no neighbour in $B_\ell \sm V(O)$, then the adhesion set of~$B_\ell$
would form a 2-separator of~$G$, contradicting that~$G$ is 3-connected.
\end{proof}

\begin{figure}[ht]
    \centering
    \includegraphics[height=7\baselineskip]{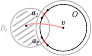}
    \caption{In the depicted situation, the separator of the pseudo-reduction induced by~$\ell$ consists of the red edges}
    \label{fig:EF}
\end{figure}

Let $\Acal=(S,\Bcal)$ be a 2-connected apex-decomposition of $G$ with centre~$v$.
We call a vertex $u$ in the adhesion set of a leaf-bag~$B_\ell$ of~$\Acal$ \emph{edgy} if all but exactly one of the neighbours of~$u$ in~$G$ lie in~$B_\ell$.

\begin{obs}\label{edgySituation}
If a vertex $u$ in the adhesion set of~$B_\ell$ is edgy, then 
\begin{enumerate}
    \item $v$ is not a neighbour of~$u$ in~$G$, and
    \item the two edges of~$O$ incident with~$u$ are bold and timid.\qed
\end{enumerate}
\end{obs}

Each leaf $\ell$ of $S$ \emph{induces} the 2-separation $(B_\ell,\bigcup_{t\in S-\ell}B_t)=:(X_\ell,Y_\ell)$ of~$G-v$.
The \emph{pseudo-reduction} of $(X_\ell,Y_\ell)$ is the mixed 3-separation $(X,Y)$ of~$G$ defined as follows (see \autoref{fig:EF}):
\begin{itemize}
    \item $X$ is obtained from $X_\ell$ by adding~$v$ unless $v$ has at most one neighbour in $X_\ell$, and
    \item $Y$ is obtained from $Y_\ell+v$ by removing any vertex that lies in the adhesion set of~$B_\ell$ and is edgy.
\end{itemize}

A set $\sigma=\{\,(A_i,B_i): i\in I\,\}$ of mixed-separations of~$G$ is a \emph{star} with \emph{leaves}~$A_i$ if $(A_i,B_i)< (B_j,A_j)$ for all distinct indices~$i,j\in I$.
In this context, we also refer to~$A_i$ as the \emph{leaf-side} of~$(A_i,B_i)$.

\begin{eg}\label{starEG}
Stars of genuine separations correspond to star-decompositions, as follows.
On the one hand, if $(S,\Vcal)$ is a star-decomposition of~$G$ and $c$ is the central node of~$S$, then the separations induced by the edges of~$S$ incident with~$c$ and directed to~$c$ form a star~$\sigma_c$ of separations with leaves~$V_\ell$ where the nodes~$\ell$ are the leaves of~$S$.
On the other hand, if $\sigma=\{\,(A_i,B_i): i\in I\,\}$ is a star of separations with leaves~$A_i$, then it defines a star-decomposition $(S,\Vcal)$ of~$G$ with leaf-bags~$A_i$ and which induces $\sigma$ in the sense that $\sigma=\sigma_c$, where
\begin{itemize}
    \item $S$ is a star whose set of leaves is equal to~$I$, and
    \item $V_i:=A_i$ for $i\in I$ while $V_c:=V(G)\sm \bigcup\,\{A_i:i\in I\}$, with $c$ denoting the central node of~$S$.
\end{itemize}
\end{eg}

The set of pseudo-reductions of the separations induced by~$\Acal$ is a star of mixed 3-separations of~$G$, which we call the \emph{tri-star} of~$\Acal$.
We will show in \autoref{is_splitting_star} that the elements of the tri-star are tri-separations, provided that $O$ essentially is not too short.

\begin{rem}
The pseudo-reduction of~$(X_\ell,Y_\ell)$ need not be a reduction of the 3-separation $(X_\ell+v,Y_\ell+v)$ of~$G$.
Indeed, suppose that $G$ is a 3-connected graph which has an apex-decomposition $\Acal=(S,\Bcal)$ with centre~$v$.
Suppose further that $\Acal$ has a leaf-bag $B_\ell$ with adhesion set $\{a_1,a_2\}$ such that no other leaf-bag contains any~$a_i$, that $a_1 a_2$ is an edge in~$G$ but neither $a_1$ nor $a_2$ is adjacent to~$v$, and that $v$ has at least two neighbours in~$B_\ell$ and at least two neighbours on~$O$.
Then $a_1$ and $a_2$ are edgy, and so they are not in the separator of the pseudo-reduction induced by~$\ell$.
However, the 3-separation $(X_\ell+v,Y_\ell+v)$ is a tri-separation of~$G$.
So in this case the pseudo-reduction is not a reduction of the 3-separation $(X_\ell+v,Y_\ell+v)$ of~$G$.
\end{rem}

Let $\sigma$ be a star of mixed-separations of a graph~$G$.
We say that a mixed-separation $(A,B)$ of~$G$ \emph{interlaces}~$\sigma$ if for every $(C,D)\in\sigma$ either $(C,D)<(A,B)$ or $(C,D)<(B,A)$; see \autoref{fig:interlace}.
Note that for every $(C,D)\in\sigma$, neither $(C,D)$ nor $(D,C)$ interlaces~$\sigma$.

\begin{figure}[ht]
    \centering
    \includegraphics[height=8\baselineskip]{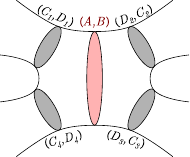}
    \caption{A mixed-separation $(A,B)$ interlaces a star $\{\,(C_i,D_i):i\in [4]\,\}$}
    \label{fig:interlace}
\end{figure}

\begin{lem}\label{apex_exists}
Let $G$ be a 3-connected graph with two tri-separations $(A,B)$ and $(C,D)$ that cross so that their separators intersect only in a vertex~$v$ and all links have size one.
Then $G$ has a 2-connected apex-decomposition $\Acal$ with centre~$v$ such that $(A,B)$ and $(C,D)$ interlace the tri-star of~$\Acal$ and the central torso-cycle of~$\Acal$ alternates between $S(A,B)-v$ and $S(C,D)-v$.
\end{lem}

\begin{proof}
Let us consider the 2-connected graph~$G':=G-v$.
Let $\Tcal=(T,\Vcal)$ be the Tutte-decomposition of~$G$ provided by the \nameref{2SeparatorTheorem} (\ref{2SeparatorTheorem}).
Since $(A,B)$ and $(C,D)$ cross in~$G$ so that the centre consists of~$v$ and all links have size one, their induced mixed 2-separations of~$G'$ cross with empty centre and all links of size one.
Hence there is a 
bag $V_t\in\Vcal$ whose torso is a cycle $O$ that
alternates between $S(A,B)-v$ and $S(C,D)-v$.

Let~$S$ be the star obtained from~$T$ by contracting 
all edges of~$T$ not incident with~$t$. 
Put~$B_t:=V_t$.
For each leaf~$\ell$ of~$S$, we let~$B_\ell$ be the union 
of all bags $V_s\in\Vcal$ with~$s\in \ell$. 
Then $\Acal:=(S,\Bcal)$ with $\Bcal:=(\,B_s: s\in V(S)\,)$ is an apex-decomposition of~$G'$.
Note that $O$ is equal to the torso of~$B_t$.

Next, we show that the leaf-bags of~$\Acal$ are 2-connected subgraphs of~$G'$.
Let~$B_\ell$ be any leaf-bag of~$\Acal$ and let $\{a_1,a_2\}$ denote its adhesion set.
By construction, the torso of the bag~$B_\ell$ is 2-connected, so $B_\ell$ has at least three vertices.
Furthermore, $B_\ell$ is a side of totally-nested 2-separation $(B_\ell,X)$ of~$G'$ with separator~$\{a_1,a_2\}$.
So at least one of $G'[B_\ell]$ and $G'[X]$ is 2-connected. If $a_1$ and $a_2$ are adjacent in $G$, both are 2-connected and we are done.
Otherwise every vertex of $O$ other than~$a_1,a_2$ witnesses that $G'[X]$ is not 2-connected, so $G'[B_\ell]$ is 2-connected as desired.

It remains to show that both $(A,B)$ and $(C,D)$ interlace the tri-star of~$\Acal$.
By symmetry, it suffices to show this for $(A,B)$.
Consider any pseudo-reduction $(X,Y)$ induced by a leaf~$\ell$ of~$S$.

\begin{sublem}\label{boldEdgeNoAdhesionSet}
The adhesion-set of $B_\ell$ is distinct from the set of endvertices of every edge in $S(A,B)\cup S(C,D)$.
\end{sublem}
\begin{cproof}
Otherwise, since $G[B_\ell]$ is 2-connected, it would provide a path joining the endvertices of the edge while avoiding the separator that contains the edge, a contradiction.
\end{cproof}\medskip

Hence, without loss of generality, the leaf-bag $B_\ell$ is included in~$A$.
We claim that $(X,Y)\le (A,B)$.
The side $X$ is obtained from $B_\ell$ by possibly adding the vertex~$v$.
Since $v$ lies in the separator of $(A,B)$ by assumption, this gives~$X\subseteq A$.
The side $Y$ is obtained from $(V(G)\sm B_\ell)\cup\{a_1,a_2\}$ by possibly removing some of the vertices in the adhesion set $\{a_1,a_2\}$ of~$B_\ell$.
Since $B$ is included in $(V(G)\sm B_\ell)\cup\{a_1,a_2\}$, it suffices to show that $a_i\notin Y$ implies $a_i\notin B$ for both~$i=1,2$.
If $a_i$ is not contained in~$Y$, then this is because $a_i$ is edgy, i.e.~$a_i$ has just one neighbour outside $B_\ell$ in~$G$.
If $B$ contains $a_i$, then $a_i$ has two neighbours in~$B$ since $(A,B)$ is a tri-separation, and these two neighbours lie in $G-v$ as~$a_i$, being edgy, is not adjacent to~$v$. 
The neighbour of~$a_i$ in~$B\cap B_\ell$ can only be $a_{3-i}$ since $B\cap B_\ell\se\{a_1,a_2\}$.
But then $O$ cannot alternate between the separators of $(A,B)$ and $(C,D)$, as $\{a_1,a_2\}\se S(A,B)$ but $a_1a_2$ is a bold edge of~$O$ and therefore cannot lie in~$S(C,D)$.
\end{proof}

We label the edges of the central torso-cycle $O$ of an apex-decomposition~$\Acal$ with the letters \ty{b} or~\ty{t}, depending on whether they are bold or timid, respectively. 
The cyclic sequence of these letters is the \emph{type} of~$O$.
For example, $O$ has type \ty{btbt} if $O$ is a 4-cycle $O=v_0 v_1 v_2 v_3 v_0$ and there is an index $i$ such that the edges along either $v_i v_{i+1} v_{i+2} v_{i+3} v_{i+4}$ or $v_i v_{i-1} v_{i-2} v_{i-3} v_{i-4}$ (using indices in~$\mathbb{Z}_4$) read \ty{btbt}.
When $O$ has type \ty{btbt}, we say that $O$ has the type $\ty{btbt}^-$ if $O$ additionally has a timid edge both of whose endvertices are not adjacent to~$v$; otherwise we say that $O$ has the type~$\ty{btbt}^+$.

\begin{obs}\label{plus}
A central torso-cycle of type $\ty{btbt}^+$ has two non-adjacent vertices that are adjacent to $v$ or else the two endvertices of one bold edge are neighbours of $v$ while no endvertex of the other bold edge is adjacent to~$v$. \qed
\end{obs}

\begin{setting}\label{X}
Let $\Acal=(S,\Bcal)$ be a 2-connected apex-decom\-po\-sition of a 3-connected $G$ with centre~$v$.
Denote the central torso-cycle of~$\Acal$ by~$O$.
\end{setting}

When we say that we assume \autoref{X} \emph{with crossing tri-separations}, this means that we also assume that the tri-star of $\Acal$ is interlaced by two crossing tri-separations of~$G$ such that their separators intersect exactly in the vertex~$v$ and such that $O$ alternates between the two separators (minus~$v$).

\begin{obs}\label{noBoldEdgeObs}
    In \autoref{X} with crossing tri-separations $(A,B)$ and $(C,D)$, bold edges of $O$ cannot belong to $S(A,B)\cup S(C,D)$.
\end{obs}
\begin{proof}
    The same as for \autoref{boldEdgeNoAdhesionSet}.
\end{proof}

\begin{lem}\label{tiny-cases}
In \autoref{X} with crossing tri-separations, $O$ does not have type \ty{bbt}, \ty{bbb} or $\ty{btbt}^-$.
\end{lem}
\begin{proof}
Let $(A,B)$ and $(C,D)$ denote the two crossing tri-separations of~$G$ from the assumption.
Let us suppose for a contradiction that $O$ has one of the types we claim it hasn't.

\textbf{Case} \ty{bbt}.
Since $O$ alternates between $S(A,B)-v$ and $S(C,D)-v$ by assumption, 
but neither $S(A,B)$ nor $S(C,D)$ can contain a bold edge of~$O$ by \autoref{noBoldEdgeObs}, it follows that one of $S(A,B)$ and $S(C,D)$ contains both endvertices of the timid edge of~$O$, say $S(A,B)$ contains them.
But then one of $A\sm B$ or $B\sm A$ is empty, contradicting that $(A,B)$ is a tri-separation.

\textbf{Case} \ty{bbb}. Here we find that $S(A,B)$ and $S(C,D)$ must share a vertex on~$O$, a contradiction.

\textbf{Case} $\ty{btbt}^-$.
Let $e=xy$ be a timid edge of $O$ such that neither $x$ nor $y$ is adjacent to~$v$ in~$G$.
Let us write $O=:wxyz$. 
The edges $wx$ and $yz$ lie in neither $S(A,B)$ nor $S(C,D)$ since they are bold.

We claim that the vertices $x$ and $y$ lie in neither $S(A,B)$ nor $S(C,D)$ as well.
Assume for a contradiction that $x\in S(A,B)$, say. 
Then $w\notin S(A,B)$, since $O$ alternates between $S(A,B)-v$ and $S(C,D)-v$, and since $wx$ is bold.
Hence the leaf-bag $B_\ell$ of $\Acal$ with adhesion set $\{w,x\}$ meets $S(A,B)$ only in~$x$. 
As $G[B_\ell]$ is 2-connected, $G[B_\ell]-x$ is connected, so $B_\ell$ is included in $A\sm B$ or in $B\sm A$, say in~$A\sm B$.
But since $v$ is not a neighbour of $x$ in~$G$, all neighbours of $x$ in~$G$ besides $y$ lie in $B_\ell$.
Thus $x$ has at most one neighbour in~$B$, contradicting that $(A,B)$ is a tri-separation.

So neither vertex $x,y$ and neither edge $wx,yz$ lies in $S(A,B)$ or $S(C,D)$.
Since $O$ alternates between $S(A,B)$ and $S(C,D)$ and only the vertices $w,z$ and the edges $e,wz$ can lie in $S(A,B)$ or $S(C,D)$, we find that $S(A,B)$ or $S(C,D)$ must contain two elements of $\{w,wz,z\}$.
Say $S(A,B)$ contains two elements.
These two elements can only be $w$ and~$z$, since separators of mixed-separations do not contain both a vertex and an edge incident to that vertex.
But then $A\sm B$ or $B\sm A$ is empty, a contradiction.
\end{proof}

\begin{prop}\label{is_splitting_star}
Assume \autoref{X} with crossing tri-separations. Then the tri-star of~$\Acal$ consists of totally-nested strong nontrivial tri-separations.
\end{prop}

We prove \autoref{is_splitting_star} across the next two sections.

\section[Special cases]{Proof of \autoref{is_splitting_star}: Special cases}\label{sec:specialAngry}

\begin{figure}[ht]
    \centering
    \includegraphics[height=7\baselineskip]{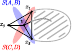}
    \caption{The situation in the proof of \autoref{is_splitting_star_case3}}
    \label{fig:Length3case}
\end{figure}

\begin{lem}\label{is_splitting_star_case3}
Assume \autoref{X} with crossing tri-separations.
If $O$ has length three, then the tri-star of~$\Acal$ consists of totally-nested strong nontrivial tri-separations, $O$ has type \ty{ttt} or \ty{btt}, and $v$ is adjacent to all vertices of~$O$.
\end{lem}

\begin{proof}
Let $(A,B)$ and $(C,D)$ denote the two crossing tri-separations from the assumptions.
If~$O$ has type \ty{ttt}, then the tri-star of $\Acal$ is empty, and we are done.
So $O$ has at least one bold edge.
By \autoref{tiny-cases}, $O$ does not have type \ty{bbt} or \ty{bbb}, so $O$ must have type~\ty{btt}.
Let $x_1 x_2 x_3:=O$ so that $x_1 x_2$ is the bold edge of~$O$; see \autoref{fig:Length3case}.

Since $O$ alternates between $S(A,B)-v$ and $S(C,D)-v$, since $x_1 x_2$ cannot lie in $S(A,B)$ or $S(C,D)$, and since neither separator can contain both endvertices of $x_2x_3$ or of $x_3x_1$, we find that $x_3$ cannot lie in either separator.
Thus only four elements of $V(O)\sqcup E(O)$ can lie in $S(A,B)$ or $S(C,D)$.
Hence the separators of $(A,B)$ and $(C,D)$ are determined up to symmetry, say $S(A,B)-v=\{x_2,x_1 x_3\}$ and $S(C,D)-v=\{x_1,x_2 x_3\}$; see \autoref{fig:Length3case}.
Using that $(A,B)$ and $(C,D)$ are tri-separations, we infer that $v$ is adjacent to all three vertices $x_1,x_2,x_3$ of~$O$.
We also recall that $v$ has a neighbour in the leaf-bag of~$\Acal$ other than $x_1$ and~$x_2$ by \autoref{v has neighbours}.
Hence $(E,F):=(V(G)-x_3,\{x_1,x_2,x_3,v\})$ with $S(E,F)=\{v,x_1,x_2\}$ is a strong tri-separation of~$G$ and the unique element of the tri-star of~$\Acal$.
It follows immediately from the definition of $(E,F)$ that $(E,F)$ is nontrivial.
Since $G[F\sm E]$ is a~$K_1$, the tri-separation $(E,F)$ is half-connected.

We claim that $(E,F)$ is totally nested. 
By \autoref{is_nested}, it suffices to show that $(E,F)$ is externally tri-connected.
As~$v$ is joined to $x_1$~and~$x_2$ by edges, it remains to show that $S(E,F)$ is externally tri-connected around~$v$. For this, note that $x_1$~and~$x_2$ are joined by 
three 
internally disjoint paths avoiding~$v$: we find two paths in the 2-connected leaf-bag, and 
the third path is~$x_1x_3x_2$. 
\end{proof}

\begin{lem}\label{o4_not_alternate}
Assume \autoref{X}.
If $O$ has type $\ty{btbt}^+$, then the tri-star of~$\Acal$ consists of totally-nested strong nontrivial tri-separations.
\end{lem}

\begin{proof}
Since $O$ has type \ty{btbt} and $\Acal$ is 2-connected, $G-v$ is obtained from the disjoint union of two 2-connected graphs $X$ and~$Y$ 
by adding a matching of size two between~$X$ and~$Y$. Call the two matching edges $e_1$~and~$e_2$.
We denote the endvertices of~$e_i$ in~$X$ and in~$Y$ by $x_i$ and~$y_i$, respectively, for both~$i=1,2$.
Since the torso-cycle $O$ has type $\ty{btbt}^+$, we can use \autoref{plus} to find that $O$ has two opposite vertices that are adjacent to $v$ or else the two ends of some bold edge of~$O$ are neighbours of $v$ while no endvertex of the other bold edge is adjacent to~$v$.
We consider the two cases separately.

\textbf{Case~1:} $O$ has two opposite vertices that are adjacent to~$v$ in~$G$.
Without loss of generality, $x_1$ and~$y_2$ are adjacent to~$v$ in~$G$.
By symmetry, it suffices to show that the pseudo-reduction induced by~$X$ (viewing $X$ as a leaf-bag of~$\Acal$) is a totally-nested nontrivial tri-separation of~$G$.
Since $x_1 v$ is an edge in~$G$, and since $v$ has a neighbour in $X\sm \{x_1,x_2\}$ by \autoref{v has neighbours}, the pseudo-reduction induced by~$X$ is either $(X+v,Y\cup\{v,x_1,x_2\})$ with separator~$\{x_1,v,x_2\}$ or $(X+v,Y\cup\{v,x_1\})$ with separator $\{x_1,v,e_2\}$, depending on whether $x_2v$ is an edge in~$G$ or not, respectively.
In either case, since $v$ also has a neighbour in $Y\sm\{y_1,y_2\}$ by \autoref{v has neighbours}, we have a half-connected nontrivial tri-separation that is strong.
So by \autoref{is_nested}, it remains to show external tri-connectivity.

\textbf{Subcase~1a:} $x_2v$ is an edge in~$G$.
Then the separator is $\{x_1,v,x_2\}$.
External tri-connectivity around~$x_i$ is witnessed by the edge~$x_{3-i}v$ for both~$i=1,2$.
External tri-connectivity around $v$ is witnessed by two internally disjoint $x_1$--$x_2$ paths through~$X$ and a third $x_1$--$x_2$ path which passes through~$Y$ via the edges $e_1$ and~$e_2$.

\textbf{Subcase~1b:} $x_2v$ is not an edge in~$G$.
Then the separator is $\{x_1,v,e_2\}$.
The endvertex $x_2$ of~$e_2$ is $v$-free as $x_2v$ is not an edge.
For external tri-connectivity around~$v$, we find two internally disjoint $x_1$--$x_2$ paths in~$X$.
The endvertex $y_2$ of~$e_2$ is $x_1$-free as $x_1 y_2$ is not an edge in~$G$.
For external tri-connectivity around~$x_1$, we find two internally disjoint $y_2$--$v$ paths in~$G-x_1-e_2$, one through $Y$ and via a neighbour of $v$ in $Y\sm\{y_1,y_2\}$ (which exists by \autoref{v has neighbours}), and the second one is~$y_2 v$.

\textbf{Case~2:} 
$x_1$ and $x_2$ are adjacent to~$v$ while $y_1$ and~$y_2$ are not, say.
First, we consider the pseudo-reduction induced by~$X$.
This is $(X+v,Y\cup\{v,x_1,x_2\})$ with separator $\{v,x_1,x_2\}$.
We verify external tri-connectivity as in Subcase~1a.
The pseudo-reduction induced by~$Y$ is $(X+v,Y+v)$ or $(X+v,Y)$, depending on whether $v$ has more than one or just one neighbour in $Y\sm\{y_1,y_2\}$, respectively.
In either case, we have a strong half-connected nontrivial tri-separation, and by \autoref{is_nested} it remains to verify external tri-connectivity.
For $(X+v,Y+v)$ we notice that $y_1$ and $y_2$ are $v$-free and find two internally disjoint $y_1$--$y_2$ paths in~$Y$, which suffices as $v$ is the only vertex in the separator $\{v,e_1,e_2\}$.
For $(X+v,Y)$ there is nothing to show as its separator consists of three edges.
\end{proof}

\section[General case]{Proof of \autoref{is_splitting_star}: General case}\label{sec:GeneralAngry}

In the previous section, we have seen that the conclusion of \autoref{is_splitting_star} holds if $O$ has length three (\autoref{is_splitting_star_case3}) or if $O$ has type \ty{btbt} (\autoref{tiny-cases} with \autoref{o4_not_alternate}).
In this section, we show that the conclusion of \autoref{is_splitting_star} also holds if $O$ has length at least four and $O$ does not have type \ty{btbt}, thereby completing the proof of \autoref{is_splitting_star}.

\begin{lem}\label{ApexGeneral}
Assume \autoref{X}.
If $O$ has length at least four and $O$ does not have type \ty{btbt}, then the tri-star of~$\Acal$ consists of totally-nested strong nontrivial tri-separations.
\end{lem}

To prove \autoref{ApexGeneral} systematically, we need some machinery.
\autoref{u-lemma} below will help us to find paths for verifying external tri-connectivity of mixed 3-separators.
To make \autoref{u-lemma} applicable in various settings, we introduce the following definitions which will allow us to deal with cases systematically in \autoref{u-lemma} and its applications.
Assume \autoref{X}.
A \emph{pattern} is any of the following five words: 
\[
    \text{\ty{bb}, \ty{btx}, \ty{tbb}, \ty{tbtx}, \ty{ttx}}.
\]
We say that a finite sequence of consecutive edges in a given cyclic orientation of~$O$ \emph{has 
pattern}~$p$ if~$p$ is a pattern and the labels of the edges in the sequence start with the 
pattern~$p$ after possibly replacing an occurrence of~\ty{x} in~$p$ with either \ty{b}~or~\ty{t}.
Note that an edge-sequence of length at least four has a unique pattern; we refer to this pattern 
as \emph{its} pattern.
If a sequence $e_0,\ldots,e_n$ has pattern~$p$, and~$e_i$ is the last edge in the sequence which 
contributes to~$p$, then the endvertex of~$e_i$ that is not incident with~$e_{i-1}$ is called the 
\emph{capstone} of the sequence $e_0,\ldots,e_n$.
If a sequence $e_0,\ldots,e_n$ has pattern $p$, then the \emph{pre-reservoir} of this sequence is 
the union of all leaf-bags of~$\Acal$ (viewed as induced subgraphs of $G-v$) whose adhesion set span edges~$e_i$ which contribute to~$p$, plus all the timid edges $e_i$ 
which contribute to~$p$.
The \emph{reservoir} of $e_0,\ldots,e_n$ is obtained from the pre-reservoir of $e_0,\ldots,e_n$ 
by adding the vertex $v$ plus all the edges in~$G$ from $v$ to the pre-reservoir and then 
deleting the capstone of $e_0,\ldots,e_n$.
Note that the reservoir is a subgraph of~$G$.

\begin{figure}[ht]
\centering
\includegraphics[height=9\baselineskip]{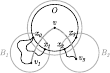}
\caption{The situation of \autoref{u-lemma} for $p=\ty{btx}$ with $\ty{x}:=\ty{b}$}
\label{fig:u-lemma}
\end{figure}

\begin{lem}[Linking Lemma]\label{u-lemma}
Assume \autoref{X}.
Suppose that $O$ has length at least four.
Let $e_1, e_2, e_3, e_4$ be consecutive edges in a cyclic orientation of~$O$ with pattern~$p$.
Denote the endvertices of~$e_1$ by~$x_0$ and~$x_1$ so that~$x_1$ is incident with~$e_2$.
\begin{enumerate}
\item If the first letter of~$p$ is \ty{b}, then there are two internally disjoint paths from~$x_0$ to~$v$ included in the reservoir of~$e_1,\ldots,e_3$.
\item Otherwise, there are two internally disjoint paths from $x_1$~to~$v$ included in the reservoir of $e_2,\ldots, e_4$
avoiding~$x_0$.
\end{enumerate}
\end{lem}

Recall that a \emph{2-fan} from $u$ to $x$ and~$y$ is the union of a $u$--$x$ path with a $u$--$y$ path where the two paths meet precisely in~$u$.

\begin{fact}\label{2fan}
In a 2-connected graph, there exists a 2-fan from $u$ to $x$ and~$y$ for every three vertices $u,x,y$ in the graph.\qed
\end{fact}

\begin{proof}[Proof of the \nameref{u-lemma} (\ref{u-lemma})]
We consider each of the five possible patterns in turn.
Let $G':=G-v$.
Whenever an edge~$e_i$ is bold, we let $B_i:=G'[B_\ell]$ for the unique leaf-bag~$B_\ell$ of~$\Acal$ whose adhesion set consists of the endvertices of~$e_i$ (tacitly assuming that~$i$ is not a node of~$S$).
By \autoref{v has neighbours}, the vertex $v$ has a neighbour in $B_i\sm O$ whenever~$B_i$ exists, and 
we choose such a neighbour~$v_i$ for each eligible~$i$.
We denote the vertex of~$O$ that is incident with both~$e_i$ and~$e_{i+1}$ by~$x_i$.
This is consistent with the naming of~$x_1$ in the statement of the lemma.

(\ty{bb}) By \autoref{2fan}, we find a 2-fan in~$B_1$ from $x_0$ to $v_1$~and~$x_1$.
Since $B_2$ is 2-connected, we find an $x_1$--$v_2$ path in~$B_2$ which avoids~$x_2$.
The subgraph of~$G$ obtained from the union of the 2-fan with the $x_1$--$v_2$ path by adding~$v$ 
and the edges $v v_1$~and~$v v_2$ contains two desired paths.

(\ty{btx}) By \autoref{2fan}, we find a 2-fan in~$B_1$ from $x_0$~to $v_1$~and~$x_1$.
If~\ty{x} is equal to~\ty{b}, then $B_3$ exists, and since $B_3$ is 2-connected, there is an $x_2$--$v_3$ path in 
$B_3$ which avoids~$x_3$.
Then the subgraph of $G$ obtained from the union of the 2-fan with the path by adding the timid 
edge $e_2$ as well as $v$ and the two edges $v v_1$ and $v v_3$ contains two desired paths.
Otherwise \ty{x} is equal to~\ty{t}.
Then the two edges $e_2$ and $e_3$ are timid, hence 3-connectivity implies that $x_2$ is adjacent 
to~$v$.
Then the subgraph of $G$ obtained from the 2-fan by adding the edges $v v_1$, $e_2$ and $v x_2$ 
contains two desired paths.

(\ty{tbb}) Since \ty{bb} is a suffix of \ty{tbb} and since the sought paths are allowed to start in~$x_1$ 
instead of~$x_0$, we may follow the argumentation of~(\ty{bb}).

(\ty{tbtx}) Since \ty{btx} is a suffix of \ty{tbtx} and since the sought paths are allowed to start in $x_1$ 
instead of~$x_0$, we may follow the argumentation of~(\ty{btx}).

(\ty{ttx}) By 3-connectivity, the edge $x_1 v$ must exist in~$G$.
Suppose first that \ty{x} is equal to~\ty{b}.
Then~$B_3$ exists. 
Since $\Acal$ is 2-connected, we find an $x_2$--$v_3$ path in~$B_3$ avoiding~$x_3$.
Then $x_1 v$ is one path and adding both edges $e_2$ and $v v_3$ to the $x_2$--$v_3$ path yields the 
second path.
Otherwise \ty{x} is equal to~\ty{t}.
By 3-connectivity, the edge $x_2 v$ must exist in~$G$.
Hence $x_1 v$ and $x_1 x_2 v$ are two desired paths.
\end{proof}

\begin{lem}\label{find 3 paths}
Assume \autoref{X}.
Suppose that $O$ has length at least four.
Let $\{a_1,a_2\}$ be the adhesion set of a leaf-bag of~$\Acal$ and let~$i\in\{1,2\}$.
If~$a_i$ is not edgy, then either there are three internally disjoint paths from~$a_i$ to~$v$ avoiding~$a_{3-i}$, or~$va_i$ is an 
edge in~$G$.
\end{lem}

\begin{proof}
Without loss of generality we have~$i=2$.
If~$va_2$ is an edge in~$G$ we are done, so let us suppose that~$v$ is not adjacent to~$a_2$.
Let $e_1,e_2,e_3$ and~$e_4$ be the four edges of~$O$ that come after $a_1a_2$ on~$O$ in the 
cyclic orientation of $O$ in which $a_1$ precedes~$a_2$. 
Since~$a_2$ is not edgy and~$v$ is not a neighbour of~$a_2$, 
the edge $e_1$ is bold. 
So the sequence $e_1, e_2, e_3, e_4$ has pattern~\ty{bb} or~\ty{btx}, both of which have length at most three. 
Now we apply the \nameref{u-lemma} (\ref{u-lemma}) to the sequence $e_1, e_2, e_3, e_4$. 
This gives us two internally 
disjoint paths  from 
$a_2$~to~$v$ included in the reservoir of $e_1,\ldots,e_3$. 
By assumption, $O$ has length at least four, so the vertex~$a_1$ is distinct from the endvertices of the edges 
$e_1$ and~$e_2$. Hence, the two internally disjoint paths avoid~$B_\ell-a_2$, where $B_\ell$ is the leaf-bag with adhesion set~$\{a_1,a_2\}$. 
By \autoref{v has neighbours}, we find a third path from $a_2$~to~$v$ 
included in~$G[B_\ell+v]$, completing the proof.
\end{proof}

\begin{lem}\label{find 2 paths}
Assume \autoref{X}.
Suppose that $O$ has length at least four and that $O$ does not have type~\ty{btbt}.
Let $\{a_1,a_2\}$ be the adhesion set of a leaf-bag of~$\Acal$ and let~$i\in\{1,2\}$.
Denote the unique neighbour of~$a_i$ on~$O$ other than~$a_{3-i}$ by~$a_i'$.
If~$a_i$ is edgy, then $a_i a_i'$ is an edge in~$G$ while $a_i v$ is not, and there are two internally 
disjoint paths from $a_i'$ to $v$ in~$G$ that avoid~$B_\ell$.
\end{lem}

\begin{proof}
Without loss of generality, we have~$i=2$.
Let $e_1,e_2,e_3$ and~$e_4$ be the four consecutive edges which come after~$a_1a_2$ in the 
cyclic orientation of~$O$ in which~$a_1$ precedes~$a_2$. 
Since~$a_2$ is edgy and $\Acal$ is 2-connected, $e_1$ is timid.
We apply the \nameref{u-lemma} (\ref{u-lemma}) to the sequence $e_1, e_2, e_3, e_4$. Since the pattern of this sequence starts with~\ty{t} and since $O$ has length at least four, we obtain two internally 
disjoint paths $P,Q$ from~$a_2'$ to~$v$ included in the reservoir of $e_2,\ldots ,e_4$ and avoiding~$a_2$.

If $O$ has length at least five, the vertex $a_1$ is distinct from the endvertices of the edges 
$e_1,e_2$ and~$e_3$. 
In particular, the two paths $P,Q$ avoid the unique leaf-bag~$B_\ell$ of~$\Acal$ with adhesion set~$\{a_1,a_2\}$.

So it remains to consider the case that $O$ has length four. 
The existence of the bag~$B_\ell$ implies that the 
edge~$e_4$ is bold. 
In combination with our assumption that~$O$ does not have the type \ty{btbt}, it follows that the pattern \ty{tbtx} is not possible (indeed, here $\ty{x}=\ty{b}$ since $e_4$ is bold).
Hence the pattern has length at most three.
So $P$ and $Q$ are contained in the reservoir of $e_2,e_3$ and avoid~$a_1,a_2$, by the earlier application of the \nameref{u-lemma} (\ref{u-lemma}) and since the 3-letter pattern starts with \ty{t} for the edge~$e_1$.
In particular, $P$ and $Q$ avoid~$B_\ell$.
This completes the proof.
\end{proof}

\begin{lem}\label{make_nice}
Assume \autoref{X}.
If $O$ has length at least four and does not have type \ty{btbt}, then~$v$ has two neighbours in~$G\sm B_\ell$ for every leaf-bag $B_\ell$ of~$\Acal$.
\end{lem}
\begin{proof}
Let $B_\ell$ be a leaf-bag of~$\Acal$, and let $\{a_1,a_2\}$ denote its adhesion set.
Let $e_1,e_2,e_3$ and~$e_4$ be the four edges on~$O$ which come after~$a_1a_2$ in the cyclic 
orientation of~$O$ in which $a_1$ precedes~$a_2$. 
Let $p$ be the pattern of this sequence.
For each bold~$e_i$, let $B_i$ denote the leaf-bag of~$\Acal$ witnessing that~$e_i$ is bold.

\textbf{Case} $p=\ty{bb}$. 
By \autoref{v has neighbours}, the vertex $v$ has two neighbours, one in $B_1\sm O$ and one in $B_2\sm O$.

\textbf{Case} $p=\ty{btx}$. 
By \autoref{v has neighbours}, the vertex $v$ has one neighbour in $B_1\sm O$.
If the third edge is bold, we find a second neighbour in $B_3\sm O$.
Otherwise, we consider the endvertex that is shared by $e_2$ and~$e_3$. 
By 3-connectivity, this endvertex must be adjacent to~$v$. 
So~$v$ has two 
neighbours outside of~$B_\ell$. 

\textbf{Case} $p=\ty{tbb}$. Since \ty{bb} is a suffix of~\ty{tbb}, we may argue as in the case $p=\ty{bb}$.

\textbf{Case} $p=\ty{tbtx}$. 
We consider two subcases. 
If $O$ has length at least five, then we may argue as in the case $p=\ty{btx}$, since \ty{btx} is a suffix of \ty{tbtx}.
Otherwise $O$ has length four. 
Then the existence of the leaf-bag $B_\ell$ entails that the edge $e_4$
is bold, and so this case is excluded as~$O$ does not have type~$\ty{btbt}$. 

\textbf{Case} $p=\ty{ttx}$.
Here we may argue similarly as in the case $p=\ty{btx}$.
\end{proof}

\begin{lem}\label{is_nontrivial}
Assume \autoref{X}.
If $O$ has length at least four and does not have type \ty{btbt}, then the tri-star of $\Acal$ consists of strong nontrivial tri-separations.
\end{lem}
\begin{proof}
Let $(X,Y)$ be an arbitrary pseudo-reduction in the tri-star of $\Acal$, induced by a leaf~$\ell$ of~$S$ (if no such $(X,Y)$ exists, then the tri-star is empty and we are done).
Let us denote the adhesion set of the leaf-bag~$B_\ell$ by~$\{a_1,a_2\}$.
We claim that very vertex $u\in S(X,Y)$ has degree at least four in~$G$ and has at least two neighbours in both~$X$ and~$Y$.

\textbf{Case~1:} $u=a_i$ for some $i\in\{1,2\}$.
Since $G'[B_\ell]$ is 2-connected, $a_i$ has at least two neighbours in~$B_\ell\se X$.
As $a_i$ lies in~$S(X,Y)$, it is not edgy, so it has at least two neighbours in $V(G)\sm B_\ell\se Y$.
As these neighbours are distinct, $a_i$ has degree at least four in~$G$.

\textbf{Case~2:} $u=v$.
Since $v$ lies in $S(X,Y)$, it has at least two neighbours in~$B_\ell\se X$.
By \autoref{make_nice}, $v$~has two neighbours in $V(G)\sm B_\ell\se Y$.

Therefore, $(X,Y)$ is a strong tri-separation.
It remains to show that $(X,Y)$ is nontrivial. 
Since $G'[B_\ell]$ is 2-connected, it contains a cycle, which is included in~$G[X]$.
To see that $G[Y]$ contains a cycle, by \autoref{trivial} it suffices to show that $|Y\sm X|\ge 2$, 
which follows from $O$ having length at least four. 
\end{proof}

\begin{proof}[Proof of \autoref{ApexGeneral}]
Assume \autoref{X}.
Further suppose that $O$ has length at least four and that $O$ does not have type~\ty{btbt}.
We have to show that the tri-star of~$\Acal$ consists of totally-nested strong nontrivial tri-separations.
By \autoref{is_nontrivial}, the tri-star of~$\Acal$ consists of strong nontrivial tri-separations.
So it remains to show that these are totally nested.

Let $(X,Y)$ be a pseudo-reduction in the tri-star of~$\Acal$.
Let $\ell$ be the leaf of~$S$ which induces~$(X,Y)$, and let $B_\ell$ denote the leaf-bag of~$\Acal$ assigned to~$\ell$.
Let $\{a_1,a_2\}$ denote the adhesion set of~$B_\ell$.
By definition, $X\cap Y$ is a subset of~$\{v,a_1,a_2\}$.

\begin{sublem}\label{sublemvinsep}
If the separator of $(X,Y)$ contains~$v$, then it is externally tri-connected around~$v$.
\end{sublem}

\begin{cproof}
We assume~$v\in S(X,Y)$.
The vertices~$a_1,a_2$ either lie in the separator 
of~$(X,Y)$ or are $v$-free.
If a vertex~$a_i$ is not in~$S(X,Y)$, then~$S(X,Y)$ contains the edge on $O$ that joins~$a_i$ to its neighbour on~$O$ other than $a_{3-i}$.
So if at least one of~$a_1$ and~$a_2$ is not in $S(X,Y)$, then the two internally disjoint $a_1$--$a_2$ paths through~$G'[B_\ell]$ provided by 2-connectedness witness that $S(X,Y)$ is externally tri-connected around~$v$, according to criterion~($=$) or~($\dotminus$).
It remains to consider the case where $S(X,Y)$ contains both $a_1$ and~$a_2$.
Then, to satisfy criterion~(:), we accompany the two $a_1$--$a_2$ paths through~$G'[B_\ell]$ with a third path, internally disjoint from the former two, which we obtain from the $a_1$--$a_2$ path 
$O-a_1 a_2$ by 
replacing torso-edges with detours through their corresponding leaf-bags if necessary. 
\end{cproof}

\begin{sublem}\label{sublem7}
If the separator of~$(X,Y)$ contains an~$a_i$, then it is externally tri-connected around~$a_i$.
\end{sublem}

\begin{cproof}
Suppose that $a_2\in S(X,Y)$, say.
We distinguish two cases.

{\bf Case~1:} the vertex~$a_1$ lies in the separator of~$(X,Y)$ as well.
Then $a_1$ is not edgy and \autoref{find 3 
paths} either yields three 
internally disjoint $a_1$--$v$ paths avoiding $a_{2}$ or that $a_1 v$ is an edge in~$G$. 
So~if~$S(X,Y)$ contains~$v$, it is externally tri-connected around~$a_2$ by criterion~(:).
Otherwise, $v$~is not in $S(X,Y)$,
and the unique neighbour $u$ of $v$ in $B_\ell$ is distinct from $a_1$ and $a_2$ by \autoref{v has neighbours},
so $v$ is $a_2$-free.
If there exist three internally disjoint $a_1$--$v$ paths avoiding~$a_2$, at least two paths also avoid the edge $uv\in S(X,Y)$; or $a_1v$ is an edge; so~$S(X,Y)$ is externally tri-connected around~$a_2$ by criterion~($\dotminus$).
 
{\bf Case~2:} not Case~1.
Then instead of the vertex $a_1$, the edge $a_1 a_1'$ lies in $S(X,Y)$, where $a_1'$ denotes the unique neighbour of $a_1$ in $G$ outside $B_\ell$.
Note that $a_1'\in O$.
By assumption, $O$ has length at least four and does not have type \ty{btbt}.
So by \autoref{find 2 paths},
there are two internally 
disjoint paths from $a_1'$~to~$v$ in~$G$ that avoid~$B_\ell$.
If~$S(X,Y)$ contains~$v$, the two paths witness that $S(X,Y)$ is externally tri-connected around~$a_2$ by criterion~($\dotminus$).
So we may assume that $v$ is not in $S(X,Y)$, so $S(X,Y)$ instead contains the edge $uv$ where $u\in B_\ell\sm\{a_1,a_2\}$ is the unique neighbour of $v$ in $B_\ell$.
Hence $v$ is $a_2$-free.
Since the vertex~$u$ lies in~$B_\ell$, it is avoided by both paths.
As~$v$ is $a_2$-free, the two paths witness that~$S(X,Y)$ is externally tri-connected around~$a_2$ by criterion~($=$).
\end{cproof}\medskip

By \autoref{sublemvinsep} and \autoref{sublem7}, $S(X,Y)$ is externally tri-connected.
Since $(X,Y)$ also is half-connected, \autoref{is_nested} gives that $(X,Y)$ is totally nested.
\end{proof}

\begin{proof}[Proof of \autoref{is_splitting_star}]
We combine \autoref{tiny-cases}, \autoref{is_splitting_star_case3}, \autoref{o4_not_alternate} and \autoref{ApexGeneral}.
\end{proof}

\begin{proof}[Proof of the \nameref{Angry} (\ref{Angry})]
Let us assume for a contradiction that there exists a 3-connected graph~$G$ that fails all three outcomes of \autoref{Angry}, that is: all nontrivial tri-separations of~$G$ are crossed; $G$~is neither a wheel nor a $K_{3,n}$ for any~$n\ge 3$; and~$G$ is not internally 4-connected.
Then $G$ has a nontrivial strong tri-separation $(A,B)$ by \autoref{tetraXtrisep}, which is half-connected by \autoref{conjunction}.
By assumption, the tri-separation $(A,B)$ is crossed by another tri-separation $(C,D)$.
The tri-separation $(C,D)$ is nontrivial by \autoref{trivial_is_nested}.
By \autoref{dasBesteLemma} and using that $G$ is not a wheel such as~$K_4$, all four links have size one, and the centre consists of a single vertex~$v$.
By \autoref{apex_exists}, $G$~has a 2-connected apex-decomposition $\Acal$ with centre~$v$, such that $(A,B)$ and $(C,D)$ interlace the tri-star of~$\Acal$, and such that the central torso-cycle of~$\Acal$ alternates between $S(A,B)-v$ and $S(C,D)-v$.
As~$G$ is not a wheel, $\Acal$ has at least one leaf-bag, and so the tri-star of $\Acal$ is non-empty.
By \autoref{is_splitting_star}, the tri-star of~$\Acal$ consists of totally-nested nontrivial tri-separations, contradicting our assumption that all nontrivial tri-separations of~$G$ are crossed.
\end{proof}

\clearpage
\renewcommand{\thechapter}{2}
\phantomsection
\noindent{\huge{\textbf{Chapter \thechapter \\ \\Decomposing 3-connected graphs}}}
\addcontentsline{toc}{section}{\thechapter\ Decomposing 3-connected graphs}

\setcounter{section}{0}

\section{Overview of this chapter}

In this chapter, we prove the main result of the paper, \autoref{mainIntro}.
The proof of \autoref{mainIntro} offers additional structural insights, which lead us to a refinement of \autoref{mainIntro} that comes in the form of \autoref{univ_3sepr}.

This chapter is organised as follows.
In the next section we introduce the notation we need to then state \autoref{univ_3sepr}.
Like \autoref{mainIntro}, this theorem will have three possible outcomes for the torsos, and we devote a section to the analysis of each possible outcome.

\section{Basics}\label{sec:basicsForDecomp}

\subsection{Generalised wheels}

\begin{figure}[ht]
    \centering
    \includegraphics[height=8\baselineskip]{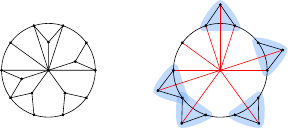}
    \caption{A concrete generalised wheel (left) and its apex-decomposition (right), where leaf-bags are indicated in blue and the centre plus its incident edges are red}
    \label{fig:genWheel}
\end{figure}

The following definitions are supported by \autoref{fig:genWheel}.
A \emph{$Y$-graph} is a 3-star~$K_{1,3}$ and the set of its 3 leaves is referred to as its \emph{attachment set}.
A \emph{concrete generalised wheel} is a triple $(W,O,v)$ where $W$ is a graph obtained from a cycle $O$ and a vertex~$v$ not on~$O$ by doing the following, subject only to the condition that the resulting graph has minimum degree three:
\begin{enumerate}
    \item for every vertex on~$O$, we may (but need not) join it to~$v$, and
    \item for every edge $xy$ on~$O$, we may (but need not) disjointly add a $Y$-graph and identify its attachment set with $\{x,y,v\}$.
\end{enumerate}
We refer to $O$ as the \emph{rim} of this concrete generalised wheel, and we refer to~$v$ as its \emph{centre}.
For convenience, we write $W$ instead of~$(W,O,v)$, and refer to $W$ as a concrete generalised wheel by a slight abuse of notation.
Since concrete generalised wheels have minimum degree three, it is straightforward to show that they are 3-connected.
The \emph{length} of a concrete generalised wheel means the length of its rim.

A \emph{generalised wheel} is a triple $(W,\Acal,v)$ where $W$ is a 3-connected graph, $v$~is a vertex of~$W$, and $\Acal$ is an apex-decomposition of~$W$ with centre~$v$ such that all leaf-bags are triangles.
The \emph{rim} of a generalised wheel is the cycle that is given by the torso of the central bag of the apex-decomposition.
The \emph{length} of a generalised wheel means the length of its rim.
The vertex $v$ is its \emph{centre}.
For convenience, we write $W$ instead of $(W,\Acal,v)$ and refer to $W$ as a generalised wheel by a slight abuse of notation.

\begin{lem}\label{make-concrete}
A graph $G$ is a generalised wheel with centre $v$ and rim~$O$ if and only if it is a concrete generalised wheel with centre $v$ and rim~$O$.
\end{lem}

\begin{proof}
Clearly, every concrete generalised wheel is a generalised wheel with the same rim and centre. 
As generalised wheels $W$ are 3-connected, the centre is adjacent to every vertex that has degree two in $W-v$, which implies that $W$ has the structure of a concrete generalised wheel with the same rim and centre. 
\end{proof}

\subsection{Splitting stars}

Recall that a set $\sigma=\{\,(A_i,B_i): i\in I\,\}$ of (oriented) mixed-separations of~$G$ is a \emph{star} with \emph{leaves}~$A_i$ if $(A_i,B_i)< (B_j,A_j)$ for all distinct indices~$i,j\in I$.
We have seen in \autoref{starEG} that these stars naturally correspond to star-decompositions if they consist of separations only.

Let $S$ be a set of mixed-separations of~$G$.
A star $\sigma=\{\,(A_i,B_i): i\in I\,\}\se S$ with leaves~$A_i$ is \emph{splitting} if for every $(C,D)\in S$ there is $i\in I$ with either $(C,D)\le (A_i,B_i)$ or $(D,C)\le (A_i,B_i)$.

\begin{eg}
Let $(T,\Vcal)$ be a tree-decomposition of~$G$, and let $S$ denote the set of induced separations of~$(T,\Vcal)$.
For every node $t\in T$, let $\sigma_t$ denote the star of separations induced by the edges of~$T$ incident with~$t$ and directed to~$t$.
The splitting stars of~$S$ are precisely the stars $\sigma_t$ with~$t\in T$.
\end{eg}

\begin{lem}\label{splitVsInterlace}
Let $N$ be a nested set of mixed-separations of a graph~$G$, and let $\sigma\se N$ be a star.
Then the following assertions are equivalent:
\begin{enumerate}
    \item $\sigma$ is a splitting star of~$N$;
    \item no element of~$N$ interlaces~$\sigma$.
\end{enumerate}
\end{lem}
\begin{proof}
(1)$\to$(2). 
Let $\sigma$ be a splitting star of~$N$, and assume for a contradiction that $(C,D)\in N$ interlaces~$\sigma$.
Then there is $(A,B)\in\sigma$ such that $(C,D)\le (A,B)$ or $(D,C)\le (A,B)$. But we also have $(A,B)<(C,D)$ or $(A,B)<(D,C)$ since $(C,D)$ interlaces~$\sigma$.
In two cases we obtain immediate contradictions, and in the other two cases we obtain $C\se D$ or $D\se C$ which contradicts the definition of a separation.

(2)$\to$(1). Assume that no element of~$N$ interlaces~$\sigma$; we show that $\sigma$ is a splitting star of~$N$.
Let $(C,D)\in N$, and assume for a contradiction that there is no $(A,B)\in\sigma$ such that $(C,D)\le (A,B)$ or $(D,C)\le (A,B)$.
Then, since $N$ is nested, for every $(A,B)\in\sigma$ we have $(A,B)< (C,D)$ or $(A,B)<(D,C)$, so $(C,D)$ interlaces~$\sigma$.
\end{proof}

\begin{lem}\label{splittingStarsDist}
Let $M$ be a nested set of mixed-separations of a graph~$G$, and let $\sigma$ and $\tau$ be two distinct splitting stars of~$M$.
Then there are $(A,B)\in\sigma$ and $(C,D)\in\tau$ such that $(B,A)\le (C,D)$.
\end{lem}
\begin{proof}
Let $(X,Y)\in\sigma$ be arbitrary.
Since $\tau$ is a splitting star, there is $(C,D)\in\tau$ such that $(X,Y)\le (C,D)$ or $(Y,X)\le (C,D)$.
In the latter case, we put $(A,B):=(X,Y)$ and are done.
In the former case, we use that $\sigma$ is a splitting star to find $(A,B)\in\sigma$ such that $(C,D)\le (A,B)$ or $(D,C)\le (A,B)$.
It suffices to derive a contradiction from $(C,D)\le (A,B)$.
Indeed, then $(X,Y)\le (C,D)\le (A,B)\le (Y,X)$ gives $X\se Y$, contradicting that $(X,Y)$ is a mixed-separation.
\end{proof}

\begin{lem}\label{unique_split}
Given a nested set of mixed-separations $M$ of a graph~$G$, a mixed-separation of~$G$ interlaces at most one splitting star of~$M$.
\end{lem}
\begin{proof}
Assume for a contradiction that some mixed-separation $(A,B)$ of $G$ interlaces two distinct splitting stars $\sigma_1$ and $\sigma_2$ of~$M$.
By \autoref{splittingStarsDist}, there exist $(C_1,D_1)\in\sigma_1$ and $(C_2,D_2)\in\sigma_2$ such that $(D_1,C_1)\le (C_2,D_2)$.
Since $(A,B)$ interlaces $\sigma_1$ and $\sigma_2$, and since $(A,B)\not < (A,B)$ nor $(B,A)\not < (B,A)$, we either have
\[
    (A,B)<(D_1,C_1)\le (C_2,D_2)<(B,A)\quad\text{or}\quad(B,A)<(D_1,C_1)\le (C_2,D_2)<(A,B).
\]
Then $A\se B$ or $B\se A$, contradicting that $(A,B)$ is a mixed-separation.
\end{proof}

\subsection{Torsos}\label{subsec:Torsos}

A~\emph{mixed-separation\plus } of a graph~$G$ is a pair $(A,B)$ such that $A\cup B=V(G)$ and no two edges in $E(A\sm B,B\sm A)$ share endvertices.
We stress that we allow $A\sm B$ and $B\sm A$ to be empty.
All the usual concepts for mixed-separations extend to mixed-separations\plus\ in the obvious way.

\begin{eg}
All nontrivial mixed 3-separations of a 3-connected graph are mixed 3-separations\plus\ by \autoref{independentEdges}.
Pairs $(A,V(G))$ for $A\se V(G)$ are separations\plus\ but not separations.
\end{eg}

Let $\sigma=\{\,(A_i,B_i): i\in I\,\}$ be a star of mixed-separations\plus\ of a graph~$G$, with leaf-sides~$A_i$.
The \emph{bag} of~$\sigma$ is the intersection $\bigcap_{i\in I}B_i$ of all non-leaf sides~$B_i$.
We follow the convention that the bag of the empty star is equal to the vertex-set of~$G$.

If all $(A_i,B_i)$ are separations\plus , then the \emph{torso} of $\sigma$ is the graph obtained from $G[\beta]$, where $\beta$ is the bag of~$\sigma$, by making every separator $A_i\cap B_i$ into a clique (by adding all possible edges inside $A_i\cap B_i$ for all $i\in I$).
In general, however, there are (at least) two ways how the notion of a torso can be generalised to stars of mixed-separations\plus .
Here we present two ways, supported by \autoref{fig:DifferentTorsos}.

The \emph{compressed torso} of~$\sigma$ is the graph that is obtained from~$G$ by contracting all edges in separators of elements of~$\sigma$, reducing parallel edges to simple ones, and then taking the torso as defined above.
The torsos in \autoref{mainIntro} that were mentioned in the introduction are the compressed torsos.

The \emph{expanded torso} of~$\sigma$ is the graph that is obtained from~$G$ as follows.
We obtain $(A_i',B_i')$ from $(A_i,B_i)\in \sigma$ by letting $A_i':=A_i$ and we obtain $B_i'$ from $B_i$ by adding all endvertices of edges in the separator of $(A_i,B_i)$.
Then $(A_i',B_i')$ is a separation\plus\ with the same order as~$(A_i,B_i)$.
We take the torso of the star $\{\,(A_i',B_i'):i\in I\,\}$ of separations\plus\ as the expanded torso of~$\sigma$.

\begin{figure}[ht]
    \centering
    \includegraphics[height=7\baselineskip]{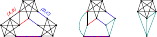}
    \caption{Left: the star $\sigma=\{\,(A,B),(C,D)\,\}$. Middle: the expanded torso of~$\sigma$. Right: the compressed torso of~$\sigma$.}
    \label{fig:DifferentTorsos}
\end{figure}

Note that the compressed torso can be obtained from the expanded torso by contracting all edges in the separators~$S(A_i,B_i)$.
If all $(A_i,B_i)$ are separations\plus , then the compressed torso and the expanded torso coincide.
If $N$ is a nested set of mixed-separations of~$G$, then the (compressed/expanded) torsos of~$N$ are the (compressed/expanded) torsos of the splitting stars of~$N$.
We remark that if $N$ is the set of induced separations of a tree-decomposition $(T,\Vcal)$ of~$G$, then the torsos of~$N$ are precisely the torsos of~$(T,\Vcal)$.

\begin{lem}\label{torsoMinors}
Let $N$ be a nonempty nested set of nontrivial mixed 3-separations of a 3-connected graph~$G$.
Then all compressed torsos and expanded torsos of~$N$ are minors of~$G$.
\end{lem}
\begin{proof}
By nontriviality, the leaves of the splitting stars induce subgraphs of~$G$ that include cycles, and using Menger's theorem we can contract these cycles onto the respective triangles in the torsos.
\end{proof}

\subsection{Statement of the main theorem}
Let $\sigma$ be a splitting star of a nested set $N$ of tri-separations of~$G$.
Let $(A,B)$ be a strong nontrivial tri-separation of~$G$.
We say that $(A,B)$ interlaces $\sigma$ \emph{lightly} if $(A,B)$ interlaces~$\sigma$ and both $G[A\sm B]$ and $G[B\sm A]$ have at least two components.
We say that $(A,B)$ interlaces $\sigma$ \emph{heavily} if $(A,B)$ interlaces $\sigma$ and $(A,B)$ is half-connected.
So if $(A,B)$ interlaces~$\sigma$, then it does so either lightly or heavily.
We stress that  tri-separations that fail to be strong or nontrivial interlace $\sigma$ neither lightly nor heavily by definition.

A \emph{thickened $K_{3,m}$} is obtained from the bipartite graph $K_{3,m}$ by making a bipartition class of size three complete; that is, we add the three edges of a triangle to that set. We allow the degenerated case of a triangle as a thickened $K_{3,0}$.

Recall that a graph $G$ is quasi 4-connected if $G$ is 3-connected, $G$ has $>4$ vertices, and every 3-separation of $G$ has a side of size at most four.

\begin{thm}\label{univ_3sepr}
Let $G$ be a 3-connected graph and let $N$ denote its set of totally-nested nontrivial tri-separations.
Each splitting star $\sigma$ of $N$ has the following structure:
\begin{enumerate}[label={\textnormal{(\roman*)}}]
    \item\label{K3Case} if $\sigma$ is interlaced lightly, then its compressed torso is a thickened $K_{3,m}$ or $G=K_{3,m}$ for some $m\geq 0$;
    \item\label{genWheelCase} if $\sigma$ is interlaced heavily, then its compressed torso is a wheel, and its expanded torso is a generalised wheel;
    \item\label{4tangleCase} if $\sigma$ is not interlaced by a strong nontrivial tri-separation, then its compressed torso is quasi 4-connected or a~$K_4$ or~$K_3$.
\end{enumerate}
Moreover, all expanded torsos and compressed torsos of~$N$ are minors of~$G$.
\end{thm}

\begin{figure}[ht]
    \centering
    \includegraphics[height=8\baselineskip]{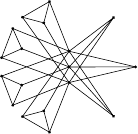}
    \caption{A graph with a thickened~$K_{3,3}$ as torso; see \autoref{eg:K3m}}
    \label{fig:my_label}
\end{figure}
\begin{rem}\label{eg:K3m}
For every integer $m\ge 0$, there exist $G$ and~$\sigma$ as in the statement of \autoref{univ_3sepr} such that \ref{K3Case} holds and the compressed torso of~$\sigma$ is a thickened~$K_{3,m}$.
Indeed, let $m$ be given.
Let $X$ and $Y$ be disjoint vertex sets of size $m$ and three, respectively.
We let $G$ be the graph obtained from the complete bipartite graph on $(X,Y)$ by disjointly adding four triangles $\Delta_1,\ldots,\Delta_4$ and joining each triangle $\Delta_i$ to the three vertices in~$Y$ by a matching of size three.
Then $\sigma:=\{\,(\Delta_i,V(G\sm\Delta_i)):i\in [4]\,\}$ is a splitting star of~$N$.
The splitting star~$\sigma$ is lightly interlaced by the tri-separation $(\Delta_1\cup\Delta_2\cup Y,V(G\sm(\Delta_1\cup\Delta_2))\,)$.
The compressed torso of~$\sigma$ is obtained from~$G$ by contracting all edges in the matchings between~$Y$ and the triangles~$\Delta_i$, so it is a thickened~$K_{3,m}$.
\end{rem}

   \begin{figure} [htpb]   
\begin{center}
   	\includegraphics[height=8\baselineskip]{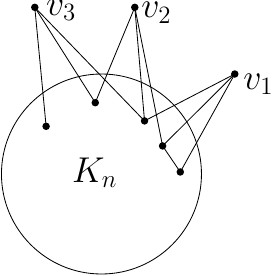}
   	\caption{This graph is obtained from a $K_n$ with $n=100$ by first attaching the three vertices $v_1$, $v_2$ and $v_3$ of degree three as illustrated, and then deleting all edges with both ends in the neighbourhood of a~$v_i$, except one edge which lies as in the figure.}\label{fig:ord}
\end{center}
   \end{figure}

\begin{rem}\label{eg100}
In \ref{4tangleCase}, we cannot replace \lq quasi 4-connected\rq\ by \lq internally 4-connected\rq.
Indeed, let $G$ be the graph depicted in \autoref{fig:ord}.
The graph $G$ has precisely one strong nontrivial tri-separation (up to flipping sides), which has the form $(A,B)=(v_1+x+y,V(G)-v_1)$ for $xy$ the unique edge with both ends in the neighbourhood of~$v_1$.
Hence $\{\,(A,B)\,\}$ is a splitting star of~$N$. 
The compressed torso of $\sigma$ is obtained from $G$ by making the neighbourhood of the vertex $v_1$ complete and then deleting $v_1$. Thus this compressed torso has a nontrivial tri-separation, similar to the tri-separation $(A,B)$ with \lq $v_2$\rq\ taking the role of~\lq $v_1$\rq.
Hence the compressed torso is quasi 4-connected but not internally 4-connected, compare \autoref{tetraXtrisep}.
\end{rem}

\begin{figure}[ht]
    \centering
    \includegraphics[height=8\baselineskip]{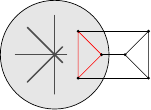}
    \caption{An expanded torso that is not quasi 4-connected; red edges are missing}
    \label{fig:KNwithPrism}
\end{figure}

\begin{rem}
In \ref{4tangleCase}, we cannot replace \lq compressed torso\rq\ by \lq expanded torso\rq.
Indeed, let $G$ be the graph obtained from $K_{10}$ by picking a triangle $\Delta$ within and attaching a new triangle~$\Delta'$ to~$\Delta$ via a matching, and then deleting the edges of~$\Delta$; see \autoref{fig:KNwithPrism}. 
The matching is a 3-edge cut of~$G$ and determines a splitting star $\{\,(\Delta',G\sm\Delta')\,\}$ of~$N$.
Hence $G$ is an expanded torso of~$N$, but the splitting star is not interlaced by a tri-separation and $G$ is not quasi 4-connected.
\end{rem}

\section[Proof of (i)]{Proof of \ref{K3Case}}

Our strategy to prove \ref{K3Case} is to construct for every tri-separation $(A,B)$ that is not half-connected a splitting star $\sigma$ interlaced by $(A,B)$. The construction is explicit and allows us to deduce that this splitting star has the structure for~\ref{K3Case}. Then we apply \autoref{unique_split} to deduce that every splitting star interlaced by $(A,B)$ must be equal to $\sigma$, completing the proof. The details are as follows. 

Let $G$ be a 3-connected graph.
Let $U$ be a set of three vertices of~$G$.
The \emph{star of 3-separations induced by~$U$} is the star
\[
    \{\,(A_K,B_K): K\text{ is a component of }G\sm U\,\}\quad\text{with leaves}\quad A_K:=V(K)\cup U\text{, where }B_K:=V(G\sm K).
\]
Suppose now that $G\sm U$ has at least four components.
By \autoref{K3n lemma}, for every 3-separation $(A_K,B_K)$ in this star the reduction $(A_K',B_K')$, which in particular is a tri-separation, satisfies $(A_K',B_K')\le (A_K,B_K)$.
The \emph{star of tri-separations induced by~$U$} is the star that consists of all the reductions of the 3-separations in the star of 3-separations induced by~$U$.

\begin{lem}\label{obs321}
Let $(A,B)$ be a mixed 3-separation of a 3-connected graph~$G$, and let $(A',B')$ be the reduction of~$(A,B)$.
Then no tri-separation $(C,D)$ of~$G$ satisfies $(A',B')<(C,D)\le (A,B)$.
\end{lem}
\begin{proof}
Suppose for a contradiction that $(C,D)$ is a tri-separation of~$G$ with $(A',B')< (C,D)\le (A,B)$.
Since $(A',B')$ is the reduction of $(A,B)$, we have $B'\se B$.
From $(A',B')\le (C,D)\le (A,B)$ we obtain $B'\supseteq D\supseteq B$, so $B'=D=B$.
Since $(A',B')<(C,D)$ and $B'=D$, the inclusion $A'\se C$ must be proper.
As $C\se A$, this means that some vertex $v\in A\cap B$ lies in~$C$ but has been removed from~$A$ to obtain~$A'$.
So $v$ has only one neighbour in~$A$.
But then $v$ has only one neighbour in $C\se A$, contradicting that $(C,D)$ is a tri-separation with $v\in C\cap D$.
\end{proof}

\begin{lem}\label{ThreeCompsTriBlock}
Let $G$ be a 3-connected graph and $U\se V(G)$ of size three such that $G\sm U$ has at least three components.
Then $U\se A$ or $U\se B$ for every nontrivial tri-separation $(A,B)$ of~$G$.
\end{lem}
\begin{proof}
Let $U=:\{u,v,w\}$.
Suppose for a contradiction that there is a nontrivial tri-separation $(A,B)$ of $G$ with $u\in A\sm B$ and $v\in B\sm A$.
Since $G\sm U$ has at least three components, and since every component $K$ has neighbourhood equal to $U$ by 3-connectivity, there exist three independent $u$--$v$ paths $P_K$ in $G$ with all internal vertices included in~$K$.
Hence $S(A,B)$ must contain an edge or an internal vertex from each path~$P_K$.
Therefore $w$ cannot be contained in $S(A,B)$ and $G\sm U$ has exactly three components.
So $w\in B\sm A$, say.
Hence $u$ is a cutvertex of $G[A]$.
As $(A,B)$ is nontrivial, there is a cycle in $G[A]$.
Let $K^*$ denote the component of $G\sm U$ that contains a vertex of this cycle.
The cycle has a vertex $x$, say, that is distinct from $u$ and not in $S(A,B)$ (since otherwise $S(A,B)$ would intersect $K^*$ in two vertices, a contradiction).
But then the element of $S(A,B)$ on $P_{K^*}$ together with $u$ separate $x$ from the other components of $G\sm U$, contradicting that $G$ is 3-connected.
\end{proof}

\begin{lem}\label{stuff_into_component}
Let $G$ be a 3-connected graph and $U\se V(G)$ of size three such that $G\sm U$ has at least three components.
Then for every half-connected nontrivial tri-separation $(C,D)$ of~$G$ there is a component $K$ of $G\sm U$ such that $(C,D)\le (A_K,B_K)$ or $(D,C)\le (A_K,B_K)$.
\end{lem}
\begin{proof}
By \autoref{ThreeCompsTriBlock} we may assume that $U$ avoids $C\sm D$, say.
Hence $U\se D$.
If $S(C,D)=U$, then $(C,D)$ being half-connected implies that $(C,D)$ or $(D,C)$ is equal to a 3-separation $(A_K,B_K)$ for some component $K$ of $G\sm U$, and we are done.
So assume that $S(C,D)\neq U$.
Let $K'$ be an arbitrary component of~$G[C\sm D]$.
Since $U$ is included in~$D$, the component $K'$ avoids~$U$.
So $K'$ is included in a unique component $K$ of $G\sm U$.
We claim that $(C,D)\le (A_K,B_K)$.

First, we show $C\se A_K$.
It suffices to show $G[C\sm D]\se K$ since this implies 
\[
    C\se (C\sm D)\cup N(C\sm D)\se V(K)\cup N(K) =A_K.
\]
If $S(C,D)$ contains an edge, then by 3-connectivity every component of $G[C\sm D]$ must contain the end of this edge in~$C$, and so $G[C\sm D]=K'\se K$ as desired.
Hence we may assume that $S(C,D)$ consists of three vertices, and since $S(C,D)\neq U$ there is a vertex $v\in S(C,D)\sm U$.
By 3-connectivity, every component of $G[C\sm D]$ has neighbourhood equal to $C\cap D$, and so $K'\se K$ with $A_K\cap B_K=U$ implies $v\in A_K\sm B_K=V(K)$.
As all components of $G[C\sm D]$ avoid~$U$ but have~$v$ in their neighbourhoods, they must all be included in the component $K$ of $G\sm U$ that contains~$v$, so $G[C\sm D]\se K$ as desired.

For $D\supseteq B_K$, we use $C\se A_K$ to get $B_K\sm A_K\se D\sm C$, and recall that $B_K\cap A_K=U\se D$.
\end{proof}

\begin{lem}\label{totallyNestedHalfConnected}
Every totally-nested tri-separation of a 3-connected graph is half-connected.
\end{lem}
\begin{proof}
We show the contrapositive.
Let $(A,B)$ be non-half-connected tri-separation of a 3-connected graph.
Then the separator of $(A,B)$ consists of vertices only.
Pick arbitrary components $\alpha$ and $\beta$ of $G[A\sm B]$ and $G[B\sm A]$, respectively.
Let 
\[
    C:=(A\cap B)\cup V(\alpha)\cup V(\beta) \quad\text{and}\quad D:=V(G\sm(\alpha\cup\beta))
\]
Since $G$ is 3-connected, every component of $G-(A\cap B)$ has neighbourhood equal to $A\cap B$.
Hence $(C,D)$ is a tri-separation of~$G$.
It is straightforward to check that $(C,D)$ crosses~$(A,B)$.
\end{proof}

\begin{lem}\label{4insepUdefinesSplittingStar}
Let $G$ be a 3-connected graph and $U\se V(G)$ of size three such that $G\sm U$ has at least three components.
Let $\sigma'$ denote the star of tri-separations induced by~$U$, and let $\sigma\se\sigma'$ consist of its nontrivial elements.
Then $\sigma$ is a splitting star of the set of all totally-nested nontrivial tri-separations of~$G$.
\end{lem}
\begin{proof}
The elements of $\sigma$ are totally-nested nontrivial tri-separations of~$G$ by \autoref{K3n lemma}.

Suppose for a contradiction that $\sigma$ is not splitting.
Then some totally-nested nontrivial tri-separation $(C,D)$ of~$G$ interlaces~$\sigma$ by \autoref{splitVsInterlace}.
Since $(C,D)$ is totally nested, it is half-connected by \autoref{totallyNestedHalfConnected}.
By \autoref{stuff_into_component} there is a component $K$ of $G\sm U$ such that $(C,D)\leq (A_K,B_K)$, say.
Let $(A_K',B_K')\in\sigma'$ denote the reduction of~$(A_K,B_K)$.

If $(A_K',B_K')$ is nontrivial, it lies in~$\sigma$, so $(A_K',B_K')<(C,D)\le (A_K,B_K)$ as $(C,D)$ interlaces~$\sigma$.
This contradicts \autoref{obs321}.
Hence $(A_K',B_K')$ is trivial; so $A_K'$ and $B_K'$ are the sides of an atomic 3-cut by \autoref{trivial}.
The only possibility here is that~$|A_K'|=1$, since $B_K'$ includes~$U$.
But then $G[A_K]$ is a~$K_{1,3}$, so $G[C]$ contains no cycle and $(C,D)$ is trivial, a contradiction.
\end{proof}

\begin{proof}[Proof of \autoref{univ_3sepr}~\textnormal{\ref{K3Case}}.]
Let $\sigma$ be a splitting star of $N$ that is interlaced lightly by a tri-separation $(A,B)$ of $G$. 
Then $U=A\cap B$ is a 3-separator of $G$ such that $G\sm  U$ has at least four components. 
Let $\bar \sigma$ be the star of tri-separations induced by $U$, and note that $(A,B)$ interlaces~$\bar \sigma$ as well.
Let $\sigma'$ consist of the nontrivial tri-separations in~$\bar \sigma$.
By \autoref{4insepUdefinesSplittingStar}, $\sigma'$ is a splitting star of $N$. 
As $(A,B)$ interlaces the splitting stars $\sigma$ and $\sigma'$ of~$N$, these splitting stars need to be equal by \autoref{unique_split}.

If $G=K_{3,m}$ for some $m\ge 4$, then we are done.
So we may assume that the graph $G[U]$ has an edge or $G\sm U$ has a component of size at least two. 
Thus $\sigma'$ is non-empty by \autoref{K3n lemma}.
The compressed torso of $\sigma'$ can be obtained from $G$ by first removing every component $K$ of $G\sm U$ with $|K|\ge 2$, then removing every component $K$ of $G\sm U$ with $|K|=1$ if $G[U]$ has an edge, and finally making $U$ into a clique (for the latter we need that~$\sigma'$ is nonempty).
Hence the compressed torso of~$\sigma=\sigma'$ is a thickened $K_{3,m}$ with $m\geq 0$.
\end{proof}

\section[Tools to prove (ii) and (iii)]{Tools to prove \ref{genWheelCase} and \ref{4tangleCase}}
In this short section we prove a few lemmas that we will use in the proofs of \ref{genWheelCase} and \ref{4tangleCase}.

\begin{dfn}[Almost interlacing]
We say that a mixed-separation\plus\ $(C,D)$ of a graph~$G$ \emph{almost interlaces} a star $\sigma$ of mixed-separations\plus\ of~$G$ if $(A,B)\leq (C,D)$ or  $(A,B)\leq (D,C)$ for all $(A,B)\in \sigma$.
\end{dfn}
The notion of `almost interlaces' is more general than the notion of `interlaces' in two ways: on the one hand, we consider mixed-separations\plus , and on the other hand we no longer require that $(A,B)$ and its inverse are not in~$\sigma$.

\begin{lem}\label{v17}
Assume \autoref{X}.
If a tri-separation $(C,D)$ of~$G$ almost interlaces the tri-star $\sigma$ of~$\Acal$, then $S(C,D)$ contains $v$ or an edge incident with $v$.
\end{lem}

\begin{proof}
Since $(C,D)$ almost interlaces~$\sigma$, the elements of the separator $S(C,D)$ are vertices or edges of $O+v$ or edges incident with~$v$.
Suppose for a contradiction that $S(C,D)$ contains neither $v$ nor an edge incident to~$v$.
Then $S(C,D)\se O$.
Every component of $G\sm (O+v)$ contains a neighbour of~$v$, since $G$ is 3-connected.
Every vertex on~$O$ that is not a neighbour of a component of $G\sm (O+v)$ is incident with two timid edges on~$O$, and hence is adjacent to~$v$ by 3-connectivity.
Hence $G\sm S(C,D)$ is connected, a contradiction.
\end{proof}

\begin{lem}\label{x13}
Let $G$ be a 3-connected graph.
Let $(A,B)$ be a mixed 3-separation of~$G$, and let $(A',B')$ be a strengthening of~$(A,B)$.
Then for every strong tri-separation $(C,D)$ of~$G$ with $(C,D)\leq (A,B)$ we also have $(C,D)\leq (A',B')$.
\end{lem}

\begin{proof}
Clearly $B'\se B\se D$.
Let $v$ be  a vertex in $A\cap B$. Assume that $v$ is in $C$. It remains to show that $v$ is in~$A'$.
Since $B\se D$, we have that $v\in C\cap D$. 
As $(C,D)$ is a strong tri-separation, $v$ has degree four in $G$ and two neighbours in $C$. 
Hence the set $C$ witnesses that in the construction of the strengthening $(A',B')$ no vertex of $C$ can be deleted from~$A$.
So $v\in A'$ as desired.
\end{proof}

\begin{cor}\label{capture_star2}
Let $\sigma$ be a star of strong tri-separations of a 3-connected graph~$G$.
If a mixed 3-separation $(A,B)$ of~$G$ almost interlaces~$\sigma$, then every strengthening of $(A,B)$ almost interlaces~$\sigma$.
\qed
\end{cor}

Similar to \autoref{x13}, we prove the following (differences are underlined).

\begin{lem}\label{x12}
Let $G$ be a 3-connected graph.
Let $(A,B)$ be a mixed 3-separation of~$G$, and let $(A',B')$ be the \underline{reduction} of~$(A,B)$.
Then for every \underline{tri-separation}~$(C,D)$ of~$G$ with $(C,D)\leq (A,B)$ we also have $(C,D)\leq (A',B')$.\qed
\end{lem}

\begin{cor}\label{capture_star}
Let $\sigma$ be a star of tri-separations of a 3-connected graph~$G$.
If a mixed 3-separation $(A,B)$ of~$G$ almost interlaces~$\sigma$, then the reduction of $(A,B)$ almost interlaces~$\sigma$ as well.
\qed
\end{cor}

\section[Proof of (ii)]{Proof of \ref{genWheelCase}}

A key step in the proof of  \ref{genWheelCase} is to show that the tri-star of the apex-decomposition from \autoref{X} is splitting, see \autoref{is_splitting} below.
Then we finish the proof similarly to the proof of \ref{K3Case}. We start preparing to prove \autoref{is_splitting}.

\begin{lem}\label{bttExtra}
Assume \autoref{X} with crossing tri-separations.
If $O$ has type \ty{btt}, then the tri-star of~$\Acal$ is a splitting star of the set of all totally-nested nontrivial tri-separations of~$G$.
\end{lem}
\begin{proof}
By \autoref{is_splitting_star_case3}, the tri-star of~$\Acal$ consists of totally-nested nontrivial tri-separations, and 
$v$ is adjacent to all three vertices of~$O$. It remains to show that the tri-star of $\Acal$ is splitting. 
Let $x_1,x_2,x_3$ denote the vertices of~$O$ so that the edge $x_2 x_3$ is bold.
Let $B_\ell$ denote the unique leaf-bag of~$\Acal$; so $\{x_2,x_3\}$ is the adhesion set of~$B_\ell$.
Since all~$x_i$ are neighbours of~$v$, the pseudo-reduction $(C,D)$ induced by~$\ell$ is given by $C:=B_\ell+v$ and $D:=\{x_1,x_2,x_3,v\}$.

Let $(U,W)$ be  a nontrivial tri-separation of $G$ with $(C,D)<(U,W)$. 
We have to show that $(U,W)$ is crossed by a tri-separation.

\begin{sublem}
$(U,W)=(C,D-y)$ for some $y\in \{x_2,x_3,v\}$.   
\end{sublem}
\begin{cproof}
    Since $D\sm C=\{x_1\}$ and $W\sm U$ is nonempty, we deduce from $(C,D)\leq (U,W)$ that  $x_1\in W\sm U$ and $C=U$. Thus since $(C,D)<(U,W)$, the side $W$ is a proper subset of $D=\{x_1,x_2,x_3,v\}$. As $G[W]$ contains a cycle, it has exactly three vertices.
\end{cproof}\medskip

It is straightforward to check that $(C,D-x_i)$ is a tri-separation for both $i\in \{2,3\}$, and that $(U,W)=(C,D-y)$ is crossed by $(C,D-x_i)$ for every $i\in \{2,3\}$ with $x_i\neq y$.
\end{proof}

\begin{dfn}[Red]
Assume~\autoref{X}.
A vertex of $O$ is \emph{red} if it is adjacent to $v$ or incident with two bold edges of~$O$. 
\end{dfn}

\begin{eg}\label{eg-red}
    Vertices incident with two timid edges of~$O$ are red: since $G$ is 3-connected, they need to have a third neighbour in~$G$, and this can only be~$v$.
\end{eg}
A~\emph{mixed 2-separat{\"o}r}\footnote{It is a technical variant of a mixed 2-separator, hence the similar name.} of $O$ is a mixed 2-separator of $O$ or else it consists of the two endvertices of a bold edge of~$O$. 
It is \emph{red} if all edges in it are timid and all vertices in it are red.

\begin{rem}(Motivation)
    In what follows, we offer a way to understand the tri-separations interlacing the tri-star of $\Acal$ in \autoref{X} via red mixed 2-separat{\"o}rs. 
    They have smaller order and hence are easier to analyse. 
\end{rem}

Given a mixed 2-separat{\"o}r~$X$, we denote by $O_X$ the topological space obtained from the geometric realisation of $O$  (which is homeomorphic to $\mathbb{S}^1$) by removing all vertices of $X$ and all interior points of edges in~$X$. We refer to the two connected components of $O_X$ as the \emph{intervals} of~$O_X$.

In the following, when $(C,D)$ is a tri-separation, we write $S(C,D)\cap O$ as an abbreviation of $S(C,D)\cap (V(O)\cup E(O))$.

\begin{lem}\label{is_red}
Assume \autoref{X}.
    If a strong tri-separation $(C,D)$ of~$G$ almost interlaces the tri-star $\sigma$ of~$\Acal$, then $S(C,D)\cap O$ is a red mixed 2-separat{\"o}r. 
\end{lem}

\begin{proof}
Since $(C,D)$ almost interlaces~$\sigma$, the elements of the separator $S(C,D)$ are vertices or edges of $O+v$ or edges incident with~$v$. 
By \autoref{v17}, $S(C,D)$ contains $v$ or an edge incident with $v$.
By \autoref{independentEdges}, no two edges in $S(C,D)$ share ends.
Thus, $X:=S(C,D)\cap O$ has size two. 

\begin{sublem}\label{nice123}
    Each interval of $O_X$ contains (the interior points of) a bold edge or a vertex.
\end{sublem}

\begin{cproof}
Suppose not for a contradiction. 
Then one of the intervals of $O_X$ is equal to the set of interior points of a timid edge $e=xy$. 
So one of the sides of $(C,D)$, say $D$, contains all vertices of~$O$. 
So $C$ intersects $V(O)$ precisely in the vertices $x$ and~$y$. 
As the leaf-bags of $\Acal$ are 2-connected and $e$ is timid, the graph $G-v-x-y$ is connected; hence $C$ contains no other vertex of $G-v$. 
Since $C\sm D$ is nonempty, it must contain a vertex and the only possibility is that $C\sm D=\{v\}$. 
By \autoref{independentEdges} the vertex $v$ has at most one neighbour in $D\sm C= D-x-y$. 
So by \autoref{v has neighbours}, at most one edge of $O$ can be bold. 
Thus one of the vertices $x$ and $y$, say $x$, is incident with two timid edges. 
So $x$ has degree at most three. 
As $x\in C\cap D$, this contradicts the assumption that $(C,D)$ is strong. 
\end{cproof}\medskip

By \autoref{nice123}, $X$ is a mixed 2-separat{\"o}r. It remains to show that $X$ is red.
Every edge in  $X$ is timid, so let $x$ be a vertex in~$X$.
Since we are done otherwise, assume that $x$ is not adjacent to $v$.
If $x$ is not adjacent to a vertex $y$ in~$X$, then the fact that it has at least two neighbours in the sides $C$ and $D$ implies that both its incident edges on $O$ must be bold. 
So assume that $x$ has a neighbour $y$ in~$X$; that is, $X=\{x,y\}$. 
The only way this is possible is that $xy=:e$ is an edge of~$O$.
By \autoref{nice123}, the edge $e$ must be bold. Let $f$ be the edge of~$O$ incident with $x$ aside from~$e$. 

Suppose for a contradiction that $f$ is timid. Then as $x$ is not adjacent to~$v$, the edge $f$ is in the separator of the pseudo-reduction corresponding to~$e$. 
Denote this pseudo-reduction by $(E,F)$ with leaf-side~$E$. 
We have shown that the vertex $x$ is in $C\cap D$ and in $E\sm F$. 
As $(C,D)$ almost interlaces, we have that $C\cap D\se F$, a contradiction. 
So both edges incident with $x$ are bold. Hence $x$ is red. 
\end{proof}

\begin{lem}\label{red_lift}
Assume \autoref{X}.
    For every red mixed 2-separat{\"o}r~$X$ of~$G$, there is a tri-separation $(C,D)$ of $G$ that almost interlaces the tri-star $\sigma$ of~$\Acal$ and satisfies $S(C,D)\cap O=X$.
\end{lem}

\begin{proof}
Denote the intervals of $O_X$ by $C_1$ and $D_1$.
We obtain $C_2$ from $C_1$ by replacing every bold edge $e$ of $O$ {with interior in~$C_1$} by the leaf-side $A_i$ of the tri-separation $(A_i,B_i)\in \sigma$ that corresponds to~$e$, and adding~$v$. 
We define $D_2$ analogously. 
Since $X$ is red, it contains no bold edges and $C_2\sm D_2$ and $D_2\sm C_2$ are nonempty. 
Thus $(C_2,D_2)$ is a mixed-separation of $G$ that almost interlaces~$\sigma$. 
Its separator is $X+v$, so it is a mixed 3-separation. Vertices of $X$ are red, so have two neighbours in $C_2$ and $D_2$. 
By 3-connectivity, $v$ has a neighbour in $C_2\sm D_2$ and in~$D_2\sm C_2$. 
So if $v$ has a neighbour in $X$, the mixed 3-separation $(C_2,D_2)$ is the desired tri-separation. 
Otherwise the reduction $(C,D)$ of $(C_2,D_2)$ satisfies $S(C,D)=X$ and almost interlaces $\sigma$ by \autoref{capture_star}, so it is the desired tri-separation. 
\end{proof}

The \emph{boundary} of an edge $e$ of $O$ is the 2-element set that, for each endvertex $u$ of~$e$, contains $u$ if $u$ is red, and otherwise contains the unique edge of~$O$ other than~$e$ that is incident with~$u$.

\begin{eg}\label{boundary_good}
    If $e$ is bold, then its boundary is a red mixed 2-separat{\"o}r.
\end{eg}

\begin{lem}\label{is_bold}
Assume \autoref{X} with crossing tri-separations and that $O$ does not have the type \ty{btt}.  
    If a tri-separation $(C,D)$ of~$G$ almost interlaces the tri-star $\sigma$ of $\Acal$ and $S(C,D)\cap O$ is equal to the boundary of a bold edge of~$O$, then $(C,D)$ or $(D,C)$ is in the tri-star of~$\Acal$ or $(C,D)$ is crossed by some tri-separation of~$G$.
\end{lem}

\begin{proof}
Let $e$ be a bold edge of~$O$ such that $S(C,D)\cap O$ is equal to the boundary of~$e$.
Let $(A,B)$ be the tri-separation in~$\sigma$ that corresponds to the edge~$e$. 
As $(C,D)$ almost interlaces~$\sigma$, we have $(A,B)\leq (C,D)$ or $(A,B)\leq (D,C)$, say the former. 
Abbreviate $X:=S(C,D)\cap O$. 
Let $P$ denote the interval of $O_X$ that does not include the interior of~$e$. 

\begin{sublem}\label{betterCalc}
    Assume that $P$ contains no neighbour of $v$ and exactly one bold edge~$f$. Then $(A,B)=(C,D)$ or $(C,D)$ is crossed by some tri-separation of~$G$.
\end{sublem}
\begin{cproof}
    If a vertex of $O$ is incident with two timid edges, this vertex is a neighbour of $v$ by \autoref{eg-red} and is in~$P$, which is excluded. So every vertex of $O$ is incident with a bold edge.
    So $O$ is a cycle with exactly two bold edges $e,f$ such that all its vertices are incident with a bold edge. So $O$ contains at most four vertices. 
    By \autoref{tiny-cases} and since $O$ does not have the type \ty{btt} by assumption, $O$ has the type $\ty{btbt}^+$. 
    As $P$ contains no neighbour of~$v$, by \autoref{plus} the two endvertices of $e$ are adjacent to~$v$ while no endvertex of~$f$ is adjacent to~$v$.
    Then both endvertices of $e$ are red, so $X$ equals the set of endvertices of~$e$.
    Let $B_e$ denote the leaf-bag at the bold edge~$e$, and let $B_f$ denote the leaf-bag at the bold edge~$f$.
    Since both endvertices of $e$ are adjacent to~$v$, by the definition of pseudo-reduction we have $A=B_e\cup\{v\}$, $B=B_f\cup\{v\}\cup X$ and $S(A,B)=X\cup\{v\}$.
    We clearly have $A-v=C-v$ and $B-v=D-v$.
    We have seen above that $(A,B)\le (C,D)$.
    So if $v\in D$, then $(A,B)=(C,D)$.

    Now assume that $v\notin D$.
    Then $C=B_e\cup\{v\}$ and $D=B_f\cup X$.
    Hence $S(C,D)=X\cup\{vu\}$, where $u$ is the unique neighbour of $v$ in $B_f\sm V(O)$ (using \autoref{v has neighbours}).
    We now find a tri-separation of $G$ that crosses $(C,D)$, as follows.
    Since the two vertices in $X$ must have at least two neighbours in~$D$, the edge $e$ of $O$ must exist in~$G$.
    So $X\cup\{v\}$ spans a triangle in~$G$.
    Let $xy$ and $x'y'$ denote the two timid edges of~$O$ so that $X=\{x,x'\}$.
    Let $U:=B_e\cup\{v\}$ and $W:=B_f\cup\{v\}\cup\{x'\}$.
    Then $S(U,W)=\{xy,v,x'\}$.
    Since $X\cup \{v\}$ spans an edge, $vx'\in E(G)$, so $(U,W)$ is a tri-separation.
    In the corner-diagram for $(C,D)$ and $(U,W)$, $x$ and $vu$ lie in opposite links, so $(C,D)$ and $(U,W)$ cross as desired.
\end{cproof}

\begin{sublem}\label{calc2}
The sum of neighbours of $v$ on $P$ plus bold edges of $P$ is at least two, or $(A,B)=(C,D)$ or $(C,D)$ is crossed by some tri-separation of~$G$.
\end{sublem}

\begin{cproof}
Assume not for a contradiction. Then by \autoref{betterCalc}, $P$ contains no bold edges and at most one neighbour of~$v$.
Then all edges of $O$ except for $e$ are timid. So all vertices not incident with $e$ are neighbours of $v$ by \autoref{eg-red} and they are in $P$. So $O$ has length three and the type \ty{btt}. 
By assumption this type is excluded, so we reach a contradiction.    
\end{cproof}\medskip

By \autoref{calc2} we may assume that the sum of neighbours of $v$ on $P$ plus bold edges of $P$ is at least two. 
So $v$ has two neighbours in $D$ by \autoref{v has neighbours}. By \autoref{independentEdges}, two edges incident with $v$ cannot both be in $S(C,D)$, so~$v\in D$. 

We shall show that $(A,B)= (C,D)$. 
We immediately get $A-v=C-v$ and $B-v=D-v$.
By definition of $\sigma$, the vertex $v$ is in~$B$. 
So~$B=D$. 
Since $(A,B)\leq (C,D)$, it remains to show that if $v\in C$, then~$v\in A$.
So assume~$v\in C$. 
Since $(C,D)$ is a tri-separation, $v$ has two neighbours in~$C$. 
As $C-v=A-v$, the vertex $v$ has two neighbours in $A-v$. 
So by the definition of $(A,B)$, the vertex $v$ is in~$A$. This completes the proof.
\end{proof}

We say that a mixed 2-separat{\"o}r $X$ is \emph{crossed} by a mixed 2-separat{\"o}r $Y$ if the two intervals of $O_X$ contain elements of $Y$ (these elements can be vertices from $Y$ or edges from $Y$ viewed as open intervals). 
Note that crossing is a symmetric relation for mixed 2-separat{\"o}rs.

\begin{lem}\label{red_angry}
Assume \autoref{X}.
A mixed 2-separat{\"o}r $X$ of~$O$ is crossed by a red mixed 2-separat{\"o}r of~$O$ if and only if $X$ is not equal to the boundary of a bold edge.    
\end{lem}
\begin{proof}
If $X$ is equal to the boundary of a bold edge~$e$, then one of the intervals of~$O_X$ consists only of the interior of~$e$ plus possibly some non-red endvertices of~$e$; thus $X$ is not crossed by a red mixed 2-separat{\"o}r. 
Conversely, if $X$ is not equal to the boundary of a bold edge, then both intervals of~$O_X$ either contain a red vertex, a timid edge or at least two edges.
Since we are done otherwise immediately, assume that we have the third outcome: two edges in an interval, and as we do not have the second outcome assume all the edges in the interval are bold. 
Then the interval has an internal vertex (a vertex that is not in the boundary of the interval), which is incident with two bold edges and thus is red. 
Hence $X$ is crossed by a red mixed 2-separat{\"o}r.
\end{proof}

\begin{lem}\label{is_splitting}
Assume \autoref{X} with crossing tri-separations.
The tri-star of $\Acal$ is a splitting star of the set of totally-nested nontrivial tri-separations.
\end{lem}

\begin{proof}
Let $(C,D)$ be a nontrivial totally nested tri-separation of~$G$ that almost interlaces the tri-star of~$\Acal$. Since every tri-separation with a degree-3-vertex $x$ in its separator is crossed (by the atomic cut at $x$), the totally nested tri-separation $(C,D)$ is strong.
     By \autoref{is_red}, $X:=S(C,D)\cap O$ is a red mixed 2-separat{\"o}r.
    By \autoref{red_angry}, either $X$ is crossed by a red mixed 2-separator $Y$ or $X$ is equal to the boundary of a bold edge of~$O$.
    In the second case, by \autoref{is_bold} and since we are otherwise done by \autoref{bttExtra}, the tri-separation $(C,D)$ or $(D,C)$ is in the tri-star of $\Acal$.
    In the first case, by \autoref{red_lift} there is a tri-separation $(E,F)$ almost interlacing the tri-star of $\Acal$ such that $S(E,F)\cap O=Y$.
    Since $X$ and $Y$ cross on $O$, the sets $X$ and $Y$ together ensure that all four links of the corner diagram for the tri-separations $(C,D)$ and $(E,F)$ are nonempty; thus the tri-separations $(C,D)$ and $(E,F)$ cross. We have shown that any nontrivial tri-separation $(C,D)$ that interlaces the tri-star of $\Acal$ is crossed by a tri-separation.
    To summarise, the tri-star of $\Acal$, which consists of  totally-nested nontrivial tri-separations by \autoref{is_splitting_star}, is splitting within the set of all totally-nested nontrivial tri-separations by \autoref{splitVsInterlace}.
\end{proof} 

\begin{lem}\label{is_gen-wheel}
Let $\Acal$ be a 2-connected apex-decomposition with central torso-cycle $O$ of a 3-connected graph, and let $\sigma$ denote the tri-star of~$\Acal$.
Then the expanded torso of $\sigma$ is a generalised wheel with rim~$O$, and the compressed torso of~$\sigma$ is a wheel with rim~$O$.
\end{lem}
\begin{proof}
An edge $xy$ of~$O$ is \emph{good} if there is some $(A,B)\in \sigma$ such that $\{x,y\}\se A\cap B$ and $v$ is not in the leaf-side~$A$.
Let $X$ denote the expanded torso of~$\sigma$.
The graph $X-v$ is isomorphic to the graph obtained from $O$ by adding for every good edge $xy$ of~$O$ a new vertex and joining it to $x$ and~$y$ (so that the three vertices form a triangle).
Hence $X$ is a concrete generalised wheel with rim~$O$.
By \autoref{make-concrete}, $X$ is also a generalised wheel, with the same rim.
The compressed torso of~$\sigma$ is obtained from~$X$ by contracting all edges that join $v$ to newly added vertices.
Since $X$ is a concrete generalised wheel with rim~$O$, it follows that the compressed torso is a wheel with rim~$O$.
\end{proof}

\begin{proof}[Proof of \autoref{univ_3sepr}~\textnormal{\ref{genWheelCase}}]
Let $G$ be a 3-connected graph, $N$ its set of totally-nested nontrivial tri-separations, and $\sigma$ a splitting star of~$N$.
Suppose that $\sigma$ is heavily interlaced by a tri-separation~$(A,B)$ of~$G$.
So $(A,B)$ is half-connected, strong and nontrivial.
As $(A,B)$ is not in~$N$, it is crossed by a tri-separation~$(C,D)$ of~$G$.
By \autoref{expendable}, we may assume that $(C,D)$ is nontrivial and strong.
By the \nameref{dasBesteLemma} (\ref{dasBesteLemma}), $(A,B)$ and $(C,D)$ cross so that the centre consists of a single vertex~$v$ and all links have size one.

By \autoref{apex_exists}, $G$ has a 2-connected apex-decomposition $\Acal$ with centre~$v$ whose tri-star~$\sigma'$ is interlaced by $(A,B)$ and~$(C,D)$, and whose central torso-cycle~$O$ alternates between $S(A,B)-v$ and $S(C,D)-v$; that is to say that we may assume \autoref{X} with crossing tri-separations.
By \autoref{is_splitting}, the tri-star~$\sigma'$ is a splitting star of~$N$. 
By \autoref{unique_split}, the fact that $(A,B)$ interlaces the two splitting stars $\sigma$ and~$\sigma'$ implies~$\sigma'=\sigma$.
Finally, by \autoref{is_gen-wheel}, the compressed and expanded torsos of $\sigma'=\sigma$ are a wheel and generalised wheel, respectively.
\end{proof}

\section[Proof of (iii)]{Proof of \ref{4tangleCase}}

A key step in this proof will be to understand how separations from the compressed torso for a splitting star $\sigma$ can be lifted to mixed-separations of~$G$ that interlace $\sigma$; this is then used to show that the compressed torso of $\sigma$ can only have very specific 3-separations when $\sigma$ is not interlaced at all, which roughly speaking is the essence of \ref{4tangleCase}. Next we prepare to lift.

\begin{lem}\label{edgeInAtMostTwoSeparators}
Let $G$ be a graph, and let $\sigma$ be a star of mixed-separations\plus\ of~$G$.
Then every edge of~$G$ lies in the separators of at most two elements of~$\sigma$.
\end{lem}
\begin{proof}
Assume for a contradiction that some edge $e$ of $G$ lies in the separators of three elements $(A_i,B_i)$ of~$\sigma$ for $i\in [3]$.
Let $x\in A_1\sm B_1$ and $y\in B_1\sm A_1$ denote the endvertices of~$e$.
Since $(A_1,B_1)\le (B_j,A_j)$ for both $j=2,3$ by the definition of star, the endvertex of $e$ in $A_j\sm B_j$ must be~$y$.
So $A_2\sm B_2$ and $A_3\sm B_3$ intersect in the vertex~$y$.
But $(A_2,B_2)\le (B_3,A_3)$ by the definition of star, which implies that $A_2\sm B_2$ and $A_3\sm B_3$ are disjoint, a contradiction.
\end{proof}

\begin{figure}[ht]
    \centering
    \includegraphics[height=8\baselineskip]{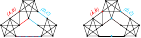}
    \caption{$\sigma$ and $\dot\sigma$}
    \label{fig:dots}
\end{figure}

\begin{dfn}[$\dot G$, $(\dot A,\dot B)$ and $\dot\sigma$]
Suppose now that $G$ is a graph and $\sigma$ is a star of mixed-separations\plus\ of~$G$.
In this context, we define the graph $\dot G$ to be the graph obtained from $G$ by subdividing every edge that lies in the separators of exactly two elements of~$\sigma$.
Let $(A,B)\in\sigma$ be arbitrary.
We define $\dot A$ to be the vertex set obtained from $A$ by adding for every edge $e\in S(A,B)$ the subdividing vertex of~$e$ if existent and the endvertex of~$e$ in~$B$ otherwise.
We define $\dot B$ to be the vertex set obtained from $B$ by adding for every edge $e\in S(A,B)$ its subdividing vertex if existent (the endvertex of $e$ in $A$ is not added).
Then $(\dot A,\dot B)$ is a separation\plus\ of~$\dot G$, which has the same order as~$(A,B)$.
We write $\dot\sigma:=\{\,(\dot A,\dot B):(A,B)\in\sigma\,\}$.
\end{dfn}

Let $G$ be a 3-connected graph, and let $\sigma$ be a star of nontrivial tri-separations of~$G$.
Let $X$ denote the compressed torso of~$\sigma$.
We define a map $\iota\colon V(X)\to V(\dot G)$ as follows.
Let $v$ be any vertex of~$X$.
If $v$ is not a contraction vertex, then $v$ also is a vertex of~$G$ and we let $\iota(v):=v$.
Otherwise $v$ is a contraction vertex with branch set~$U$.
If $U$ is spanned in $G$ by a single edge $e$ that lies in the separators of two elements of~$\sigma$, then we let $\iota$ take $v$ to the subdividing vertex of~$e$ in~$\dot G$.
Else $U$ intersects the bag of~$\sigma$ in a unique vertex, by \autoref{independentEdges}, and we let $\iota$ take $v$ to this unique vertex.
Let $\dot\iota$ be the restriction of $\iota$ onto its image.

\begin{lem}\label{dotGivesStar}
Let $G$ be a 3-connected graph, and let $\sigma$ be a star of nontrivial tri-separations of~$G$.
Then $\dot\sigma$ is a star of 3-separations\plus\ of~$\dot G$, and $\dot\iota$ is a graph-isomorphism between the compressed torso of $\sigma$ in~$G$ and the torso of $\dot\sigma$ in~$\dot G$.\qed
\end{lem}

\begin{lem}[Lifting Lemma]\label{separationLifting}
Let $\sigma$ be a star of separations\plus\ of a graph~$G$ and let $X$ denote the torso of~$\sigma$.
For every separation $(A,B)$ of $X$ there exists a separation $(\hat A,\hat B)$ of~$G$ such that $\hat A\cap V(X)=A$ and $\hat B\cap V(X)=B$ and $\hat A\cap \hat B=A\cap B$.
Moreover, $(\hat A,\hat B)$ almost interlaces $\sigma$.
\end{lem}
\begin{proof}
Let $(A,B)$ be given.
For every $(C,D)\in\sigma$, the separator $C\cap D\se V(X)$ is complete in~$X$, so $C\cap D$ is included in $A$ or in~$B$ (possibly in both).
We obtain $\hat A$ from $A$ by adding all vertices in $C\sm D$ from elements $(C,D)\in\sigma$ with $C\cap D\se A$, and we obtain $\hat B$ from $B$ by adding all vertices in $C\sm D$ from elements $(C,D)\in\sigma$ with $C\cap D\not\subseteq A$.
Then $\hat A\cup\hat B=V(G)$.
Let us assume for a contradiction that $G$ contains an edge $ab$ with $a\in\hat A\sm\hat B$ and $b\in\hat B\sm\hat A$.
Since $(A,B)$ is a separation of~$X$, not both $a$ and~$b$ can lie in~$V(X)$.
So $a\in C\sm D$ for some $(C,D)\in\sigma$ with $C\cap D\se A$, say.
Since $(C,D)$ is a separation\plus , it follows that $b$ must lie in~$C$, contradicting that $C\se \hat A$.

The equalities $\hat A\cap V(X)=A$ and $\hat B\cap V(X)=B$ are immediate from the fact that $C\sm D$ avoids $V(X)$ for all $(C,D)\in\sigma$.
The equality $\hat A\cap \hat B=A\cap B$ follows from the fact that $C\sm D$ is disjoint from $C'\sm D'$ for every distinct two $(C,D),(C',D')\in\sigma$.

To see that $(\hat A,\hat B)$ almost interlaces~$\sigma$, let $(C,D)\in\sigma$ be arbitrary.
Since the separator $C\cap D$ is complete in~$X$, we have $C\cap D\se A$ or $C\cap D\se B$.
In the former case, we have $(C,D)\le (\hat A,\hat B)$ since $C\se\hat A$ by construction and $D\supseteq V(X)\cup\bigcup_{(C',D')\in\sigma\sm\{(C,D)\}}D'\supseteq\hat B$ by the definition of star and construction of~$\hat B$.
In the latter case, we similarly find that $(C,D)\le (\hat B,\hat A)$.
Therefore, $(\hat A,\hat B)$ almost interlaces~$\sigma$.
\end{proof}

In the context of \autoref{separationLifting}, we say that $(A,B)$ \emph{lifts} to~$(\hat A,\hat B)$, and call $(\hat A,\hat B)$ a \emph{lift} of~$(A,B)$.

\begin{cor}\label{compressed_torso_3con}
    If $G$ is a subdivision of a 3-connected graph, and $\sigma$ is a star of 3-separations\plus\ of~$G$, and the bag of $\sigma$ does not include a degree-two vertex plus both its neighbours, then the torso of $\sigma$ is 3-connected or a~$K_3$.\qed
\end{cor}

\begin{dfn}[Hyper-lift]
Let $G$ be a 3-connected graph, and let $\sigma$ be a star of nontrivial tri-separations of~$G$.
Let $(C,D)$ be a separation of the compressed torso $X$ of $\sigma$.
A \emph{hyper-lift} of $(C,D)$ to~$G$ is a mixed-separation $(\hat C,\hat D)$ of~$G$ that is obtained from $(C,D)$ in the following way.
First, we view $(C,D)$ as a separation of the torso of~$\dot\sigma$, using \autoref{dotGivesStar}.
Next, we lift $(C,D)$ from the torso of $\dot\sigma$ to a separation $(C',D')$ of~$\dot G$, using the \nameref{separationLifting} (\ref{separationLifting}).
Finally, we let $\hat C:=C'\cap V(G)$ and $\hat D:=D'\cap V(G)$.
\end{dfn}

\begin{lem}\label{HyperLiftIsUseful}
Let $G$ be a 3-connected graph, and let $\sigma$ be a star of nontrivial tri-separations of~$G$.
Let $(C,D)$ be a separation of the compressed torso $X$ of $\sigma$, and let $(\hat C,\hat D)$ be a hyper-lift of $(C,D)$ to~$G$.
Then:
\begin{enumerate}
    \item the order of $(\hat C,\hat D)$ is at most the order of $(C,D)$;
    \item $(\hat C,\hat D)$ almost interlaces~$\sigma$;
   \item $|\hat C\sm \hat D|\geq |C\sm D|$;
   \item  for every $(A,B)\in \sigma$ we have $|(\hat C\sm \hat D)\cap B|\geq |C\sm D|$.
\end{enumerate}
\end{lem}

\begin{proof}
Let $(C',D')$ denote the separation of $\dot G$ that was used to obtain the hyper-lift $(\hat C,\hat D)$.

(1). First, we shall define an injection from $S(\hat C,\hat D)$ to $S(C',D')$.
To this end, let $x$ be an edge in  $S(\hat C,\hat D)$. Since $(C',D')$ has no edges in its separator, 
the edge $x$ must be an edge of $G$ that is not an edge of~$\dot G$. 
Denote by $y$ the unique subdivision vertex of $x$ in $\dot G$. 
Since the two neighbours of $y$ are in each of $C'\sm D'$ and $D'\sm C'$, and $S(C',D')$ contains no edges, the vertex $y$ is in $S(C',D')$.
Let $\varphi(x):=y$. 
We extend $\varphi$ to a map from $S(\hat C,\hat D)$ to $S(C',D')$ by taking the identity on vertices. 
This map is injective since subdivision vertices $y\in V(\dot G)$ of distinct edges of $G$ are distinct. 
Hence the order of $(\hat C,\hat D)$ is at most the order of $(C',D')$. 
By \autoref{separationLifting}, the order of $(C',D')$ is equal to the order of $(C,D)$. 

(2). We prove the stronger statement that $(C',D')$ almost interlaces $\dot\sigma$ (which gives the desired result as $V(G)\se V(\dot G)$ and we just need to restrict sides). This follows from \autoref{separationLifting}. 

(3). If $\sigma$ is empty, we are done and otherwise (3) follows from (4), so it remains to prove~(4).

(4). Let $\beta$ denote the bag of $\dot\sigma$ in~$\dot G$, and let~$(A,B)\in\sigma$. 
By \autoref{separationLifting}, we have that $|(C'\sm  D')\cap \beta|= |C\sm D|$. 
We define an injection from $(C'\sm  D')\cap \beta$ to $(\hat C\sm \hat D)\cap B$. 
Assume that $v\in C'\sm D'$ is a subdivision vertex of an edge $e$ of~$G$ (so in particular $v\in\beta$). 
Then there is some $(E,F)\in \sigma$ that is different from $(A,B)$ that has the edge $e$ in its separator. 
Let $x$ be the endvertex of $e$ in~$E\sm F$.
\begin{sublem}
    $x\in (\hat C\sm \hat D)\cap B$.
\end{sublem}
\begin{cproof}
In the proof of (2) we proved that $(C',D')$ almost interlaces $\dot\sigma$. As $v\in C'\sm D'$, we have that $(\dot E,\dot F)\leq (C',D')$. 
So $\dot E\sm \dot F\se C'\sm D'$. 
So $x\in C'\sm D'$. Since $x\in V(G)$, we deduce that  $x\in \hat C\sm \hat D$. As $(E,F)\leq (B,A)$, we have that $x\in B$.
\end{cproof}\medskip

Let $\varphi$ denote the map from $(C'\sm  D')\cap \beta$ to $(\hat C\sm \hat D)\cap B$ that is the identity on vertices of $G$ and maps subdivision vertices $v$ to vertices $x$ as defined above. 
Since $x$ is not in $\beta$, the sets $E'\sm F'$ for $(E',F')\in\sigma$ are disjoint, and since $x$ is incident with at most one edge of $S(E,F)$ by \autoref{independentEdges}, the map $\varphi$ is injective. 
\end{proof}

\begin{lem}\label{montag}
Let $G$ be a 3-connected graph.
Let $(A,B)$ be a strong tri-separation of~$G$, and let $(C,D)$ be a mixed 3-separation of~$G$ such that $(A,B)\le (C,D)$.
If $|(C\sm D)\cap B|\geq 1$, then every strengthening $(C',D')$ of $(C,D)$ satisfies $(A,B)<(C',D')$.
\end{lem}

\begin{proof}
By \autoref{x13}, we have $(A,B)\le (C',D')$.
By assumption, there is a vertex $v\in C\sm D$ that lies in~$B$.
Then $v$ also lies in $(C'\sm D')\cap B$.
Hence the inclusion $B\supseteq D'$ is proper.
\end{proof}

\begin{lem}\label{iii-lifting}
Let $G$ be a 3-connected graph.
Let $\sigma$ be a star of strong tri-separations of~$G$.
Suppose that the compressed torso of $\sigma$ has a 3-separation $(C,D)$ such that both sides have size at least five.
Then $\sigma$ is interlaced by a strong nontrivial tri-separation of~$G$.
\end{lem}
\begin{proof}
Let $(\hat C,\hat D)$ be a hyper-lift of $(C,D)$ to~$G$.
By \autoref{HyperLiftIsUseful} (applied to $(C,D)$ and $(D,C)$), $(\hat C,\hat D)$ is a mixed 3-separation of~$G$ that almost interlaces~$\sigma$, and it satisfies $|\hat C\sm \hat D|\geq 2$ and $|\hat D\sm \hat C|\geq 2$ by~(3). 
Moreover, by~(4), for every $(A,B)\in \sigma$ we have $|(\hat C\sm \hat D)\cap B|\geq 2$ and $|(\hat D\sm \hat C)\cap B|\geq 2$. 
Let $(\bar C,\bar D)$ be a strengthening of $(\hat C,\hat D)$.
Since $\bar{C}\sm\bar{D}=\hat C\sm\hat D$ and $\bar{D}\sm\bar{C}=\hat D\sm\hat C$, it follows with \autoref{trivial} that $(\bar C,\bar D)$ is nontrivial.
By \autoref{capture_star2}, $(\bar C,\bar D)$ almost interlaces~$\sigma$.
By \autoref{montag}, neither $(\bar C,\bar D)$ nor~$(\bar D,\bar C)$ lies in~$\sigma$.
Hence $(\bar C,\bar D)$ interlaces~$\sigma$.
\end{proof}

\begin{proof}[Proof of \autoref{univ_3sepr}~\textnormal{\ref{4tangleCase}}]
Let $G$ be a 3-connected graph, let $N$ denote its set of totally-nested nontrivial tri-separations, and let $\sigma$ be a splitting star of~$N$.
Suppose that $\sigma$ is not interlaced by a strong nontrivial tri-separation of~$G$.
We denote by $X$ the compressed torso of~$\sigma$.
By \autoref{dotGivesStar} and \autoref{compressed_torso_3con}, $X$ is 3-connected or a $K_3$, and we are done in the latter case.
The tri-separations in~$\sigma$ are strong by \autoref{notStrongImpliesCrossed}.
Hence by the contrapositive of \autoref{iii-lifting}, every 3-separation of $X$ has a side with at most four vertices.
Hence $X$ is quasi 4-connected or a~$K_4$.
\end{proof}

\begin{proof}[Proof of \autoref{univ_3sepr}]
We have proved \ref{K3Case}, \ref{genWheelCase} and \ref{4tangleCase} in the respective sections above.
The `Moreover' part holds by \autoref{torsoMinors}.
\end{proof}

\begin{proof}[Proof of \autoref{mainIntro}]
\autoref{univ_3sepr} implies \autoref{mainIntro}.
\end{proof}

\section{Tutte's Wheel Theorem}

If $G$ is a graph and $e$ is an edge of~$G$, we denote by $G/e$ the (multi-)graph that arises from $G$ by contracting~$e$.
Recall that a 3-connected (multi-)graph does not have parallel edges.
A graph $G$ is \emph{minimally 3-connected} if it is 3-connected and for every edge $e$ of~$G$ neither $G-e$ nor $G/e$ is 3-connected.

\begin{thm}[Tutte's Wheel Theorem~\cite{TutteWheel}]\label{tutte-wheel}
    Every minimally 3-connected finite graph~$G$ is a wheel.
\end{thm}

In this section we give an automatic proof of Tutte's wheel theorem; the proof strategy is as follows. 
Take a minimally 3-connected graph~$G$. 
First, we show that all totally-nested tri-separators of~$G$ consist of three vertices that do not span any edge. 
Now consider a `leaf-torso'\footnote{A \emph{leaf-torso} means a compressed torso of a splitting star $\{\,(A,B)\,\}$ where $(A,B)$ is a $\le$-maximal totally-nested nontrivial tri-separation.} of the set of totally-nested nontrivial tri-separations. 
By \autoref{univ_3sepr} there are three options how this torso might look like and one easily checks that none of them is possible. 
Hence $G$ has no totally-nested nontrivial tri-separation, and again by \autoref{univ_3sepr} we have three options how $G$ might look like; two are excluded for the same reasons and thus the only possibility is that $G$ is a wheel.
The details are as follows.
We use `(c)' and `(d)' as a reminder for \underline{c}ontraction and \underline{d}eletion.

\begin{obs}\label{tetra-pack}
Let $e$ be an arbitrary edge of a 3-connected graph~$G$. Then:
\begin{enumerate}
    \item[\textnormal{(c)}] if the ends of $e$ do not lie in a common 3-separator of~$G$ and $e$ does not lie in a triangle of~$G$, then $G/e$ is 3-connected;
    \item[\textnormal{(d)}] if $e$ does not lie in a mixed 3-separator of~$G$, then $G-e$ is 3-connected.\qed
\end{enumerate}
\end{obs}

An edge $e$ of~$G$ is of \emph{type c or d} if it satisfies the premises of conditions (c) or (d) in \autoref{tetra-pack}, respectively.

\begin{obs}\label{eg17}
    Minimally 3-connected graphs do not have any edges of type c or~d.\qed
\end{obs}

\begin{lem}\label{triangleTrick}
    Let $G$ be a 3-connected graph.
    Let $(A,B)$ be a totally-nested nontrivial tri-separation of~$G$.
    Assume that $S(A,B)$ contains an edge $ab$ with $a\in A\sm B$ and $b\in B\sm A$ such that $ab$ lies in a triangle $\Delta\se G$.
    Let $A':=A\cup\{b\}$ and $B':=B$.
    Then $(A',B')$ is a totally-nested nontrivial tri-separation of~$G$ with $(A,B)<(A',B')$.
\end{lem}
\begin{proof}
    Let $x$ denote the third vertex in $\Delta$ besides $a$ and~$b$, and note that $x\in S(A,B)$.
    The separator $S(A',B')$ is obtained from $S(A,B)$ by replacing the edge $ab$ with the vertex~$b$; in particular, $x\in S(A',B')$.
    The edge $xb$ ensures that $(A',B')$ is a tri-separation, and clearly $(A',B')$ is nontrivial.
    Assume for a contradiction that $(A',B')$ is crossed by a tri-separation~$(C,D)$.
    Since $(A,B)<(A',B')$, we must have $(A,B)\le (C,D)$ or $(A,B)\le (D,C)$; say $(A,B)\le (C,D)$.
    The edge $ab$ witnesses that $G[A'\sm B']$ is connected, so $(A',B')$ is half-connected.
    Hence $(A',B')$ and $(C,D)$ cross with all links of size one and a single vertex~$v$ in the centre, by the \nameref{dasBesteLemma}~(\ref{dasBesteLemma}).

    We claim that $v=x$.
    Since $b,x\in S(A',B')$ and $bx$ is an edge in~$G$, we have $v=b$ or $v=x$.
    We have $b\notin C$ since otherwise $A'=A\cup\{b\}\se C$ and $B'=B\supseteq D$ (from $(A,B)\le (C,D)$) together give $(A',B')\le (C,D)$, a contradiction.
    Hence $ab\in S(C,D)$.
    So neither $a$ nor $b$ lies in $S(C,D)$.
    In particular, $v\neq b$, and therefore~$v=x$.

    Let $C':=C$ and $D':=D\cup\{a\}$.
    We claim that $(C',D')$ is a tri-separation that crosses $(A,B)$.
    The separator of $(C',D')$ is obtained from the separator of $(C,D)$ by replacing the edge $ab$ with the vertex~$a$.
    The vertex $x\in S(C',D')$ is a neighbour of~$a$, which ensures that $(C',D')$ is a tri-separation.
    To show that $(C',D')$ crosses $(A,B)$, it suffices to find two opposite links that are non-empty.
    The link for $A$ contains the vertex~$a$.
    The link for $B$ contains the non-empty link for $B'$ with regard to how $(C,D)$ crosses $(A',B')$, as $B=B'$ and $\{a,b\}\se A'$.
    So $(C',D')$ crosses $(A,B)$, contradicting that $(A,B)$ is totally-nested.
\end{proof}

\begin{lem}\label{totally-nice}
    Let $G$ be a minimally 3-connected graph. 
    Among the totally-nested nontrivial tri-separations of~$G$, let $(C,D)$ be maximal with regard to the partial order~$\le$.
    Then $S(C,D)$ consists of three vertices that do not span any edge.
\end{lem}

\begin{proof}
We claim that every edge $e=vw$ in $S(C,D)$ is of type~c. 
Indeed, by \autoref{triangleTrick} and the maximality of $(C,D)$, the edge $e$ does not lie in a triangle.
Moreover, the reduction $(E,F)$ of any mixed 3-separation with $v$ and $w$ in its separator also contains $v$ and $w$ in its separator (since $vw$ is an edge), and hence $(E,F)$ crosses $(C,D)$ with $v$ and $w$ in opposite links, which is not possible by total-nestedness of~$(C,D)$.
So every edge in $S(C,D)$ is of type~c.

As $G$ is minimally 3-connected, $S(C,D)$ does not contain any edge by \autoref{eg17}.
So $S(C,D)$ consists of three vertices.

We claim that every edge $e=vw$ between two vertices $v,w\in S(C,D)$ is of type~d.
Indeed, the reduction $(E,F)$ of any mixed 3-separation with $e$ in its separator crosses $(C,D)$ with $v$ and $w$ in opposite links, which is not possible by total-nestedness of~$(C,D)$.

As $G$ is minimally 3-connected, no two vertices in $S(C,D)$ span an edge by \autoref{eg17}. 
This completes the proof. 
\end{proof}

\begin{lem}\label{technical1}
    Let $G$ be a minimally 3-connected graph and let $X$ be a nonempty set of vertices of~$G$ such that the neighbourhood $N(X)$ does not span any edge. 
    Then there is a nontrivial tri-separation $(U,W)$ of~$G$ whose separator contains a vertex of~$X$ or an edge that is incident with a vertex of~$X$. 
\end{lem}

\begin{proof}
    For this, let $e$ be an arbitrary edge of $G$ with an endvertex in~$X$. Since $e$ is not of type~d by \autoref{eg17}, it lies in the separator of a mixed 3-separation $(A,B)$ of~$G$.
    Since we are done otherwise, we may assume that the reduction of $(A,B)$ is trivial. 
    So an endvertex $v$ of $e$ has degree three. 
    The lemma is trivial for $G=K_4$. 
    So assume that there is a mixed 3-separation $(C,D)$ of~$G$ with separator equal to~$N(v)$.
\begin{sublem}\label{no-edge}
   If $N(v)$ spans an edge  $f=ab$, then there is a nontrivial tri-separation of~$G$ whose separator contains a vertex of~$X$ or an edge that is incident with a vertex of~$X$.
\end{sublem}
\begin{cproof}
    The mixed 3-separation $(\{a,b,v\},V(G)-v)$ is a nontrivial tri-separation of $G$.
    Every endvertex of the edge $e$ is incident with the unique edge of its separator or is in its separator. Thus the endvertex of $e$ in $X$ witnesses that this tri-separation has the desired property. 
\end{cproof}

    Since we are done otherwise by \autoref{no-edge}, assume that $N(v)$ does not span an edge; that is, $e$ is not in a triangle. 
    Since $e=vw$ is not of type~c by \autoref{eg17}, there is a 3-separation $(E,F)$ with $v$ and $w$ in its separator; its reduction 
    has the neighbours $v$ and $w$ in its separator and hence is nontrivial by \autoref{trivial}. 
    As one of $v$ and $w$ is in $X$, this gives the desired result.
\end{proof}
\begin{obs}\label{technical2}
The nontrivial tri-separation $(U,W)$ in \autoref{technical1} can be chosen strong.
\end{obs}
\begin{proof}
Via \autoref{reduceToStrong}, take a strengthening of the tri-separation given by \autoref{technical1}. 
\end{proof}

\begin{lem}\label{no-sep}
    A minimally 3-connected finite graph $G$ has no totally-nested nontrivial tri-separation.
\end{lem}

\begin{proof}
    Suppose for a contradiction that $G$ has a totally-nested nontrivial tri-separation~$(A,B)$. 
    Pick such an $(A,B)$ that is maximal with regard to the partial order~$\le$ on mixed-separations.
    Then $\sigma:=\{\,(A,B)\,\}$ is a splitting star of the set of totally-nested nontrivial tri-separations of~$G$. 
    Let $U:=B\sm A$.
    By \autoref{totally-nice}, the separator of $(A,B)$ consists of three vertices that do not span an edge.
    In particular, $N(U)=A\cap B$. 
    Applying \autoref{technical1} together with \autoref{technical2} to $U$ yields that there is a strong nontrivial tri-separation $(C,D)$ of~$G$ such that $S(C,D)$ contains a vertex of $U$ or an edge incident with a vertex of~$U$. 
    Since $(A,B)$ is totally-nested, this implies that
    $(A,B)<(C,D)$ or $(A,B)<(D,C)$.   
    So $(C,D)$ interlaces~$\sigma$.
    By \autoref{univ_3sepr}, the compressed torso~$X$ of $\sigma$ either is a wheel or a thickened $K_{3,m}$ or $G=K_{3,m}$ with $m\geq 0$.
    As $K_{3,m}$ has no totally-nested nontrivial tri-separation, $G$ is not a~$K_{3,m}$.
    Note that as $S(A,B)$ consists only of vertices, $X$ is a genuine torso. 
    Recall that $S(C,D)$ contains a vertex of $U$ or an edge. 
    
    We claim that $X$ is a wheel, and suppose that $X$ is a thickened~$K_{3,m}$.
    Then $G[B]$ is a $K_{3,m}$ with $A\cap B$ equal to the left class of size three.
    As $(C,D)$ is strong, its separator contains no vertex of~$U$, so $S(C,D)$ contains an edge. 
    Then both $G[C\sm D]$ and $G[D\sm C]$ are connected, and so $(C,D)$ interlaces $\sigma$ heavily and $X$ is a wheel~$W$.
    
    The set $A\cap B$ spans a triangle in $W$. 
    Since this set spans no edge in~$G$, the graph $G[B]$ is obtained from $W$ by deleting the edges of a triangle. 
    Since all but at most one vertex of $W$ have degree three, this leaves a vertex of $A\cap B$ with degree one in~$G[B]$, a contradiction to the assumption that $(A,B)$ is a tri-separation.
\end{proof}

\begin{proof}[Automatic proof of \autoref{tutte-wheel}]
Every edge of the graph $K_{3,m}$ with $m\geq 3$ is of type~c. So by \autoref{eg17}, $G$ is not a $K_{3,m}$ with $m\geq 3$.
Applying \autoref{technical1} with $X=V(G)$ yields that $G$ has a nontrivial tri-separation, so $G$ is not internally 4-connected by \autoref{tetraXtrisep}.
By \autoref{no-sep}, $G$ has no totally-nested nontrivial tri-separation.
Hence by the \nameref{Angry} (\ref{Angry}), $G$~is a wheel.
\end{proof}

A natural problem in this area is to understand which edges of 3-connected graph are \emph{essential} in that they cannot be contracted or deleted without destroying 3-connectivity; see for example \cite{ando1987contractible}, and \cite{kriesell2008number} for further extensions. \autoref{univ_3sepr} and our automatic proof of Tutte's wheel theorem provide a new perspective on essential edges, and it is not unreasonable to conjecture that these ideas can be used to resolve this problem.

\clearpage
\renewcommand{\thechapter}{3}
\phantomsection
\noindent{\huge{\textbf{Chapter \thechapter \\ \\Outlook}}}\label{sec:conc}
\addcontentsline{toc}{section}{\thechapter\ Outlook}\bigskip\bigskip

\setcounter{section}{1}

We start by reviewing directions to continue this research. Similarly as for graphs, decompositions along 3-separations are a key tool to study matroids, for example in the context of matroids representable over finite fields  \cite{geelen2013inequivalent} and for splitter theorems (and strengthenings thereof) \cite{brettell2014splitter,Splitter4conMatroids,Matroids3conSplitter}. 

\begin{oprob}
Extend \autoref{Angry} (and then \autoref{univ_3sepr}) to 3-connected matroids. 
\end{oprob}

To this end, a natural way to define tri-separations of matroids is the following.
Given a 3-connected matroid~$M$, a \emph{nontrivial mixed 3-separation} of~$G$ is a triple $(A,S,B)$ such that $A,S,B$ partition the ground set of~$M$ and for every bipartition $S=A'\sqcup B'$ we have that $(A\cup A',B\cup B')$ is a 3-separation of~$M$.
If $S$ is inclusionwise maximal, meaning that there is no $S'\supsetneq S$ such that $(A\sm S',S',B\sm S')$ is a nontrivial mixed 3-separation of~$M$, then $(A,S,B)$ is a \emph{nontrivial tri-separation} of~$M$.
Note that $(A,S,B)$ is a nontrivial tri-separation of $M$ if and only if it is a nontrivial tri-separation of the dual $M^*$ of $M$.

\begin{eg}
    In $U_{3,m}$ for $m\geq 6$ every nontrivial tri-separation is crossed by a {nontrivial} tri-separation.
\end{eg}

Another direction for future research is the following:

\begin{oprob}\label{etd}
Extend \autoref{Angry} (and then \autoref{univ_3sepr}) to separators of size larger than~3.
\end{oprob}

\autoref{etd} has recently been solved for size~4 in~\cite{Tutte4con}.

An instructive example concerning \autoref{etd} is the line-graph of the 3-dimensional cube.
In~this graph, there are three 4-separations that cross \lq 3-dimensionally\rq, as depicted in \autoref{fig:3Dcube} below.

\begin{figure}[ht]
    \centering
    \includegraphics[height=10\baselineskip]{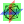}
    \caption{Three 4-separations crossing 3-dimensionally}
    \label{fig:3Dcube}
\end{figure}

(Mixed)~$k$-separations of order~$k<4$ do not cross `3-dimensionally': 
this is trivial for $k=1$; for $k=2$ we can read it from the \nameref{2SeparatorTheorem} (\ref{2SeparatorTheorem}); and for $k=3$ this follows by combining the \nameref{dasBesteLemma} (\ref{dasBesteLemma}) with \autoref{apex_exists}.
Some hope towards a solution of \autoref{etd} stems from the results of~\cite{carmesin2014k} (building on earlier work of~\cite{Watkins70}), where it is proved that if a $k$-connected but not $(k+1)$-connected graph has minimum degree larger than $\frac{3k}{2}-1$, then it has a totally-nested $k$-separation (and in fact every $k$-separation $(A,B)$ with $A$ minimal is totally-nested).

\begin{rem}
 We expect that many results about 3-connected graphs in the literature can be derived in a fairly straightforward way from \autoref{univ_3sepr}, for example those in \cite{DiestelBookCurrent} or \cite{oxley2}.
\end{rem}

Our main result \autoref{univ_3sepr} has quite a few applications in addition to the ones presented here. Whilst for some of these applications, our papers are at an early stage, the following applications will appear on the arXiv shortly (or already have appeared):

\begin{enumerate}
    \item Consider the following \emph{connectivity augmentation problem from 0 to~4}.
    Suppose that we are given a graph~$G$, a set $F\se [V(G)]^2$ of edges not in~$E(G)$, and an integer $k\ge 0$.
    Decide whether there is a $k$-element subset $X$ of $F$ such that $G+X$ is 4-connected.
    In upcoming work, the first author and Sridharan present an algorithm that solves this problem and that is an FPT-algorithm: its running time is upper-bounded by some function in $k$ times a polynomial in $|V(G)|$. 
    The property of total-nestedness is crucial for this algorithm\jaf{~\cite{conny_aug}}{~\cite{conny_aug}}{}.
    \item We characterised 4-tangles through a connectivity property\jaf{~\cite{greedy_grohe}}{~\cite{greedy_grohe}}{}.
    \item A~wheel-minor $W$ of a 3-connected graph $G$ is \emph{stellar} if $G$ admits a star-decomposition of adhesion three such that $W$ is equal to the central torso and all leaf-bags include a cycle. 
    We shall show that every stellar wheel-minor of $G$ where the rim is sufficiently large is a minor of an expanded torso of the set of totally-nested nontrivial tri-separations of~$G$\jaf{~\cite{stellar_wheel}}{~\cite{stellar_wheel}}{}.
\end{enumerate}

In the following, we compare the decomposition of this paper with Grohe's~\cite{grohe2016quasi} and with the findings of the upcoming work\jaf{~\cite{greedy_grohe,stellar_wheel}}{~\cite{greedy_grohe,stellar_wheel}}{ [references removed for anonymisation, also in the following]}.
Details that we skip here will be addressed in\jaf{~\cite{greedy_grohe,stellar_wheel}}{~\cite{greedy_grohe,stellar_wheel}}{ the upcoming work}.
The results of this paper and related works give rise to three types of decompositions of 3-connected graphs, labelled below by (D1) to~(D3).
The decomposition~(D1) is obtained by taking an inclusion-wise maximal set of pairwise nested nontrivial 3-separations; this is essentially the decomposition constructed by Grohe~\cite{grohe2016quasi} and we refer to the upcoming work\jaf{~\cite{greedy_grohe}}{~\cite{greedy_grohe}}{} for a refined analysis of this decomposition.
The decomposition (D3) is that of \autoref{univ_3sepr}.
The decomposition (D2) is obtained from (D3) by applying to each quasi 4-connected compressed torso the decomposition~(D1); so (D2) refines~(D3).

We made a list of desirable properties for such decompositions, (A1)--(C3) below, and compare the decompositions on the basis of these properties in the following chart.

\medskip

\noindent\begin{tabular}{ | l | l | l | l | l |}
\hline 
(D1) & (D2) & (D3) &  & Property\\ \hline
$\times$ & $\times$ & $\times$ & (A1) & 4-tangles appear${}^1$ as torsos\\ \hline
\checkmark & \checkmark & $\times$ & (A2) & non-cubic${}^2$ 4-tangles appear as torsos\\ \hline
\checkmark & \checkmark & \checkmark & (A3) & non-cubic 4-tangles live in different quasi 4-connected torsos\\ \hline
\checkmark & \checkmark & \checkmark & (B1) & every non-cubic internally 4-connected minor of $G$ is a minor of some torso\\ \hline
$\times$ & \checkmark & \checkmark & (B2) & every stellar $m$-wheel minor of $G$ with $m\geq 5$ is a minor of some torso\\ \hline
\checkmark & \checkmark & $\times$ & (C1) & all torsos are internally 4-connected, thickened $K_{3,m}$'s or generalised wheels\\ \hline
\checkmark & \checkmark & \checkmark & (C2) & all torsos are quasi 4-connected, thickened $K_{3,m}$'s or generalised wheels\\ \hline
$\times$ & $\times$ & \checkmark & (C3) & canonical \\ \hline
\end{tabular}\bigskip

To see that (A1) fails, construct a graph $G$ as follows. 
Start with a set $X$ of four vertices and glue a clique $K_{10}$ at each 3-element subset of~$X$.
Then $G$ has a cubic 4-tangle $\theta$ that lives on~$X$. 
In every decomposition (D$i$) the set $X$ determines a $K_4$~torso such that $\theta$ can only possible live in that torso, but $K_4$ has no 4-tangle. 
The results from the upcoming work\jaf{~\cite{greedy_grohe}}{~\cite{greedy_grohe}}{} show that the properties (A3) and (C2) hold for all three decompositions, and that (A2) and (C1) hold for (D1) and~(D2).
\autoref{eg100} shows that (A2) and (C1) fail for~(D3).
(B1) follows from (A3) via our characterisation of 4-tangles \jaf{from~\cite{greedy_grohe}}{from~\cite{greedy_grohe}}{}.
In the upcoming work\jaf{~\cite{stellar_wheel}}{~\cite{stellar_wheel}}{} we show that (B2) holds for (D2) and (D3) and that $m\geq 5$ is necessary, and we also show that no tree-decomposition can possess this property, so in particular not~(D1).
Clearly, (C3) holds for (D3).
The necklace of $K_5$'s from the introduction shows that (C3) fails for (D1).
\autoref{eg100} shows that (C3) fails for (D2) as well.

\vspace*{\fill}
\noindent ${}^1$ The~4-tangles of~$G$ \emph{appear as torsos} of~(D$i$) if there is a natural injection $\iota$ from the 4-tangles to the torsos of~(D$i$) such that $\iota(\theta)$ is internally 4-connected and its unique 4-tangle lifts to $\theta$; see \cite{greedy_grohe} for details.

\noindent ${}^2$ A set $X=\{v_1,v_2,v_3,v_4\}$ of four vertices of a graph $G$ is \emph{cubic} if there
are components $C_1,C_2,C_3,C_4$ of $G\sm X$ such that $N(C_i)=X-v_i$ for $i=1,2,3,4$, and no 
component of $G\sm X$ has the whole of $X$ in its neighbourhood.
If $X$ is cubic, then $G$ has the cube $Q_3$ as a minor where one bipartition class is $X$ and the other is $\{C_1,C_2,C_3,C_4\}$.
A 4-tangle is \emph{cubic} if it lives on a cubic vertex set~$X$; that is, every big side in the 
tangle contains~$X$.

\clearpage
\renewcommand{\thechapter}{4}
\phantomsection
\noindent{\huge{\textbf{Chapter \thechapter \\ \\Reviewing the 2-Separation Theorem}}}
\addcontentsline{toc}{section}{\thechapter\ Reviewing the 2-Separation Theorem}

\setcounter{section}{0}

\section{Overview of this chapter}\label{sec:2SeparatorTheorem}

A basic fact about graphs states that every connected graph can be cut along its cutvertices in a tree-like way into maximal 2-connected subgraphs and bridges.
2-connected graphs can be decomposed further in the same vein, which is useful to study planar embeddings of graphs, but it is no longer obvious where to best cut these graphs.
MacLane found for every 2-connected graph $G$ a tree-decomposition of adhesion two all whose torsos are 3-connected, cycles or~$K_2$'s~\cite{MacLane}.
Tutte~\cite{TutteGrTh} later found a canonical such tree-decomposition, for which Cunningham and Edmonds discovered an elegant one-step construction~\cite{cunningham_edmonds_1980}.
Here we review and prove a structural version of Tutte's result with the description by Cunningham and Edmonds, using the terminology of this paper, and then derive the \nameref{2SeparatorTheorem} (\ref{2SeparatorTheorem}) from it.

\section{Characterising nestedness through connectivity}

\begin{fact}\label{connectedSides}
If $G$ is 2-connected and $(A,B)$ is a 2-separation of~$G$, then $G[A]$ and~$G[B]$ are connected, and neither vertex in $A\cap B$ is a cutvertex of $G[A]$ or~$G[B]$.
\end{fact}
\begin{proof}
Every component of $G-(A\cap B)$ has neighbourhood equal to~$A\cap B$.
\end{proof}

The two situations in~\ref{crossCross} and~\ref{crossFourFlip} of the following lemma are depicted in~\autoref{fig:TwoCrossings}.
Recall that a separation $(A,B)$ \emph{separates} two vertices $u,v$ if $u\in A\sm B$ and $v\in B\sm A$ or vice versa.

\begin{figure}[ht]
    \centering
    \includegraphics[height=8\baselineskip]{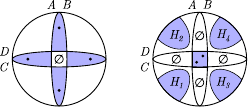}
    \caption{The two ways in which 2-separations can cross}
    \label{fig:TwoCrossings}
\end{figure}

\begin{lem}\label{2sepCrossingChar}
Two 2-separations $(A,B)$ and $(C,D)$ of a 2-connected graph~$G$ cross if and only if one of the following assertions holds:
\begin{enumerate}[label={\textnormal{(X\arabic*)}}]
    \item\label{crossCross} $(A,B)$ separates the two vertices in $C\cap D$ while $(C,D)$ separates the two vertices in $A\cap B$; or
    \item\label{crossFourFlip} $A\cap B=C\cap D$ and there are four components $H_1,\ldots,H_4$ of $G-(A\cap B)$ such that $H_1,H_2\subseteq G[A]$ and $H_3,H_4\subseteq G[B]$, while $H_1,H_3\subseteq G[C]$ and $H_2,H_4\subseteq G[D]$.
\end{enumerate}
\end{lem}

If $(A,B)$ and $(C,D)$ cross as in~\ref{crossCross}, we say that they cross \emph{like in a cycle}.
If $(A,B)$ and $(C,D)$ cross as in~\ref{crossFourFlip}, we say that they cross \emph{with a four-flip}.

\begin{proof}[Proof of \autoref{2sepCrossingChar}]
The backward implication is straightforward.
For the forward implication, suppose that $(A,B)$ and $(C,D)$ cross.
Let us write $X:=A\cap B$ and $Y:=C\cap D$.
We consider three cases.

\textbf{Case} $X\cap Y=\emptyset$.
We have to show that $(A,B)$ and $(C,D)$ cross like in a cycle.
Let us suppose for a contradiction that they don't.
Then the two vertices in~$X$, say, are not separated by~$(C,D)$.
With $X\cap Y=\emptyset$, it follows that $X\se C\sm D$, say.
As $G[D]$ is connected by \autoref{connectedSides} and avoids~$X$, it follows that $G[D]$ is included in a unique component~$I$ of $G-X$.
Without loss of generality, $I\se G[B]$, so $B\supseteq D$.
From $Y\se I\se G[B]$ and $X\cap Y=\emptyset$ we deduce that $G[A]$, which is connected by \autoref{connectedSides}, is a connected subgraph of $G-Y$, and lies in the component $J$ of $G-Y$ that contains the subset $X\se A$.
Thus $G[A]\se J\se G[C]$, that is, $A\se C$.
Hence $(A,B)\le (C,D)$, a contradiction.

\textbf{Case} $|X\cap Y|=1$.
We will show that this case is impossible.
Let us denote the vertex in the intersection $X\cap Y$ by~$z$, and let us denote the vertices in $X\sm Y$ and $Y\sm X$ by $x$ and~$y$, respectively; so $X=\{x,z\}$ and $Y=\{y,z\}$.
Let $K(X,y)$ denote the component of $G-X$ that contains~$y$, and let $K(Y,x)$ denote the component of $G-Y$ that contains~$x$.
Without loss of generality, $K(X,y)\se G[B]$ and $K(Y,x)\se G[C]$.
Since $G[A]-z$ is connected by \autoref{connectedSides} and a subgraph of $G-Y$ that contains~$x$, it must be included in the component $K(Y,x)$ of $G-Y$ that contains~$x$.
Hence $G[A]\se K(Y,x)\cup N(K(Y,x))\se G[C]$.
A~symmetric argument shows $G[D]\se G[B]$.
So $(A,B)\le (C,D)$, a contradiction.

\textbf{Case} $X=Y$. Then it is straightforward to deduce that $(A,B)$ and~$(C,D)$ cross with a four-flip.
\end{proof}

A 2-separation $(A,B)$ of a graph~$G$ is \emph{\tn} if
\begin{itemize}
    \item at least one of $G[A]$ and $G[B]$ is 2-connected, and
    \item at least one of $G[A\sm B]$ and $G[B\sm A]$ is connected.
\end{itemize}
\autoref{2sepCrossingChar} implies the following characterisation of total nestedness through external connectivity:

\begin{cor}\label{SufficientForTotallyNested}
A 2-separation of a 2-connected graph is totally nested if and only if it is externally 2-connected.\qed
\end{cor}

\section{When all 2-separations are crossed}

\begin{thm}[Angry 2-Separation Theorem]\label{Angry2Sep}
If a 2-connected graph $G$ has a 2-separation and every 2-separation of $G$ is crossed by another 2-separation, then $G$ is a cycle of length~$\ge 4$.
\end{thm}

\begin{proof}[Proof of \autoref{Angry2Sep}]
Let $(A,B)$ be a 2-separation of~$G$.
We may choose $(A,B)$ so that $G[A\sm B]$ is connected.
Let $(C,D)$ be a 2-separation of~$G$ that crosses $(A,B)$.
Since $G[A\sm B]$ is connected, $(C,D)$ must cross $(A,B)$ like in a cycle.
Let $T_1$ and $T_2$ be the block graphs of $G[A]$ and of~$G[B]$, respectively, where we use the definition of block graphs as in \cite[§3.1]{DiestelBookCurrent}.

The two vertices in $A\cap B$ are separated in $G[A]$ by the vertex of $C\cap D$ that lies in~$A$, so they lie in distinct blocks of~$G[A]$.
Since $G$ is 2-connected, the vertices in $A\cap B$ are not cutvertices of~$G[A]$, so the blocks containing them are unique.
Let $P_1$ be the unique path in~$T_1$ that links these two blocks.
Then $T_1=P_1$, because otherwise some edge of~$T_1$ leaving~$P_1$ would induce a 1-separation of~$G[A]$ with $A\cap B$ contained in one side, which in turn would extend to a 1-separation of~$G$, contradicting that $G$ is 2-connected.
Similarly, we find that $T_2$ is a path linking the unique blocks of $G[B]$ containing the two vertices in~$A\cap B$.
So it remains to show that all blocks of $G[A]$ and of $G[B]$ are~$K_2$'s.

Let us assume for a contradiction that $G[A]$, say, has a 2-connected block~$X$.
Let $Y$ denote the union of all blocks of $G[A]$ and of~$G[B]$ except~$X$.
Then $\{V(X),V(Y)\}$ is a 2-separation of~$G$.
By \autoref{SufficientForTotallyNested}, it it totally nested as $X$ is 2-connected and $Y\sm X$ is connected, a contradiction.
\end{proof}

\begin{cor}\cite[Theorem~3]{infiniteSPQR}\label{CayleyCase2}
Every vertex-transitive finite connected  graph either is 3-connected, a cycle, a $K_2$ or a~$K_1$.
\end{cor}
\begin{proof}
Let $G$ be a finite connected vertex-transitive graph.
If $|G|\le 3$, then $G$ is a complete graph on $\le 3$ vertices, so we may assume that $|G|\ge 4$.

We claim that $G$ is 2-connected.
Otherwise $G$ has a cutvertex.
Then every vertex of $G$ is a cutvertex.
Let $T$ be the block graph of~$G$, and let $t$ be a leaf of~$T$.
Then $t$ is a block, but contains at most one cutvertex.
So some vertex in~$t$ is not a cutvertex of~$G$, a contradiction.

Let us suppose now that $G$ is not 3-connected, so $G$ has a 2-separation.
If every 2-separation of $G$ is crossed by another one, then $G$ is a cycle of length~$\ge 4$ by the \nameref{Angry2Sep} (\ref{Angry2Sep}).
Otherwise $G$ has a totally-nested 2-separation.
Let $O$ denote its orbit under the action of the automorphism group of~$G$, and pick $(A,B)\in O$ such that $A$ is minimal.
Pick any vertex $v\in A\sm B$.
By vertex-transitivity, there is $(C,D)\in O$ such that $v\in C\cap D$.
Since $O$ is nested, and since $v$ obstructs both of $(A,B)\le (C,D)$ and $(A,B)\le (D,C)$, we have $(C,D)\le (A,B)$ or $(D,C)\le (A,B)$.
Hence $C\se A$ or $D\se A$.
As $v$ lies in $C\cap D$ but not in~$B$, the inclusion $C\se A$ or $D\se A$ must be proper, contradicting the choice of~$(A,B)$.
\end{proof}

\section{A structural 2-Separation Theorem}

We say that a star $\sigma$ is \emph{$U$-principal} for a vertex set $U\se V(G)$ if $G\sm U$ has at least three components and
\[
    \sigma=\{\,s_K\colon K\text{ is a component of }G\sm U\,\}
\]
where
\[
    s_K:=\big(\,V(K)\cup U,\,V(G)\sm V(K)\,\big).
\]
The bag of a $U$-principal star $\sigma$ is equal to~$G[U]$, and the separators of the elements of~$\sigma$ are equal to~$U$.

\begin{thm}[Structural 2-Separation Theorem]\label{2sepThm}
Let $G$ be a 2-connected graph, and let $\sigma$ be any splitting star with torso~$X$ of the set~$N$ of all totally-nested 2-separations of~$G$.
If $|X|\le 2$, then $X$ is a~$K_2$ and $\sigma$ is $V(X)$-principal.
Otherwise $|X|\ge 3$, and exactly one of the following is true:
\begin{enumerate}
    \item $\sigma$ is interlaced by a 2-separation of~$G$ that is crossed like in a cycle, and $X$ is a cycle of length~$\ge 4$;
    \item $\sigma$ is not interlaced by a 2-separation of~$G$, and $X$ is 3-connected or a triangle.
\end{enumerate}
\end{thm}

We remark that the set~$N$ in \autoref{2sepThm} is canonical.

\begin{lem}\label{3compsEnsuresTotallyNested}
Let $G$ be a 2-connected graph and $U\se V(G)$ a set of two vertices such that $G\sm U$ has at least three components.
Then the $U$-principal star $\sigma$ of separations is a splitting star of the set of all totally-nested 2-separations of~$G$.
\end{lem}
\begin{proof}
Clearly, $\sigma$ is a star.
Its elements are totally nested by \autoref{SufficientForTotallyNested}.
Let $(C,D)$ be any 2-separation of~$G$ that interlaces~$\sigma$.
Then $(C,D)$ defines a bipartition $(\cC,\cD)$ of the set of components of~$G\sm U$, where $\cC$ and $\cD$ consist of the components contained in $G[C]$ and in~$G[D]$, respectively; also $C\cap D\se U$, and hence $C\cap D=U$.
Since $(C,D)$ is not in~$\sigma$, both $\cC$ and $\cD$ contain at least two components.
Hence $(C,D)$ is not totally-nested by \autoref{SufficientForTotallyNested}.
\end{proof}

\begin{lem}\label{TDCvsNested}
Let $G$ be a 2-connected graph.
Let $N$ be a nested set of half-connected 2-separations of~$G$.
Then there is no $(\omega+1)$-chain in~$N$, and for every $\omega$-chain $(A_0,B_0)<(A_1,B_1)<\ldots$ in~$N$ we have $\bigcap_{n\in\N}\,(B_n\sm A_n)=\emptyset$.
\end{lem}

\begin{proof}
Suppose for a contradiction that $(A_0,B_0)<(A_1,B_1)<\ldots<(A_\omega,B_\omega)$ is an $(\omega+1)$-chain in~$N$.
Since the elements of $N$ are half-connected, the separations in any 3-chain in~$N$ do not all have the same separators.
Therefore, we may assume without loss of generality that the $(A_i,B_i)$ have pairwise distinct separators.
Let $x\in A_0\sm B_0$ and $y\in B_\omega\sm A_\omega$.
Then $x$ and $y$ are separated by infinitely many pairwise distinct separators of size two.
Since $G$ is 2-connected, these separators are inclusionwise minimal $x$--$y$ separators in~$G$.
This contradicts a lemma of Halin~\cite[2.4]{HalinMinimal}, which states that any two vertices $u,v$ in a graph are separated by only finitely many inclusionwise minimal $u$--$v$ separators of size at most an arbitrarily prescribed~$k\in\N$.
The same argument also shows that $\bigcap_{n\in\N}\,(B_n\sm A_n)=\emptyset$ for every $\omega$-chain $(A_0,B_0)<(A_1,B_1)<\ldots$ in~$N$.
\end{proof}

\begin{cor}\label{crossing2sepsAreSurrounded}
Let $G$ be a connected graph, and let $N$ be a nested set of half-connected 2-separations of~$G$.
Then every separation $(A,B)$ of~$G$ with $(A,B)\notin N$ that is nested with every separation in~$N$ interlaces a unique splitting star of~$N$.
\end{cor}

\begin{proof}
The maximal elements of 
\[
    \{\,(C,D)\in N\colon (C,D)<(A,B)\text{ or }(C,D)<(B,A)\,\}
\]
form a star~$\sigma\se N$ that is interlaced by~$(A,B)$.
To show that $\sigma$ is a splitting star of~$N$, let $(X,Y)\in N$ be arbitrary.
Without loss of generality we have $(X,Y)<(A,B)$.
The set of all $(C,D)\in N$ with $(X,Y)\le (C,D)<(A,B)$ has a maximal element by \autoref{TDCvsNested}.
This element is also contained in~$\sigma$, so $\sigma$ is indeed a splitting star of~$N$.
By \autoref{unique_split}, $(A,B)$ interlaces no other splitting star of~$N$.
\end{proof}

\begin{proof}[Proof of \autoref{2sepThm}]
Let $\sigma$ be a splitting star of~$N$ with torso~$X$.
If $X$ has at most two vertices, then $X=K_2$, so we may assume that $X$ has at least three vertices.

We claim that $X$ is 2-connected, and assume for a contradiction that it is not.
Then $X$ has a separation $(A,B)$ of order at most one.
By the \nameref{separationLifting} (\ref{separationLifting}), $(A,B)$ lifts to a separation $(\hat A,\hat B)$ of~$G$ of order at most one, contradicting that $G$ is 2-connected.
So $X$ is 2-connected.

If $X$ has precisely three vertices, then $X=K_3$ as $X$ is 2-connected, so we may assume that $X$ has at least four vertices.

(i). Suppose that $\sigma$ is not interlaced by a 2-separation of~$G$.
If $X$ is not 3-connected, then $X$ has a 2-separation~$(A,B)$ (since $X$ has at least four vertices).
By the \nameref{separationLifting} (\ref{separationLifting}), $(A,B)$ lifts to a 2-separation $(\hat A,\hat B)$ of~$G$, which interlaces~$\sigma$, a contradiction.

(ii). Suppose that $\sigma$ is interlaced by a 2-separation of $G$.

\begin{sublem}\label{InterlacingInducesCrossed}
Every 2-separation $(A,B)$ of $G$ that interlaces~$\sigma$ induces a 2-separation $(A\cap V(X),B\cap V(X))$ of~$X$ that is crossed by a 2-separation of~$X$.
\end{sublem}
\begin{cproof}
Since $(A,B)$ interlaces~$\sigma$, it is not in~$N$.
Hence $(A,B)$ is crossed by a 2-separation~$(C,D)$ of~$G$.
We note that $(C,D)$ interlaces $\sigma$ as well, since otherwise $(C,D)$ would be nested with~$(A,B)$.
So the separators $A\cap B$ and $C\cap D$ are included in~$X$.
If $(A,B)$ and $(C,D)$ cross like in a cycle, then $(A\cap V(X),B\cap V(X))$ and $(C\cap V(X),D\cap V(X))$ are two crossing 2-separations of~$X$ as desired.
So it remains to show that $(A,B)$ and $(C,D)$ cannot cross with a four-flip.
Indeed, otherwise $A\cap B=C\cap D$, and $G\sm (A\cap B)$ has at least four components which define a splitting star of~$N$ as in \autoref{3compsEnsuresTotallyNested}.
As this splitting star is interlaced by~$(A,B)$, it must be equal to~$\sigma$ by \autoref{unique_split}.
But then $V(X)=A\cap B$ contradicts our assumption that $X$ has at least four vertices.
\end{cproof}\medskip

We recall that $X$ is 2-connected. 
By \autoref{InterlacingInducesCrossed} and our assumption, $X$ has a 2-separation.
Every 2-separation $(A,B)$ of~$X$ lifts to a 2-separation of~$G$ by the \nameref{separationLifting} (\ref{separationLifting}), which interlaces~$\sigma$ and through \autoref{InterlacingInducesCrossed} yields a 2-separation of~$X$ that crosses~$(A,B)$.
Hence $X$ is a cycle of length $\ge 4$ by the \nameref{Angry2Sep} (\ref{Angry2Sep}).
\end{proof}

\section[Proof of the 2-Separation Theorem]{Proof of the \nameref{2SeparatorTheorem} (\ref{2SeparatorTheorem})}\label{sec:20}

The \emph{bag} of a star $\sigma=\{\,(A_i,B_i)\colon i\in I\,\}$ of separations of a graph~$G$ is the graph obtained from $G$ by deleting $A_i\sm B_i$ for all~$i\in I$.
For example, if $(T,\Vcal)$ is a tree-decomposition of~$G$ and $t$ is a node of~$T$, then the separations induced by the edges of~$T$ incident with~$t$ and directed to~$t$ form a star~$\sigma_t$ of separations.
The bag of~$\sigma_t$ is equal to the bag~$G[V_t]$ associated with~$t$, where~$V_t\in\Vcal$.
The \emph{torso} of a star $\sigma$ of separations of~$G$ is obtained from the bag of~$\sigma$ by making $A\cap B$ complete for every $(A,B)\in \sigma$.
The torsos of the stars~$\sigma_t$ coincide with the torsos of the bags of~$(T,\Vcal)$.

Let $N$ be a nested set of separations of~$G$.
We define a candidate $\Tcal(N)=(T,\Vcal)$ for a tree-decomposition of~$G$, as follows.
The vertices of~$T$ are the splitting stars of~$N$.
We make two nodes $t_1\neq t_2$ of~$T$ adjacent if $(A,B)\in t_1$ and $(B,A)\in t_2$ for some separation $(A,B)$ of~$G$.
For each splitting star $t\in T$ we let $V_t$ be the vertex set of the bag of the star~$t$, and put $\Vcal=(V_t)_{t\in T}$.

\begin{lem}\label{TDCchar}
Let $G$ be a connected graph and $N$ a symmetric nested set of separations of~$G$.
Then the following two assertions are equivalent:
\begin{enumerate}
    \item $\Tcal(N)$ is a tree-decomposition of~$G$ whose set of induced separations is equal to~$N$;
    \item there is no $(\omega+1)$-chain in~$N$, and for every $\omega$-chain $(A_0,B_0)<(A_1,B_1)<\ldots$ in~$N$ we have $\bigcap_{n\in\N}\,(B_n\sm A_n)=\emptyset$.
\end{enumerate}
\end{lem}
\begin{proof}
(1)$\to$(2). Every $\omega$-chain in $N$ determines a ray in the decomposition tree $T$ of~$\Tcal(N)$. An $(\omega+1)$-chain in $N$ would determine a ray in $T$ followed by an edge, which is impossible. An $\omega$-chain in $N$ such as $(A_0,B_0)<(A_1,B_1)<\ldots$ determines a ray $R\se T$, and a vertex $v\in\bigcap_{n\in\N}(B_n\sm A_n)$ would have to live in the direction of~$R$ but in no bag in particular, contradicting~(T1).

The proof of \cite[Lemma~2.7]{InfiniteSplinters} shows (2)$\to$(1), even though the statement of \cite[Lemma~2.7]{InfiniteSplinters} says otherwise.
\end{proof}

\begin{lem}\label{2conAroundCycleTorso}
    Let $G$ be a 2-connected graph, and let $N$ denote the set of totally-nested 2-separations of~$G$.
    Let $\sigma$ be a splitting star of~$N$ such that the torso of~$\sigma$ is a cycle~$O$.
    Then, for every $(A,B)\in\sigma$, the side $G[A]$ is 2-connected.
\end{lem}
\begin{proof}
    If $G[A]$ is not 2-connected, then $G[A]$ has a cutvertex~$u$.
    Since $G$ is 2-connected, $u$ must be contained in $A\sm B$.
    Let $v$ be any vertex on~$O$ that is not in~$A$.
    Then $\{u,v\}$ is a 2-separator of~$G$, and every 2-separation of~$G$ with separator $\{u,v\}$ crosses $(A,B)$ like in a cycle, contradicting $(A,B)\in N$.
\end{proof}

\begin{lem}\label{2conAround3conTorso}
    Let $G$ be a 2-connected graph, and let $N$ denote the set of totally-nested 2-separations of~$G$.
    Let $\sigma$ be a splitting star of~$N$ such that the torso of~$\sigma$ is 3-connected.
    Then, for every $(A,B)\in\sigma$, the side $G[B]$ is 2-connected.
\end{lem}
\begin{proof}
If $G[B]$ has a cutvertex $v$, then $v$ separates the two vertices in $A\cap B$ as $G$ is 2-connected.
Hence $v$ together with the edge in the torso joining the two vertices in $A\cap B$ forms a mixed 2-separator of the torso, contradicting that the torso is 3-connected.
\end{proof}

\begin{fact}\label{DegreeThreeTotallyNested}
    Let $G$ be a 2-connected graph with a vertex $v$ of degree two such that $v$ lies in a 2-separator of~$G$.
    Then, for every totally-nested 2-separation $(A,B)$ of~$G$, we have $v\notin S(A,B)$.
\end{fact}
\begin{proof}
    If $v\in S(A, B)$, then $(N(v)+v,V(G)-v)$ is a 2-separation of~$G$ that crosses~$(A,B)$.
\end{proof}

\begin{proof}[Proof of \autoref{2SeparatorTheorem}]
Let $G$ be a 2-connected graph, and let $N$ denote the set of totally-nested 2-separations of~$G$.
By \autoref{TDCvsNested} and \autoref{TDCchar}, $N$~induces a tree-decomposition $\Tcal(N)=:(T,\Vcal)$ of~$G$.
Since $G$ is 2-connected and $(T,\Vcal)$ has adhesion two, it follows with Menger's theorem that all torsos of $(T,\Vcal)$ are minors of~$G$.
By \autoref{2sepThm}, the torsos of $(T,\Vcal)$ (which coincide with the torsos of~$N$) are 3-connected, cycles or~$K_2$'s.\medskip

\textbf{(1).} Let $(A,B)$ and $(C,D)$ be two mixed 2-separations of~$G$ that cross so that all four links have size one (and the centre is empty).
Let $F$ denote the set of all edges in $S(A,B)$ or $S(C,D)$.
Let $G'$ be the graph obtained from $G$ by subdividing all the edges in~$F$.
For each edge $e\in F$, we denote the subdividing vertex by~$v_e$.

If~$(U,W)$ is a mixed 2-separation of~$G$, then every edge $e\in F$ has an end that is not in~$S(U,W)$.
We obtain $U'$ from $U$ by adding all vertices $v_e$ for which $e$ has an end in~$U\sm W$.
Similarly, we obtain $W'$ from $W$ by adding all vertices $v_e$ for which $e$ has an end in~$W\sm U$.
Then $(U',W')$ is a 2-separation of~$G$: the separator $S(U',W')$ is obtained from $S(U,W)$ by replacing every edge $e$ in it that is in~$F$ with~$v_e$.

Let $N'$ denote the set of all totally-nested 2-separations of~$G'$.
As $(A',B')$ and $(C',D')$ cross like in a cycle, they are not members of~$N'$.
By \autoref{crossing2sepsAreSurrounded}, $(A',B')$ interlaces a unique splitting star~$\sigma'$ of~$N'$, and $(C',D')$ interlaces~$\sigma'$ as well (since otherwise $(C',D')$ would be nested with $(A',B')$).
By the \nameref{2sepThm} (\ref{2sepThm}), the torso of~$\sigma'$ is a cycle; let us denote this cycle by~$O'$.
Since $(A',B')$ and $(C',D')$ cross like in a cycle, $O'$ alternates between the separators $S(A',B')$ and $S(C',D')$.

\begin{sublem}\label{varphiClaim}
    The map $\varphi\colon (U,W)\mapsto (U',W')$ is a bijection between $N$ and~$N'$.
\end{sublem}
\begin{cproof}
Let $(U,W)\in N$.
By \autoref{SufficientForTotallyNested}, $(U,W)$ is externally 2-connected.
This is preserved by subdivision, so $(U',W')$ is externally 2-connected, and totally-nested by \autoref{SufficientForTotallyNested}.
Hence $(U',W')\in N'$.

Clearly, the map $\varphi$ is injective.
It remains to show that it is surjective, so let $(X,Y)\in N'$ be given.
By \autoref{DegreeThreeTotallyNested}, the separator of $(X,Y)$ contains no subdividing vertices, so $(U,W)$ where $U:=X\cap V(G)$ and $W:=Y\cap V(G)$ is a 2-separation of~$G$ which $\varphi$ sends to~$(X,Y)$.
As above, applying \autoref{SufficientForTotallyNested} twice gives $(U,W)\in N$.
\end{cproof}\medskip

By \autoref{varphiClaim}, $\sigma:=\varphi^{-1}(\sigma')$ is a splitting star of~$N$.
Since the separators of the elements of~$\sigma'$ contain no subdividing vertices, the torso~$O$ of~$\sigma$ is obtained from~$O'$ be replacing every subpath $x v_e y$ where $e=xy\in F\cap E(O')$ with the edge~$e$.
Thus, $O$ is a cycle.
As the cycle $O'$ alternates between $S(A',B')$ and $S(C',D')$, the cycle $O$ alternates between $S(A,B)$ and $S(C,D)$.
\medskip

\textbf{(2).} Assume that the torso $X$ associated with~$t\in T$ is 3-connected or a cycle. 
Let $xy$ be an edge of~$X$, and let $\{\,s_i t : i\in I\,\}=:F$ be the set of all edges of~$T$ incident with~$t$ that induce the adhesion set~$\{x,y\}$.
If $|I|\le 1$ we are done, so let us suppose for a contradiction that $|I|\ge 2$.
For each $i\in I$, let $T_i$ denote the component of $T-s_i t$ that contains~$s_i$.
Let $T_t$ denote the component of $T-F$ that contains~$t$.
Putting
\[
    A:=\bigcup_{i\in I}\bigcup_{r\in T_i}V_r\quad\text{and}\quad B:=\bigcup_{r\in T_t}V_r
\]
defines a separation $(A,B)$ of~$G$ with separator~$\{x,y\}$.

We claim that $(A,B)\in N$.
If $X$ is 3-connected, then $G[B]$ is 2-connected by \autoref{2conAround3conTorso}.
If $X$ is a cycle, then $G[A]$ is 2-connected by \autoref{2conAroundCycleTorso}.
Hence at least one of $G[A]$ or $G[B]$ is 2-connected.
Since $G$ is obtained from $X$ by replacing some edges with connected graphs containing their endvertices, and since the graph $X\sm A$ is connected, also the graph $G[B\sm A]$ is connected.
Hence $(A,B)\in N$ by \autoref{SufficientForTotallyNested}.

But then $(A,B)$ is an element of~$N$ that interlaces~$t$ (viewed as a splitting star of~$N$), which contradicts \autoref{splitVsInterlace}.
\end{proof}